\documentclass{article}

\usepackage{bbm}

\usepackage[all]{xy}\usepackage[latin1]{inputenc}  
\usepackage[dvips]{graphics,graphicx}
\usepackage{amsfonts,amssymb,amsmath,color,mathrsfs, amstext,bbm}
\usepackage{amsbsy, amsopn, amscd, amsxtra, amsthm,authblk, enumerate,mathbbol}
\usepackage{upref}
\usepackage{geometry}
\usepackage{bm}
\usepackage{dsfont}
\usepackage{todonotes}
\usepackage{mathtools}
\usepackage{appendix}
\usepackage{float}

\usepackage[all]{xy}\usepackage[latin1]{inputenc}        
\usepackage[dvips]{graphics,graphicx}
\usepackage{amsfonts,amssymb,amsmath,xcolor,mathrsfs, amstext}
\usepackage{amsbsy, amsopn, amscd, amsxtra, amsthm,authblk, enumerate}
\usepackage{upref}
\usepackage{geometry}
\usepackage{float}

\usepackage[colorlinks,
            linkcolor=blue,
            anchorcolor=green,
            citecolor=blue
            ]{hyperref}

\numberwithin{equation}{section}

\def\e{\varepsilon}
\def\epsilon{\varepsilon}
\def\eps{\varepsilon}

\newcommand{\ol}{\overline}
\newcommand{\wt}{\widetilde}
\newcommand{\wh}{\widehat}
\def\alb#1\ale{\begin{align*}#1\end{align*}}
\newcommand{\eqb}{\begin{equation}}
\newcommand{\eqe}{\end{equation}}

\DeclareMathOperator{\dist}{dist}

\newcommand{\inner}[2]{\langle {#1}, {#2}\rangle}
\newcommand{\bbC}{\mathbb{C}}
\newcommand{\bbD}{\mathbb{D}}
\newcommand{\bbE}{\mathbb{E}}
\newcommand{\bbH}{\mathbb{H}}

\newcommand{\bbN}{\mathbb{N}}
\newcommand{\VEL}{\mathrm{VEL}}

\newcommand{\bbR}{\mathbb{R}}
\newcommand{\bbP}{\mathbb{P}}

\newcommand{\bbZ}{\mathbb{Z}}

\newcommand{\cA}{\mathcal{A}}
\newcommand{\cD}{\mathcal{D}}
\newcommand{\cE}{\mathcal{E}}
\newcommand{\cF}{\mathcal{F}}
\newcommand{\cP}{\mathcal{P}}
\newcommand{\cG}{\mathcal{G}}

\newcommand{\cI}{\mathcal{I}}

\newcommand{\cT}{\mathcal{T}}
\newcommand{\cJ}{\mathcal{J}}

\newcommand{\cM}{\mathcal{M}}
\newcommand{\cS}{\mathcal{S}}
\newcommand{\cH}{\mathcal{H}}

\newcommand{\cV}{\mathcal{V}}

\newcommand{\cR}{\mathcal{R}}
\newcommand{\cX}{\mathcal{X}}

\newcommand{\fc}{\mathfrak{c}}

\newcommand{\fD}{\mathfrak{D}}

\newcommand{\rmT}{\mathrm{T}}

\newcommand{\fre}{\mathfrak{e}}

\newcommand{\diam}{\mathrm{diam}}
\newcommand{\area}{\mathrm{area}}

\newtheorem{theorem}{Theorem}[section]
\newtheorem{lemma}[theorem]{Lemma}
\newtheorem{proposition}[theorem]{Proposition}
\newtheorem*{proposition*}{Proposition}
\newtheorem{corollary}[theorem]{Corollary}
\newtheorem*{corollary*}{Corollary}
\newtheorem{remark}[theorem]{Remark}

\theoremstyle{definition}
\newtheorem{definition}[theorem]{Definition}
\newtheorem*{definitions*}{Definitions}

\newtheorem{example}[theorem]{\bf Example}
\newtheorem{problem}[theorem]{\bf Problem}
\newtheorem*{example*}{\bf Example}

\numberwithin{equation}{section}

\title{Circle packing and Riemann uniformization of random planar maps in an ergodic scale-free environment}
\author{Nina Holden\thanks{Courant Institute, New York University, nina.holden@nyu.edu} \qquad   Pu Yu\thanks{Courant Institute, New York University, py628@nyu.edu}}

\begin{document}

\maketitle

\begin{abstract} 
We prove that embedded infinite planar maps in ergodic scale-free environments are close to their circle packing and Riemann uniformization embeddings on a large scale, as long as suitable  moment and connectivity  conditions are satisfied. In subsequent work we use these results to prove convergence of random planar maps to Liouville quantum gravity under circle packing and Riemann uniformization embeddings. 
%
\end{abstract}

\tableofcontents

\section{Introduction}

A planar map is, roughly speaking, a planar graph where, for each vertex, the cyclic ordering of its edges in a planar embedding is given. Planar maps are studied in many different branches of both mathematics and physics, including combinatorics, probability, geometry, string theory, and conformal field theory.

Given a planar map $\cM$, how  {is it natural to} 
draw it in the plane? In this paper we consider two well-studied ways to embed $\cM$ which give a discrete conformal structure to $\cM$: One is the Koebe-Andreev-Thurston \emph{circle packings}, where the vertices of $\cM$ are represented by a collection of closed disks with disjoint interiors, such that two disks are tangent if and only if their corresponding vertices on $\cM$ are adjacent. The second way, which is called the \emph{Riemann uniformization embedding}, is to view each face of $\cM$ as a regular polygon and glue them isometrically along their edges. This defines a Riemann surface $M$, and one   uses the Riemann uniformization theorem to draw the vertices and edges of $\cM$ via a conformal map $\varphi:M\to\bbC$.

In this paper, we work with  random planar maps with  {whole-}plane topology, i.e., one-ended planar maps without boundary, under the two aforementioned  conformal embeddings. Circle packings for random triangulations have been studied in e.g.~\cite{BenjaminiSchramm01,GNrecurrence,ABGN16,GDN19VEL,OV23,bougwynne2024random}, and the Riemann uniformization embeddings for random planar maps have been considered in e.g.~\cite{GRRiemann,Curienglimpse}. Our main result is, for a family of random planar maps with  {whole-}plane topology in ergodic scale-free environments, they are close to their circle packing and Riemann uniformization embeddings on a large scale.  Our result holds under a moment condition on the degrees of the triangulation and the regularity of the environment, along with a condition guaranteeing that the  {embedded} triangulation is sufficiently well-connected. Convergence of circle packing for random  {finite} Delaunay triangulations was shown in~\cite[Theorem 1.3]{OV23}.   Example~\ref{example:Voronoi}, which is an immediate application of our Theorem~\ref{thm:CellSystemPack}, is a version of ~\cite[Theorem 1.3]{OV23} in the plane.  To the best of our knowledge, there are no other earlier works in literature proving convergence of random planar maps under circle packing and Riemann uniformization embeddings. 

Ergodic scale-free environments were originally considered by Gwynne, Miller and Sheffield~\cite{GMSinvariance} where they proved the convergence of random walk to  {planar} Brownian motion. {In subsequent work \cite{GMS21tutte} they used this to prove that a particular family of random planar maps known as mated-CRT maps converge under the Tutte embedding to continuum random surfaces arising in the study of Liouville quantum gravity.}  Liouville quantum gravity (LQG) was introduced by Polyakov~\cite{polyakov1981quantum} and is conjectured and proven to describe the scaling limit of several models of random planar maps{; see e.g.\ the book \cite{BP24GFFnotes} for an introduction to the mathematical theory of LQG}. Other  {recent} works which establish convergence results of random planar maps to LQG under conformal embedding include \cite{HS19,GMS20Voronoi,bertacco2025scaling,Tutte}. As an application of this paper, in the subsequent work~\cite{MatedCRT,HY26a}, we further prove the convergence of the mated-CRT map and tree-weighted planar maps to LQG under circle packing and Riemann uniformization embeddings. We expect that with the main results of this paper and the techniques from~\cite{Tutte,MatedCRT,HY26a}, one can also prove convergence of other planar map models to LQG under circle packing and Riemann uniformization embeddings, including UIPT and bipolar-oriented maps.

The rest of this section  is organized as follows. In Section~\ref{subsec:intro-1} we introduce some background on circle packings and the Riemann uniformization embedding. In Section~\ref{subsec:intro-2} we introduce  {ergodic scale-free cell configurations}.  
In Section~\ref{subsec:intro-3} we state our main results, Theorems~\ref{thm:CellSystemPack} and~\ref{thm:invariance-uniformization}. In Section~\ref{subsec:intro-4} we discuss some applications and future directions related to this work. Finally in Section~\ref{subsec:intro-5}, we give a detailed outline of our proof of Theorems~\ref{thm:CellSystemPack} and~\ref{thm:invariance-uniformization} in the rest of the paper.

\subsection{Discrete conformal embeddings}\label{subsec:intro-1}
Let $\cG$ be a {locally finite} planar (multi)graph with vertex set $\cV\cG$ and edge set $\cE\cG$. Multiple edges connecting a pair of vertices are allowed, while loops (i.e., edges starting and ending at the same vertex) are not allowed. We say that a collection of distinct points $\{z_v:v\in\cV\cG\}\subseteq \wh\bbC:=\bbC\cup\{\infty\}$ and a collection of simple (undirected) curves $\{\gamma_e:e\in\cE\cG\}$ on $\wh\bbC$ form a \emph{planar embedding}   of $\cG$, if the endpoints of curves $\gamma_e$ for $e$ connecting vertices $v,v'$ are at $z_v$ and $z_{v'}$, and different curves $\gamma_e$ and $\gamma_{e'}$ may only intersect each other at starting and ending points.  A planar map $\cM$ is an embedding of a planar graph $\cG$ modulo orientation-preserving homeomorphisms of $\wh\bbC$. 
The face set $\cF\cM$ is the set of the connected components of the complement of the planar map, and we say that $\cM$ has  {whole-}plane topology if all the faces are bounded in $\bbC$, and the union of faces in $\cF\cM$ along with their closures is equal to $\bbC$.
A triangular face is a face with three edges on its boundary.  
A (planar) triangulation is the union of the vertices and edges of a planar map $\cM$, as well as (a subset of) the triangular faces of $\cM$. We will work with \emph{infinite plane triangulations}, where the underlying planar map is one-ended and every face has degree three and belongs to the triangulation, and \emph{disk triangulations}, where the triangulation is finite and topologically homeomorphic to the unit disk. See Section~\ref{pre:triangulation} for more details.

A planar map $\cM$ can be embedded in $\wh{\bbC}$ (as described in the previous paragraph) in a number of different ways. In this paper, we focus on two natural  embeddings: the \emph{circle packing} and the \emph{Riemann uniformization embedding}. These two embeddings are canonical examples of \emph{discrete conformal embeddings}.  {This means in particular that} embeddings of the same planar map in different domains induce (or is believed to induce) a map that approximates the conformal map between the two domains. For circle packing, this was conjectured by Thurston and solved in~\cite{RodinSullivan87} in the case of the hexagonal lattice, 
in~\cite{kS90,kS96} for circle packings with uniformly bounded radii ratio for neighboring disks via probabilistic methods, 
and by~\cite{HeSchramm} in full generality. For Riemann uniformization embedding, this follows from the definition.
 
Let $\cM$ be a simple planar map with vertex set $\cV\cM$. We say that a collection $\cP = (D_v:v\in\cV\cM)$ of closed disks with disjoint interiors is a \textit{circle packing} for $\cM$, if two disks $D_v$ and $D_{v'}$ are tangent if and only if the vertices $v$ and $v'$ are connected by an edge. This gives an embedding of the planar map, with the vertices embedded at the centers of the disks and the edges are straight line segments connecting the centers of tangent disks. By the celebrated Koebe-Andreev-Thurston circle packing theorem,  any finite simple planar map has a circle packing, and for disk triangulations, there exists a unique circle packing in the unit disk $\bbD$ up to M\"{o}bius transformations. For infinite plane triangulations, the existence of a circle packing is a straightforward consequence of the tightness from the Ring Lemma~\cite{RodinSullivan87} (Lemma~\ref{lem:Ring-Lemma}), and the uniqueness up to M\"{o}bius transformations was proved by Schramm~\cite{Schramm91}. In this case, there is a dichotomy: an infinite plane triangulation $\cT$ is \emph{circle packing parabolic} (resp.\ \emph{circle packing hyperbolic}), if there exists a circle packing $\cP$ for $\cT$, such that the union of all the faces and edges (a.k.a.\ the \emph{carrier} of $\cP$) is equal to the whole plane (resp.\ unit disk). In~\cite{HeSchramm-hy-pa}, it was further shown that, if the simple random walk on $\cT$ is recurrent, then $\cT$ is circle packing parabolic, while if $\cT$ has uniformly bounded degree and the simple random walk on $\cT$ is transient, then $\cT$ is circle packing hyperbolic. There are numerous applications of circle packings. For instance, circle packing is closely related to He and Schramm's work on the Koebe conjecture~\cite{HSKoebe}, and has played an important role on the proof of recurrence of random walks on random graphs~\cite{BenjaminiSchramm01,GNrecurrence}. See also~\cite{circlepackingbook,Nachmiaslec} for a comprehensive introduction on circle packing and related topics.

A second natural way to embed random planar maps is through the Riemann uniformization theorem. Roughly speaking, for a planar map $\cM$, one associates each face of degree $p$ with a regular $p$-gon of unit side length and glues the polygons isometrically along their edges to get a Riemann surface $M(\cM)$; see Section~\ref{pre:triangulation} for details. If $\cM$ has whole-plane topology, there are again two possible cases, where the Riemann surface $M(\cM)$ is \emph{parabolic} (resp.\ \emph{hyperbolic}) if it is conformally equivalent to the whole plane (resp.\ the unit disk). As commented in~\cite{GRRiemann,AHNR18unimodular}, similar to the circle packing case, this is equivalent to the recurrence/transience of the Brownian motion on the surface $M(\cM)$. See e.g.~\cite{GRRiemann,Curienglimpse,AHNR18unimodular} for some works in this direction.

There exist a number of discrete conformal embeddings beyond the two discussed above, including the Tutte embedding, the Smith embedding \cite{bsst-smith}, the Cardy-Smirnov embedding \cite{HS19}, and the t-embedding/Coulomb gauge/s-embedding \cite{chelkak-icm,clr-t-emb,klrr-circle-patterns}. 

\subsection{Ergodic scale-free cell configurations}\label{subsec:intro-2}

In~\cite{GMSinvariance}, Gwynne, Miller and Sheffield introduced a certain type of ergodic scale-free cell configurations, and proved the convergence of random walks on such cell configurations to Brownian motion in the quenched sense, which is an important step towards proving the convergence of random planar maps to LQG under Tutte embedding. In this subsection, we {introduce} cell configurations~\cite{GMSinvariance}, as well as a number of conditions we need in our main results.

\begin{definition}\label{def:cell-config}
    A \textit{cell configuration} on $\bbC$ consists of the following objects.
    \begin{enumerate}
        \item A locally finite collection $\cH$ of  compact connected subsets of $\bbC$ (``cells") with non-empty interiors whose union is all of $\bbC$ and such that the intersection of any two elements of $\cH$ has zero Lebesgue measure.
        \item A planar map $\cM$, which we call the \emph{associated map} of $\cH$, with vertices $\{v_H:H\in\cH\}$. We require that if $v_H$ and $v_{H'}$ are connected by an edge of $\cM$ then $H\cap H'\neq\emptyset$. 
    \item A function $\fc = \fc_\cH$ (``conductance") on the edges of  $\cM$ with values in $(0,\infty)$. 
    \end{enumerate}
\end{definition}

We say $H\sim H'$ if $H$ and $H'$ are connected by  {at least one} edge in $\cM$. Note that we use a different definition of a cell configuration than in~\cite[Definition 1.1]{GMSinvariance}, where a cell configuration is defined to be a tuple $(\cH,\sim,\fc)$. We need to also include $\cM$ in the definition since our goal is to prove closeness between $\cH$ and conformal embeddings of the map $\cM$ and, furthermore, will work with settings where $\cM$ has multiple edges.  
We will typically slightly abuse notation by making  the associated map $\cM$ and the function $\fc$ implicit, so we write $\cH$ instead of $(\cH,\cM,\fc )$. 

For $A\subseteq\bbC$, let $\cH(A) = \{H\in\cH:H\cap A\neq\emptyset\}$, and we view $\cH(A)$ as an edge-weighted planar map with edge set $\cE\cH(A)$ consisting of all of the edges in $\cE\cH$ joining elements in $\cH(A)$ and conductances given by the restriction of $\fc$. We also define a metric on the space of cell configurations by
 \begin{equation}\label{eq:metric-cc}
     \begin{split}
         \mathbb{d}^{\rm CC}(\cH,\cH') := \int_0^\infty e^{-r}\wedge \inf\Big\{\sup_{z\in\bbC} |z-f_r(z)|+\sup_{e\in \cE\cH(B(0;r))}\big|\fc(e)-\fc'(f_r(e)) \big|   \Big\}\,dr,
     \end{split}
 \end{equation}
 where the infimum is over all homeomorphisms $f_r:\bbC\to\bbC$ such that $f_r$ takes each cell in $\cH(B(0;r))$ to $\cH'(B(0;r))$   and  preserves the associated map when restricted to cells intersecting $B(0;r)$, and $f_r^{-1}$ does the same with $\cH$ and $\cH'$ reversed. Here $f_r(\fre)$ is understood as the image of the edge $e$ under map isomorphisms. We will typically be working with a random cell configuration, i.e., a random
variable taking values in the space of cell configurations equipped with the Borel $\sigma$-algebra generated
by the above metric.

\begin{definition}\label{def:translation-modulo-scaling}
    A random cell configuration $\cH$  is \textit{translation  invariant modulo scaling} if it satisfies the following condition. There is a (possibly random and $\cH$-dependent) increasing sequence of open sets $U_j\subseteq\bbC$,  each of which is either a square or a disk, whose union is all of $\bbC$, such that the following is true. Conditional on $\cH$ and $U_j$, let $z_j$ for $j\in\bbN$ be sampled uniformly from Lebesgue measure on $U_j$. Then there exist random numbers $C_j>0$ (possibly depending on $\cH$ and $z_j$) such that $C_j(\cH-z_j)\to\cH$ in law with respect to the metric~\eqref{eq:metric-cc}. 
\end{definition}

In~\cite{GMSinvariance}, the authors introduced four more conditions that are equivalent to the translation invariant modulo scaling condition in Definition~\ref{def:translation-modulo-scaling}; see~\cite[Definition 1.2 and Lemma 2.3]{GMSinvariance} 
along with Lemma~\ref{lem:dyadic-resample} below. 

\begin{definition}\label{def:ergodic-modulo-scaling}
    Let $\cH$ be a random cell configuration that is translation invariant modulo scaling in the sense of Definition~\ref{def:translation-modulo-scaling}. We  say $\cH$ is \textit{ergodic modulo scaling}, if every real-valued measurable function $F=F(\cH)$ which is invariant under translation and scaling, i.e., $F(C(\cH-z)) = F(\cH)$ for each $z\in\bbC$ and $C>0$, is almost surely constant.
\end{definition}

We emphasize that, in the definition of cell configurations, 
it is allowed that two cells $H$ and $H'$ have nonempty intersections while $H$ and $H'$ are not connected by any edge in $\cE\cH$. We also allow that the cells in $\cH$ are not simply connected and may cross each other. See Figure~\ref{fig:cellconfigdef} for an illustration.  
On the other hand, we will need the following property on the connectedness of $\cH$.

\begin{definition}\label{def:connectedness}
    A random cell configuration $\cH$  is \textit{connected along macroscopic lines}, if almost surely,  {for each given $\e>0$, there exists $R>0$ such that for $r>R$ and each horizontal or vertical line segment $\ell\subseteq B(0;r)$ with length at least $\e r$}, the subgraph of $\cH$ induced by the set of cells which intersect $\ell$ is connected. 
\end{definition}

In~\cite{GMSinvariance}, it is shown that for cell configurations satisfying the constraints in Definition~\ref{def:ergodic-modulo-scaling} and Definition~\ref{def:connectedness}, {along with a moment condition}, random walks on $\cH$ converge to Brownian motion (see Theorem~\ref{thm:GMSinvariance}). This will be applied to the proof of the closeness of $\cH$ and its circle packing, where in Section~\ref{subsec:Covergence-RW-cell-Dubejko} we establish a variant of this result for a certain conductance from circle packing.

\subsection{Main results}\label{subsec:intro-3}

In this subsection, we state our main results on the closeness between cell configurations and their discrete conformal embeddings. We start with the case where the associated map $\cM$ of $\cH$ is an infinite plane triangulation. 

\begin{figure}
    \centering
    \includegraphics[scale=0.9]{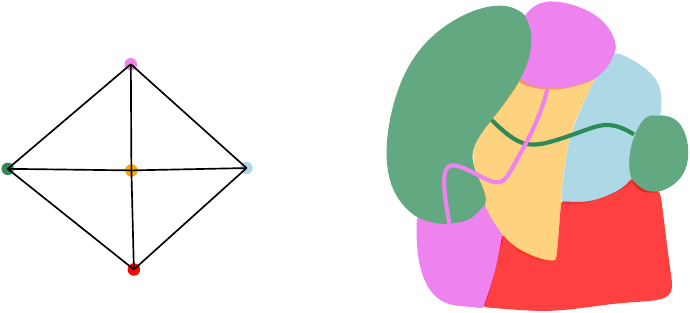}
    \caption{An illustration of five cells in a cell configuration $\cH$. The cells in $\cH$ may not be simply connected and could cross each other. We allow the green and blue cells, and the red and purple cells, to touch each other while not being connected by an edge in $\cE\cH$.}
    \label{fig:cellconfigdef}
\end{figure}


\begin{definition}\label{def:rotation-inv}
      Let $\cH$ be a cell configuration as in Definition~\ref{def:cell-config}. We say $\cH$ is \emph{rotation invariant in law for some $\theta_0\in (0,2\pi)\backslash\{\pi\}$}, if the rotated cell configuration $e^{i\theta_0}\cH$ has the same law as $\cH$.
\end{definition}

The reason we exclude the value $\pi$ here is that, for any 2D Brownian motion $(X,Y)$, the covariance matrix for $(X,Y)$ and $(-X,-Y)$ is the same, and in this case there is no further information. On the other hand, one can check that for $\theta_0\in(0,2\pi)\backslash\{\pi\}$, if a 2D Brownian motion and its rotation over angle $\theta_0$ have the same covariance matrix, then this covariance matrix must be a scalar times the identity.

We write $H_0$ for  the (a.s.\ unique) cell in $\cH$ that contains 0, and $\deg(H_0)$ for the degree of $H_0$ in the associated map.

\begin{theorem}\label{thm:CellSystemPack}
    Let $\cH$ be a random cell configuration such that its associated map is a simple infinite plane triangulation. Suppose $\cH$   satisfies all the conditions in Definitions~\ref{def:ergodic-modulo-scaling} and \ref{def:connectedness}. Further assume $\cH$ satisfies the moment bound
    \begin{equation}\label{eq:thm-inv-circle}
        \bbE\Big[\frac{\diam(H_0)^2}{\area(H_0)}\deg(H_0)^4\Big]<\infty.
    \end{equation}
    Then $\cH$ is a.s.\  circle packing parabolic, and there exists a circle packing $\cP$ for $\cH$ and a deterministic $2\times2$ matrix $\mathbf{A}$ with $\det(\mathbf{A})=\pm 1$ such that the following is true. Almost surely, for any $\e>0$, there exists $R_0>0$ such that for every $r>R_0$, the maximal diameter of disks in $\cP$ intersecting $B(0;r)$ is less than $\e r$. Furthermore, for $H\in\cH$, letting $c(H) = \frac{1}{\area(H)}\int_H z\,dz$ be the Euclidean center of $H$, and letting $o(H)$ be the center of the disk for $H$ in $\cP$, almost surely, 
    \begin{equation}\label{eq:thm:CellsystemPack}
        \lim_{r\to\infty}\frac{1}{r}\max_{H\in \cH(B(0;r))}|\mathbf{A}c(H)-o(H)| = 0.
    \end{equation}
    Finally, if $\cH$ satisfies the rotation invariance in Definition~\ref{def:rotation-inv}, then $\mathbf{A}$ can be taken to be the identity matrix or the reflection matrix $\begin{pmatrix}
        1  & 0\\0&-1
    \end{pmatrix}$.
\end{theorem}

\begin{theorem}\label{thm:invariance-uniformization}
     Let $\cH$ be a random cell configuration such that its associated map is an infinite plane triangulation. Suppose  $\cH$ satisfies the conditions in  Definitions~\ref{def:ergodic-modulo-scaling} and \ref{def:connectedness}. Further assume $\cH$ satisfies the moment bound~\eqref{eq:thm-inv-circle}.
    Then the Riemann surface $M(\cH)$ formed by the associated map of $\cH$ is a.s.\   parabolic, and there exists a deterministic $2\times2$ matrix $\mathbf{A}$ with $\det(\mathbf{A})=\pm1$ and a conformal map $\varphi_0:M(\cH)\to\bbC$ for $\cH$ such that the following is true. For $H\in\cH$, let $c(H) = \frac{1}{\area(H)}\int_H z\,dz$ be the Euclidean center of $H$, and $v_H$ be the   vertex for $H$ in $M(\cH)$. Then {almost surely}, 
    \begin{equation}\label{eq:thm:invariance-uniform}
        \lim_{r\to\infty}\frac{1}{r}\max_{H\in \cH(B(0;r))}|\mathbf{A}c(H)-\varphi_0(v_H)| = 0.
    \end{equation}
    Furthermore,  for each $e\in \cE\cH$, we write $\mathfrak{e}_{e}$ for the edge for $e$ on  $M(\cH)$,  
     \begin{equation}\label{eq:thm:invariance-uniform-1}
        \lim_{r\to\infty}\frac{1}{r}\max_{e\in \cE\cH(B(0;r))}\diam(\varphi_0(\mathfrak{e}_{e})) = 0.
    \end{equation}
    Finally, if $\cH$ further satisfies the rotation invariance in Definition~\ref{def:rotation-inv}, then $\mathbf{A}$ can be taken to be the identity matrix or the reflection matrix $\begin{pmatrix}
        1  & 0\\0&-1
    \end{pmatrix}$.
\end{theorem}

For infinite plane triangulations, the circle packing is unique. For general planar maps, the circle packing is not necessarily unique; however, one can get uniqueness if one imposes further conditions; see e.g.~\cite{BS93, Sch92, Thur97}. There are several kinds of conditions that one can impose. One of the simplest ways to define a unique circle packing for general simple planar maps is as follows. Let $\cM$  be a 2-connected simple planar map with whole-plane topology and face set $\cF\cM$. For each face $f\in\cM$, we add a vertex $v_f$ and add edges to connect $v_f$ to the vertices incident to $f$. This turns $\cM$ into a simple infinite plane triangulation $\cT_\cM$. We can consider the circle packing for $\cT_\cM$ and remove the disks for all the vertices $\{v_f:f\in\cF\cM\}$. This defines a circle packing for $\cM$ and it is the unique, up to M\"{o}bius transformation, circle packing constructed in this way.

Let $\cH$ be a cell configuration such that the associated map $\cM$ is 2-connected and has  {whole-}plane topology. Let $\cT_\cM$ be the triangulation generated from $\cM$ as above. We write $v_H\sim f$ if the vertex $v_H$ for $H$ in the associated map $\cM$ is incident to the face $f\in\cF\cM$. 
  We write  $\cH^{\mathrm{face}}_0$ for the set of cells $H\in\cH$ such that $v_H$ and $v_{H_0}$ are incident to the same face of $\cM$. We write $d_0^f$ for the maximum degree of faces incident to $v_{H_0}$.
  \begin{theorem}\label{thm:CellSystemPack-extension}
     Let $\cH$ be a cell configuration with the associated map $\cM$ being simple, 2-connected and having {whole-}plane topology. Assume $\cH$ satisfies all the conditions in Definitions~\ref{def:ergodic-modulo-scaling} and \ref{def:connectedness}. 
     Further assume 
      \begin{equation}\label{eq:thm:CellSystemPack-extension}
       \bbE\Big[ \max_{H\in \cH^{\mathrm{face}}_0}\Big(\frac{\diam(H)^2}{\area(H)}\big(d_0^f+\deg(H)\big)^6\Big)\Big]<\infty. 
   \end{equation}
   Then the conclusions in Theorem~\ref{thm:CellSystemPack} hold for the circle packing of $\cM$ defined as above.
 \end{theorem}
 In the setting of the Riemann uniformization, let $\cM$ be a planar map with  {whole-}plane topology and face set $\cF\cM$. For each face $f$ enclosed by a single loop, we simply remove this loop. For each face $f\in\cF\cM$ with 2 edges,  we contract the two edges to turn it into a single edge, such that each face of the resulting map has at least 3 edges. Let $M(\cM)$ be the Riemann surface obtained by  associating each remaining face  {of degree $p\geq 3$} with  {regular $p$-gon of unit side length} and glue them isometrically along their edges.
 \begin{theorem}\label{thm:invariance-uniformization-extension}
     Let $\cH$ be a cell configuration with the associated map $\cM$ having  {whole-}plane topology.
     Assume $\cH$   satisfies all the conditions in Definitions~\ref{def:ergodic-modulo-scaling} and \ref{def:connectedness}. 
     Further assume that~\eqref{eq:thm:CellSystemPack-extension} holds.
      Then the conclusions in Theorem~\ref{thm:invariance-uniformization} hold for the Riemann surface  $M(\cM)$ defined as above. Furthermore, for each $f\in \cF\cH$, we write $\mathfrak{f}_{f}$ for the face for $f$ on  $M(\cM)$,  then
     \begin{equation}\label{eq:thm:invariance-uniform-10}
        \lim_{r\to\infty}\frac{1}{r}\max_{f\in \cF\cH(B(0;r))}\diam(\varphi_0(\mathfrak{f}_{f})) = 0,
    \end{equation}
    where $\cF\cH(B(0;r))$ is the set of faces of the associated map for $\cH$ which are incident to cells in $\cH(B(0;r))$. Finally, for each $\e>0$, almost surely for sufficiently large $r>0$, if the embedded face $\varphi_0(\mathfrak{f}_{f})$ intersects $B(0;r)$, then $\diam(\varphi_0(\mathfrak{f}_{f}))<\e r$.
 \end{theorem} 
 As discussed in~\cite[Section 6]{AHNR18unimodular}, there are {several} other natural ways to define Riemann surfaces associated with planar maps; for instance, one can assign each face of $\cM$ of degree $p$ a disk of circumference $p$,  with boundary split into $p$ arcs of length one corresponding to the edges, and glue adjacent faces according to arc length along their shared edges. We expect that similar conclusions still hold by minor modifications of our proof of Theorem~\ref{thm:invariance-uniformization-extension}.


We comment that the convergence results of random planar maps under discrete conformal embeddings in~\cite{GMS20Voronoi,GMS21tutte,HS19,OV23,bertacco2025scaling} are all stated for finite planar maps, whereas Theorems~\ref{thm:CellSystemPack}-\ref{thm:invariance-uniformization-extension} are stated for infinite plane triangulations. In our subsequent works~\cite{MatedCRT,HY26a}, we apply Theorems~\ref{thm:CellSystemPack}-\ref{thm:invariance-uniformization-extension} together with the comparison of circle packings in different domains developed in~\cite[Section 1.3.2]{MatedCRT} to work on the convergence of random planar maps under circle packing and Riemann uniformization embedding in \emph{both} finite volume and infinite volume settings. More precisely, we will use a version of Theorems~\ref{thm:CellSystemPack}-\ref{thm:invariance-uniformization-extension}  (Theorem~\ref{thm:large}) with a weaker version of the line connectivity in Definition~\ref{def:connectedness}, which will be proved in Section~\ref{subsec:large}.

\subsection{Applications}\label{subsec:intro-4}
In this section, we discuss some applications  for this work. We start with the following two simple examples, which provide applications of Theorems~\ref{thm:CellSystemPack} and~\ref{thm:invariance-uniformization}.
 We believe the results in these examples are new, except that a finite-volume version of Example \ref{example:Voronoi} for circle packings was proven in \cite{OV23}. 

\begin{example}[Voronoi tessellation of a homogeneous Poisson point process]\label{example:Voronoi}
  Let $\lambda>0$, and consider a Poisson point process $\Lambda$ on $\bbC$ with intensity measure equal to $\lambda$ times the Lebesgue measure on $\bbC$. Let $\cH$ be the cell configuration induced by the Voronoi tessellation associated with $\Lambda$, where the cells are the Voronoi cells, and two cells are connected by an edge in $\cE\cH$ if and only if they have nonempty intersection. Almost surely, there do not exist four points in $\Lambda$ that are jointly on a   circle, so the associated map of $\cH$ is a.s.\ a simple infinite plane triangulation. Since $\Lambda$ is translation invariant in law, one can verify that the conditions in Theorem~\ref{thm:CellSystemPack} and Theorem~\ref{thm:invariance-uniformization} hold. To see that the moment bound~\eqref{eq:thm-inv-circle} holds, 
  one can first prove that $\bbP[\area(H_0)\leq s]\leq cs^{4/3}$ for some constant $c>0$. On the event $$E:=\{\Lambda\cap B(M^{1/4}(m+ni);2M^{1/4})\neq\emptyset, \text{ for all } m,n\in [-10,10]\cap\bbZ \}\cap \{\#\Lambda\cap B(0;100M^{1/4})\leq M^{3/4}\},$$
  we have $\diam(H_0)\leq M$ and $\deg(H_0)\leq M$. On the other hand, one can check that for some constants $c_1,c_2,c_3>0$, we have   $\bbP[E]\geq 1-c_1e^{-c_2M^{c_3}}$, and thus $$\bbE\big[\frac{1}{\area(H_0)^p}\big]<\infty, \ \ \bbE\big[\diam(H_0)^q\big]<\infty, \ \ \bbE\big[\deg(H_0)^q\big]<\infty,  \ \ \text{ for $1\leq p<4/3$ and every $q>0$},$$
  which further implies~\eqref{eq:thm-inv-circle}. Therefore by Theorem~\ref{thm:CellSystemPack} and Theorem~\ref{thm:invariance-uniformization}, the circle packing of $\cH$ and the Riemann uniformization embedding of $\cH$ must be close to $\cH$ on a large scale.
  \end{example}


  \begin{example}[Percolation clusters]
      Consider a hexagonal lattice. Uniformly pick some point $w$ in some hexagon and shift the lattice by $w$. Consider an independent Bernoulli percolation  with probability $p$ on the edges of the dual lattice. Let $\cH_0$ be the cell configuration where the cells are given by the unions of the (closed) hexagons in each cluster. Assume $p$ is subcritical such that there is a.s.\ no infinite cluster. For each cell $H$ in $\cH_0$, we join $H$ with the hexagons that are separated from infinity by $H$ and remove these separated clusters, and let $\cH$ be the cell configuration we obtain. Here for each connected component of $H_1\cap H_2$ with $H_1\neq H_2$ and $H_1\cap H_2\neq\emptyset$, we connect $H_1$ and $H_2$ by an edge crossing this component of $H_1\cap H_2 $,  defining an associated planar map   with possibly multiple edges. Then the associated map of $\cH$ is an infinite plane triangulation.  Using the exponential decay over $n$ for the probability that 0 is connected to hexagons intersecting $\partial B(0;n)$, translation invariance, and a union bound, one can check that the moment bound~\eqref{eq:thm-inv-circle} holds, and $\cH$  satisfies all the conditions in Theorem~\ref{thm:invariance-uniformization}, as $\cH$ is invariant in law under a rotation of $2\pi/3$. Therefore by Theorem~\ref{thm:invariance-uniformization}, the Riemann uniformization embedding of   $\cH$ is close to $\cH$ on a large scale. If we further collapse the double edges, in the sense that for any two different cells $H_1,H_2$ forming a double edge such that they are not separated from infinity by the union of any two other cells, we independently merge the union of the cells separated from infinity by $H_1\cup H_2$ into $H_1$ or $H_2$ with equal probability and remove them from $\cH$.
      Then we get a cell configuration $\wh\cH$ whose associated map is a simple infinite plane triangulation. In this case, by Theorem~\ref{thm:CellSystemPack} and Theorem~\ref{thm:invariance-uniformization}, $\wh\cH$ is close to its circle packing and Riemann uniformization embedding on a large scale. 
  \end{example}

\begin{example}[Random diagonal triangulation of $\bbZ^2$]
    Let $\cH = \{(m,n)+[-\frac{1}{2},\frac{1}{2}]^2:(m,n)\in\bbZ^2\}$ where the initial associated map is the square lattice $\bbZ^2$. For each face $f$ of $\bbZ^2$, let $\sigma_f\in\{-1,1\}$ be random. If $\sigma_f = 1$, then we further connect the northeast/southwest corners of $f$ in the  associated map, while if $\sigma_f = -1$,  we connect the southeast/northwest corners. This is legitimate since in our definition of the cell configuration we allow that $H_1$ and $H_2$ intersect at a single point if $H_1\sim H_2$. If we assume that the field $\{\sigma_f:f\in\bbZ^2+(\frac{1}{2},\frac{1}{2})\}$ is stationary and ergodic, e.g., $\sigma_f$ are i.i.d.\ or taken from critical Ising spin, then by Theorem~\ref{thm:CellSystemPack} and Theorem~\ref{thm:invariance-uniformization}, the circle packing/Riemann uniformization embedding of this random diagonal triangulation of $\bbZ^2$ is close to its a priori drawing on a large scale.
\end{example}

\begin{example}[Planar maps from random tiling models]
    Consider a random domino tiling of $\bbZ^2$ or a lozenge tiling such that it is stationary and ergodic with respect to spatial translations; for instance, take the measure of maximal entropy as in~\cite[Section 2.2]{Kenyondimer}. This generates a cell configuration $\cH$ where the cells are the domino/lozenge tiles and the associated map is the dual map of the tiling. By a further independent uniform translate inside the unit square or unit hexagon, it is straightforward to check that $\cH$ satisfies the conditions of Theorem~\ref{thm:CellSystemPack-extension} or Theorem~\ref{thm:invariance-uniformization-extension}, and therefore the circle packing in Theorem~\ref{thm:CellSystemPack-extension} and the Riemann uniformization embedding  of the dual map are close to the a priori embedding as drawn from the tiling.  One can also connect the centers of each face of the tiling to the midpoint of its incident edges and assign each vertex to the union of the connected components containing it. This generates a cell configuration $\cH$ whose associated map is the primal map of the tiling, and therefore Theorem~\ref{thm:CellSystemPack-extension} and Theorem~\ref{thm:invariance-uniformization-extension} imply that the primal map is also close to its circle packing/Riemann uniformization embedding on a large scale.
\end{example}
An important motivation for this work is that we, in future work, will apply the results of this paper to planar maps arising in the study of Liouville quantum gravity. A longstanding conjecture is that a large class of random planar maps, including uniform maps and spanning-tree weighted planar maps, converge to $\gamma$-Liouville quantum gravity under discrete conformal embeddings. There have been a  {few earlier works where} statements of this kind have been proven: For the mated-CRT map, this is done in~\cite{GMS21tutte} for the Tutte embedding  and in~\cite{bertacco2025scaling} for the Smith embedding.  In~\cite{HS19}, it is shown that under the Cardy embedding, uniform triangulations converge to $\sqrt{8/3}$-LQG. As an application of Theorem~\ref{thm:CellSystemPack} and Theorem~\ref{thm:invariance-uniformization}, in the recent work~\cite{MatedCRT}, we prove  convergence of the mated-CRT map to $\gamma$-LQG. Moreover, as an application of Theorem~\ref{thm:large}, based on the recent work~\cite{Tutte} on convergence of the Tutte embedding of tree-weighted planar maps and the aforementioned comparison of circle packing in different domains in~\cite[Section 1.3.2]{MatedCRT}, in~\cite{HY26a} we further prove the convergence of tree-weighted planar maps to $\sqrt2$-LQG under circle packing and Riemann uniformization embedding for disk, sphere and whole-plane topologies. We expect that by combining the results of this paper and ~\cite[Section 1.3.2]{MatedCRT} as well as the techniques from~\cite{Tutte,HY26a}, one can further prove the convergence to LQG under circle packing/Riemann uniformization embedding for other planar map models, including uniform triangulations and bipolar-oriented maps.

\subsection{Outline}\label{subsec:intro-5}

In this section we provide a moderately detailed overview of the content of the rest of the paper,  as well as the outline of the proof of Theorems~\ref{thm:CellSystemPack} and~\ref{thm:invariance-uniformization}. 

\medskip

\noindent
In \textbf{Section}~\ref{sec:pre}, we 
 {present various preliminaries}.
In Section~\ref{pre:triangulation}, we review various properties of circle packings and the Riemann surface associated with planar maps. 
In Section~\ref{subsec:ergodic-cell-system}, we gather some results on the ergodic averages of the ergodic modulo scaling cell configuration $\cH$ in the framework of~\cite[Section 2.2]{GMSinvariance}. We will repeatedly use the following two facts:  the maximum diameter of the cells which intersect a large origin-containing square is typically of smaller order than its side length (Lemma~\ref{lem:no-macroscopic-cell}), and the sum of $\diam(H)^2\deg(H)^4$ over all the cells $H$ intersecting a large square $S$ which is not too far from the origin is bounded from above by a deterministic constant times the area of $S$ (Lemma~\ref{lem:average-diamdeg4}). As a consequence, similarly as in ~\cite[Section 2.3]{GMSinvariance}, we start from a function $\phi_0:M(\cH)\to\bbC$ which serves as the  {(non-injective)} a priori embedding of $M(\cH)$ and prove some of its topological properties.
 {Much care needs to be taken with the exact definition of $\phi_0$ and we choose to define it by mapping each vertex of $M(\cH)$ to a uniform point in the corresponding cell of $\cH$,}
and then extend to every equilateral triangle in $M(\cH)$ via linear interpolation.    In Section~\ref{subsec:pre-invprinciple}, we recap the convergence of random walk on $\cH$ to Brownian motion from~\cite{GMSinvariance}, and use it to prove that for each large square $S$, there is \emph{some} planar  embedding of $\cH(S)$ that is close to $\cH(S)$ (Proposition~\ref{prop:almost-planarity}).

\medskip

\noindent In \textbf{Section}~\ref{sec:circle-packing}, we prove Theorem~\ref{thm:CellSystemPack}. 
 {In Section~\ref{subsec:pf-Cell-system-Pack} we} prove that if the random walk on two different embeddings of the same  {infinite} planar map both converge to  {planar} Brownian motion, then the two embeddings must be close. In order to see this, we consider the wrapping around events for the random walk, and use this to prove that under appropriate scaling, the natural homeomorphisms induced by the embeddings are tight on compact sets. Then under the image of the limiting homeomorphism, the trace of a 2D Brownian motion still has the law of the trace of a 2D Brownian motion. In Proposition~\ref{prop:BM-f} we prove that  {any} limiting homeomorphism must be a linear transform. 
 By the almost planarity  {result}
 from Proposition~\ref{prop:almost-planarity} it follows that in order to conclude the proof, it is sufficient to prove that the random walk with Dubejko weights on the cell configuration $\cH$ and on the circle packing $\cP$ of $\cH$ both converge to planar Brownian motion.  

To prove that the random walk on $\cH$ with Dubejko weights converges to Brownian motion, which is handled in Section~\ref{subsec:Covergence-RW-cell-Dubejko},  {a natural first attempt is to try} 
to directly use the convergence of random walks on $\cH$ with generic conductance from~\cite{GMSinvariance} (Theorem~\ref{thm:GMSinvariance}). To use this result, one needs to check that the conductance satisfies the translation invariance modulo scaling property as in Definition~\ref{def:translation-modulo-scaling}, and the corresponding moment bounds~\eqref{eq:GMSinvariance-moment}. For the first condition, we use Rodin and Sullivan's Ring Lemma~\cite{RodinSullivan87} along with the uniqueness of circle packing for triangulations to prove that if triangulations $\cT_n$ converge to an infinite plane triangulation $\cT$ in Benjamini-Schramm topology, then the radii of the disks in the circle packing of $\cT_n$ also converge to the radii of the disks in the circle packing of $\cT$ (Proposition~\ref{Prop:circle-conv}).  For the moment bound condition, the first half of~\eqref{eq:GMSinvariance-moment} is trivial since the Dubejko conductance is always bounded from above by 1. For the second half of~\eqref{eq:GMSinvariance-moment} involving the inverse conductance, if we directly use the Ring Lemma, which bounds the ratio of the radii of tangent disks in a circle packing in terms of degree, then the inverse conductance is bounded from above by  {$\exp(C\deg(H_0))$ for some $C>0$.}  {It is not possible to obtain the second half of~\eqref{eq:GMSinvariance-moment} using this bound; in order to do this we would have needed a stronger version of \eqref{eq:thm-inv-circle} which would have been hard to verify in the relevant applications.} 

To circumvent the aforementioned issue, in Section~\ref{subsec:ringlemma} we establish a variant of the Ring Lemma (Lemma~\ref{lem:circle}), where for three mutually tangent disks in a circle packing, the ratio between the radii of the second smallest and the smallest disks is bounded above by 100 times the square of the degree. The proof of Lemma~\ref{lem:circle} is based on a 
 {geometric argument and} 
Descartes' theorem (Lemma~\ref{thm:Descartes}). With Lemma~\ref{lem:circle}, we modify the proof of Theorem~\ref{thm:GMSinvariance} from ~\cite{GMSinvariance}, where  every edge $\{H,H'\}$ with large inverse Dubejko conductance can be replaced by a path with total inverse conductance controlled by $\deg(H)^2$. This is carried out in Proposition~\ref{prop:inv-circle}, where we are able to reduce the exponential degree bound from the Ring Lemma to {a polynomial bound which allows us to apply~\eqref{eq:thm-inv-circle}.} 
 {Proposition~\ref{prop:inv-circle} is a variant of Theorem~\ref{thm:GMSinvariance} from~\cite{GMSinvariance} for the case where the conductances are given by the Dubejko weights and where the moment bound~\eqref{eq:GMSinvariance-moment} has been replaced by~\eqref{eq:thm-inv-circle}.}

It remains to show that the random walk with Dubejko weights  on the circle packing $\cP$  converges to Brownian motion, which is carried out in Section~\ref{subsec:Covergence-RW-CP-Dubejko}. Following work by Bou-Rabee and Gwynne~\cite{bougwynne2024random} it is sufficient to check that up to scaling, there are no macroscopic disks in $\cP$. In~\cite{GDN19VEL}, Gurel-Gurevich, Jerison and Nachmias proved that the maximal diameter of the disks in a circle packing can be controlled by the inverse of Cannon's vertex extremal length~\cite{CannonVEL}. In Proposition~\ref{prop:finite-VEL}, using this and the averaging over $\diam(H)^2$ from Lemma~\ref{lem:average-diamdeg4}, we prove that for any number $q_0$, there are at most finitely many cells with associated vertex extremal length at most $q_0$, which eventually leads to the convergence of the random walk on $\cP$.

\medskip

\noindent In \textbf{Section}~\ref{sec:uniformization} we prove Theorem~\ref{thm:invariance-uniformization}.  {In~\cite{GMSinvariance}, Theorem~\ref{thm:GMSinvariance} is proven by studying discrete harmonic functions on the graph defined by $\cH$. Our proof is partially following their strategy, with the difference that we will instead be working with (continuum) harmonic functions defined intrinsically on the manifold $M(\cH)$.}

 {In Section~\ref{subsec:regularity} we establish a number of regularity estimates on the Riemann surface $M(\cH)$ under a conformal map $\varphi$.}
In particular, in Lemma~\ref{lem:area/length-1} and Lemma~\ref{lem:area/length-2}, we prove exact estimates on the continuity of  the mapping between $\cH$ and $\varphi(M(\cH))$. In order to prove this, we assign a semiflower $P_H \subseteq M(\cH)$ to each vertex $H$ (see Figure~\ref{fig:Koebe-distortion}) where the union of the semi-flowers is equal to the whole surface. In Lemma~\ref{lem:Koebe-distortion}, we use Koebe's distortion theorem to prove that, the ratio between the outradius and inradius of the conformal image   $\varphi(P_H)$ is bounded by a constant times $\deg(H)^2$.  Another ingredient for the proof of Lemma~\ref{lem:area/length-1} and Lemma~\ref{lem:area/length-2} is a length-area argument, which was introduced by He and Schramm~\cite{HeSchramm} in the context of circle packings and originated from the Length-Area Lemma in~\cite{RodinSullivan87}. The terms in our length-area argument will be controlled by the average of $\diam(H)^2\deg(H)^4$ as in Lemma~\ref{lem:average-diamdeg4}. Given Lemma~\ref{lem:area/length-1} and Lemma~\ref{lem:area/length-2}, in Lemma~\ref{lem:parabolic} we prove that the surface $M(\cH)$ is a.s.\ parabolic, and in Lemma~\ref{lem:cell-polynomial-growth} we prove that the mapping between $\cH$ and $\varphi(M(\cH))$ has polynomial growth.

 {In Section~\ref{subsec:GMS2.3}} we construct a harmonic function $\phi_\infty:M(\cH)\to\bbC$ using the function $\phi_0$ defined in Section~\ref{subsec:ergodic-cell-system}.
For each $m\in\bbN$ we let $\phi_m$ be equal to $\phi_0$ on the boundary of the $\phi_0^{-1}$ image of certain large squares,  and harmonic elsewhere. 
The Dirichlet energy of the piecewise linear function $\phi_0$ can be controlled via the average in Lemma~\ref{lem:average-diamdeg4}, and the  Dirichlet energy of $\phi_m-\phi_{m'}$ is small when $m,m'$ is large (Lemma~\ref{lem:dirichlet-B}). 
By standard properties of harmonic functions, this implies that the $\phi_m$'s are Cauchy in probability in the local uniform topology, and the limit $\phi_\infty$ of $\phi_m$ exists.

In Section~\ref{subsec:GMS2.4}, we prove that the harmonic function $\phi_\infty$ is close to $\phi_0$ on a large scale (Proposition~\ref{prop:harmonic-sublinear}). We first follow the same outline as~\cite[Section 2.4]{GMSinvariance} to prove that $\phi_\infty$ is close to $\phi_0$ on a   set of line segments{, where in our setting the line segments are contained in the image of a suitable subset of $M(\cH)$ under a conformal map}. To prove the closeness of $\phi_\infty$ to $\phi_0$ on the complementary set, the   oscillation of $\phi_\infty$  in each connected component  {of the complement} is equal to the oscillation of $\phi_\infty$ on the boundary of that component thanks to the maximum principle for harmonic functions, whereas the oscillation of $\phi_0$ in each connected component can be bounded by the continuity from Lemma~\ref{lem:area/length-2}. This shows that $|\phi_0-\phi_\infty|$ is also small in every connected component of the complementary set, which concludes the proof of Proposition~\ref{prop:harmonic-sublinear}.

 In Section~\ref{subsec:pf-thm:invariance-uniformization}, we conclude our proof of Theorem~\ref{thm:invariance-uniformization} by showing that $\phi_\infty$ has to be the composition of   a deterministic linear transform, a random rotation  {and a scaling}.   
Using the polynomial growth from Lemma~\ref{lem:cell-polynomial-growth},  the harmonic function  $\phi_\infty\circ\varphi^{-1}$ on $\bbC$ has polynomial growth and must be a polynomial (Lemma~\ref{lem:harmonic-polynomial}). Furthermore, in Lemma~\ref{lem:polynomial-C}, we prove that this polynomial must have degree 1 because if this is not the case, then one can find two cells that are close in $\cH$ and have  macroscopic distance in $\varphi(M(\cH))$, which contradicts the continuity as in Lemma~\ref{lem:area/length-1} and Lemma~\ref{lem:area/length-4}. Furthermore, we observe that the linear transform $\phi_\infty\circ\varphi^{-1}$ is measurable with respect to $\cH$ and $(\phi_\infty\circ\varphi^{-1})(\phi_\infty\circ\varphi^{-1})^\rmT/|\det(\phi_\infty\circ\varphi^{-1})|$ is invariant under translation or scaling of $\cH$ and therefore must be deterministic by Definition~\ref{def:ergodic-modulo-scaling}. The decomposition in Theorem~\ref{thm:invariance-uniformization} then follows from standard linear algebra. 
This finishes the proof of Theorem~\ref{thm:invariance-uniformization}.  Our proof in Section~\ref{subsec:pf-thm:invariance-uniformization} works under quite mild assumptions and gives a general method for proving that a random harmonic function is of a particular form (namely, the composition of a deterministic linear transformation and a random rotation and scaling) as long as it has polynomial growth, is almost injective, and the averages of the Dirichlet energy over large regions are approximately a deterministic constant.

\medskip

\noindent In \textbf{Section}~\ref{subsec:extension}, we extend our results to general planar maps by proving Theorem~\ref{thm:CellSystemPack-extension} and Theorem~\ref{thm:invariance-uniformization-extension}.
In Section~\ref{subsec:extension-circ}, we prove  Theorem~\ref{thm:CellSystemPack-extension}, where we construct a cell configuration $\wt\cH$ in the framework of Theorem~\ref{thm:CellSystemPack} from  $\cH$. In Section~\ref{subsec:extension-unif} we prove Theorem~\ref{thm:invariance-uniformization-extension}. We first add a vertex on each edge and face of the planar map $\cM$, and further add edges to turn $\cM$ into an infinite plane triangulation.  We  define a different version of semi-flowers for the vertices and the added vertices in Lemma~\ref{lem:Koebe-distortion} and prove that under conformal maps, the ratio of the outradius and inradius of the semi-flowers can be bounded by the square of the degree, as in Lemma~\ref{lem:Koebe-distortion-1}. We also split the cell configuration $\cH$ to get a new cell configuration whose associated map is the new infinite triangulation. This together with the same arguments in Section~\ref{sec:uniformization} gives the proof of Theorem~\ref{thm:invariance-uniformization-extension}. Finally, in Section~\ref{subsec:large}, we prove   Theorem~\ref{thm:large}, a version of Theorems~\ref{thm:CellSystemPack}-\ref{thm:invariance-uniformization-extension} under a weaker version of the line connectivity property.  

\medskip
\noindent\textbf{Acknowledgements.}  {We thank Ewain Gwynne for mentioning to us at an earlier stage of the project that Proposition \ref{prop:max-diam-0-cone} holds under the assumptions of \cite{GMSinvariance}, which weakened our initial assumptions.} 
N.H.\ was supported by grant
DMS-2246820 of the National Science Foundation and the Simons Collaboration Grant \emph{Probabilistic Paths to Quantum Field Theory}.

\medskip
\noindent\textbf{Statement on AI usage.} The first two arXiv versions of this paper were completed without any use of AI. Later we used ChatGPT Astra for proofreading and revising the paper.

\section{Preliminaries}\label{sec:pre}

\subsection{Planar maps, circle packings and Riemann surfaces}\label{pre:triangulation}

We follow the papers~\cite{BenjaminiSchramm01, ASUIPT, GRRiemann} for the definition of planar maps and triangulations. 
A planar map $\cM$  is a proper ({in particular,} without edge-crossings) embedding of a connected planar (multi)graph $\cG$ in the two-dimensional sphere $\wh{\bbC}=\bbC\cup\{\infty \}$,  where we identify two planar maps if one can be obtained from the other via an orientation-preserving homeomorphism of the sphere which maps vertices (resp.\ edges, faces) to vertices (resp.\ edges, faces). For convenience we  abbreviate ``equivalence class of embedded planar maps" to ``planar map". We will always assume that $\cM$ is locally finite, i.e., each vertex and each face has finite degree.  The face set $\cF\cM$ of $\cM$ is  the set of connected components of $\wh\bbC\backslash \cG$, and the support $S(\cM)\subseteq \wh\bbC$ is defined to be the union of $\cM$ and a subset of the faces in $\cF\cM$. In this paper we assume that all the faces in $S(\cM)$ are bounded in $\bbC$. We call a vertex $v\in\cV\cM$  an interior (resp.\ boundary) vertex of $\cM$ if $v$ is in the interior (resp.\ on the boundary) of the support $S(\cM)$. We say $\cM$ has \emph{{whole-}plane topology} if the underlying graph is one-ended and $\cM$ has no boundary vertices, and we say $\cM$ \emph{has disk topology}, if $\cM$ is finite and the support $S(\cM)$ is homeomorphic to a closed disk.  A \emph{rooted planar map} $(\cM, e)$ is a planar map $\cM$ together with an oriented edge $e = (x,y)$ of $\cM$, called the root edge. The vertex $x$ is the \emph{root vertex}. 

For a planar map $\cM$ on the sphere $\wh\bbC = \bbC\cup\{\infty\}$, we call a connected component of $\wh\bbC\backslash \cM$ a triangular face if its boundary consists of precisely three edges of $\cM$. An  \emph{embedded (planar) triangulation} $\cT$ is a finite or infinite connected planar map $\cM$ in the sphere, such that all the faces in the support $S(\cM)$ are triangular faces.
 Following the terminology in~\cite[Section 1.2]{ASUIPT}, we call $\cT$ a \emph{type II triangulation} if $\cT$ has no loops but possibly multiple edges, and call $\cT$ a \emph{type III triangulation} if $\cT$ has no loops and no multiple edges.  
We call a type II (resp.\ type III) triangulation $\cT$ an \emph{infinite plane triangulation} (resp.\ \emph{simple infinite plane triangulation}), if $\cT$ is a planar map with  {whole-}plane topology. {We call a type II (resp.\ type III) triangulation $\cT$ a (finite) \emph{disk triangulation} (resp.\ \emph{simple disk  triangulation}), if $\cT$ is a planar map with disk topology. 

Now we recall some fundamental results on circle packing; see~\cite{circlepackingbook} for further information on this topic. For a planar map $\cM$ with vertex set $\cV\cM$, we call a collection $\cP = (D_v:v\in\cV\cM)$ of closed disks a circle packing of $\cM$ if (i) each vertex in $\cM$ is assigned a unique disk $D_v\in\cP$, 
(ii) all the disks in $\cP$ have disjoint interiors, and 
(iii) two disks $D_v$ and $D_{v'}$ are tangent if and only if the two vertices $v$ and $v'$ are connected by an edge. The Koebe-Andreev-Thurston Circle Packing Theorem asserts that, for any finite simple {(i.e., with no multiple edges or loops)} planar map $\cM$, there exists a circle packing of $\cM$. Moreover, if $\cM$ is a simple disk triangulation, then up to M\"{o}bius transformations, there exists a unique circle packing of $\cM$ in $\bbD$ where the disks for the boundary vertices are   tangent to $\partial\bbD$. As explained in~\cite[Section 8]{circlepackingbook}, the existence and uniqueness of circle packing also hold for simple infinite triangulations. In this setting there are two regimes, where there is a circle packing $\cP$ such that the union of the disks in $\cP$ and the interstices (bounded connected components of $\bbC\backslash \cup_{v}D_v$)   is equal to the unit disk $\bbD$ or the whole plane $\bbC$. The planar map $\cM$ is called circle packing hyperbolic in the first case and circle packing parabolic in the second case. By~\cite[Theorem 1.1]{HeSchramm-hy-pa}, to determine whether $\cM$ is circle packing hyperbolic or circle packing parabolic, one can study the recurrence of the simple random walk on $\cM$ as below. 

\begin{theorem}[\cite{HeSchramm-hy-pa}]\label{thm:circle-packing-KAT}
    Let $\cT$ be a simple infinite plane triangulation. If the simple random walk on $\cT$ is recurrent, then $\cT$ is circle packing parabolic. Moreover, if $\cT$ has bounded degree and the simple random walk on $\cT$ is transient, then $\cT$ is circle packing hyperbolic.
\end{theorem}

\begin{figure}
    \centering
    \includegraphics[scale=0.55]{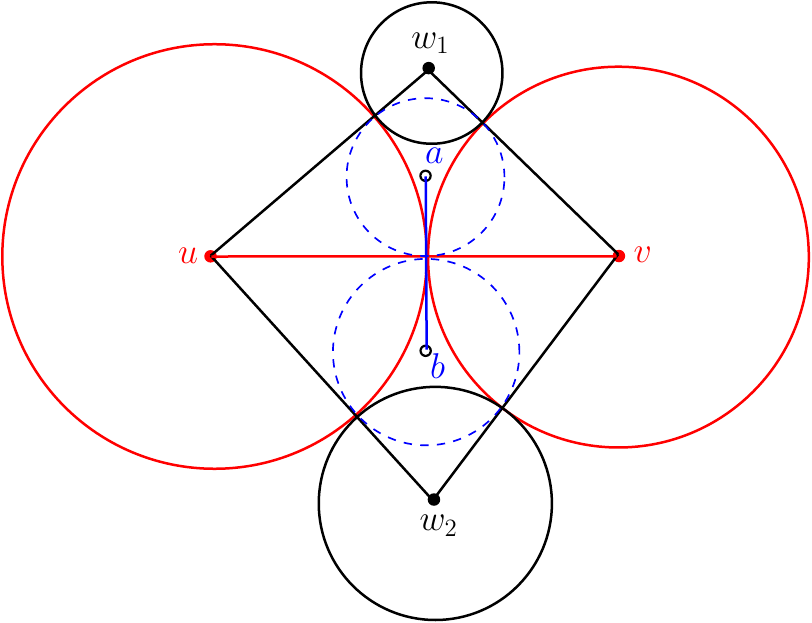}
    \caption{Two adjacent vertices $u,v$. The Dubejko conductance of the edge $uv$ is defined to be $\fc(u,v) = \frac{|o_{a}-o_{b}|}{|o_u-o_v|}=\frac{\sqrt{r_ur_v}}{r_u+r_v}\cdot\big(\sqrt{\frac{r_{w_1}}{r_{w_1}+r_u+r_v}}+\sqrt{\frac{r_{w_2}}{r_{w_2}+r_u+r_v}}\big)$.}
    \label{fig:Dubejko}
\end{figure}

One can also consider random walks on circle packings. Consider the circle packing $\cP$ of some simple infinite triangulation $\cT$. Let $u,v$ be adjacent vertices in $\cT$, and $w_1,w_2$ be the vertices  on the two faces of $\cT$ containing the edge $uv$  other than $u$ and $v$. Let $o_u,o_v,o_{w_1},o_{w_2}$ be the centers of the circles for $u,v,w_1,w_2$, and $o_a,o_b$ be the centers of circles which are inscribed in the triangles $o_uo_vo_{w_1}$ and $o_uo_vo_{w_2}$, respectively. The Dubejko conductance of the edge $uv$ is defined to be $\fc(u,v) = \frac{|o_{a}-o_{b}|}{|o_u-o_v|}$; see Figure~\ref{fig:Dubejko} for an illustration. If we write $r_u,r_v,r_{w_1},r_{w_2}$ for the radii of the disks $D_u,D_v,D_{w_1},D_{w_2}$, then 
\begin{equation}\label{eq:Dubejko}
    \fc(u,v) = \frac{\sqrt{r_ur_v}}{r_u+r_v}\cdot\bigg(\sqrt{\frac{r_{w_1}}{r_{w_1}+r_u+r_v}}+\sqrt{\frac{r_{w_2}}{r_{w_2}+r_u+r_v}}\bigg).
\end{equation}
One can infer from~\eqref{eq:Dubejko} that $\fc(u,v)<1$ for any edge $uv$ using $2\sqrt{r_ur_v}\leq r_u+r_v$. Furthermore, by Dubejko's theorem (see e.g.~\cite[Theorem 18.3]{circlepackingbook}), for a circle packing $\cP$ of a simple infinite triangulation, if we consider the random walk $(Z_n)_{n\geq0}$ moving between the centers of the circles in $\cP$, then $(Z_n)_{n\geq0}$ is a martingale.

Next we define the Riemann surfaces associated with general planar maps, following the presentation in~\cite[Section 6]{AHNR18unimodular}. Let $\cM$ be a planar map with  {whole-}plane topology. For  each face $f$ of $\cM$ of degree 1, i.e., $f$ is enclosed by a single edge of $\cM$, we remove this edge and the face from $\cM$. For each face $f$ of $\cM$ of degree 2, i.e., $f$ is formed by a double edge of $\cM$, we collapse this double edge into a single edge. Now we assume that each face of $\cM$ has degree at least 3.  We associate to $\cM$ a metric space $(M(\cM), d_M)$, by regarding each face of $\cM$ of degree $p$ as a regular $p$-gon of side length 1 and defining the distance between two points as the length of the shortest path joining them. Equivalently, $M(\cM)$ is obtained by gluing regular $p$-gons of unit side length according to their adjacency pattern in $\cM$. The metric space $M(\cM)$ can be equipped with a compatible Riemann surface structure by defining coordinate charts as follows.    For a face $f$, we use the regular polygon associated with $f$, viewed as a subset of the plane,  as the chart, and use the identity map as the projection.  For each edge $e$ of $\cM$, we define an open neighborhood of the interior of $e$ in $M(\cM)$ by adding to $e$ the two triangles formed by the endpoints of $e$ and the centers of the two faces adjacent to $e$. To define a coordinate chart on this neighborhood, we simply place the two triangles next to each other in the plane. The reason to take only a {subset of the faces} 
 and not the entire face{s} is that this chart is well defined even if both sides of $e$ are incident to the same face. For each vertex $v$ of $\cM$, we define an open neighborhood of $v$ in $M(\cM)$ similarly by
intersecting the corners of the faces adjacent to $v$ with open discs of radius 1/2 centered at $v$. We define a chart on this neighborhood by first laying the corners
out sequentially around the origin (with possible overlapping), and then applying the function $z\mapsto z^{2\pi/\theta(v)}$,  suitably interpreted to get an injective map into the plane. The radius is chosen so that this definition remains valid when multiple corners of the same face are located at the same vertex.
We will also work with the setting where $\cM$ is a disk triangulation;  in this case, the boundary vertices of degree $n$ can be treated similarly using the map $z\mapsto z^{3/(n-1)}$.


Following the uniformization theorem, for a   planar map $\cM$ with whole-plane topology, the surface $M(\cM)$ is either conformally equivalent to the complex plane and thus \emph{parabolic}, or conformally equivalent to the unit disk and thus \emph{hyperbolic}. As pointed out by~\cite{GRRiemann,AHNR18unimodular}, an equivalent characterization is that the Brownian motion on the surface $M(\cM)$ is either recurrent (in the parabolic case) or transient (in the hyperbolic case). A function $f:M(\cM)\to\bbC$ is harmonic if and only if $f\circ\varphi^{-1}$ is a harmonic function in the usual sense, where $\varphi:M(\cM)\to\bbC$ is some conformal map.

\subsection{Ergodic averages of cell configurations using dyadic systems}\label{subsec:ergodic-cell-system}

In this subsection we gather some results on the ergodic averages of the cell configuration $\cH$ in the framework of~\cite[Section 2.2]{GMSinvariance}. We start with  the definition of dyadic systems and uniform dyadic systems as in~\cite{GMSinvariance}.  For a square $S$, we write $|S|$ for the side length of $S$. A \emph{dyadic child} of $S$ is one of the 4 squares with half side length of $S$ whose corners include one corner of $S$ and the center of $S$. A \emph{dyadic parent} of $S$ is one of the four squares with $S$ as a dyadic
child. Then each square has 4 dyadic children and 4 dyadic parents. A dyadic descendant (resp.\ dyadic ancestor) of $S$ is a square which can be obtained from $S$ by iteratively choosing dyadic children (resp.\ parents) finitely many times. A \emph{dyadic system} is a collection $\cD$ of closed squares with their sides parallel to the $x$-axis or $y$-axis with the following properties:
\begin{enumerate} 
    \item If $S\in\cD$, then each of the four dyadic children of $S$ is in $\cD$.
    \item If $S\in\cD$, then exactly one of the dyadic parents of $S$ is in $\cD$.
    \item Any two squares in $\cD$ have a common dyadic ancestor.
\end{enumerate}
As explained in~\cite[Section 2.1]{GMSinvariance}, one can define a topology on the space of dyadic systems by looking at the local Hausdorff distance on the union of the origin-containing squares.
\begin{definition}\label{def:dyadicsystem}
Let $s$ be uniformly sampled from $[0,1]$, and conditioned on $s$, we uniformly randomly sample $w$ in $[0,2^s]\times[0,2^s]$. Let $S_0 = [0,2^s]\times[0,2^s]-w$. For $k\in\bbN$, inductively let $S_k$ be sampled uniformly from the four dyadic parents of $S_{k-1}$.  The \emph{uniform dyadic system} $\cD$ is defined by the set of all dyadic descendants of $S_k$ for each $k\in\bbN$. For $z\in\bbC$ and $m>0$, we define
\begin{equation}\label{eq:def-S_m^z}
    \wh S_m^z:=\{\text{largest square }S\in\mathcal{D}: z\in S \text{ and } \sum_{H\in\cH(S)}\frac{\area(H\cap S)}{\area(H)} \leq m\}
\end{equation}
which is well defined on the probability one event that $z$ is not on the boundary of any square in $\cD$, and we abbreviate $\wh S_m:=\wh S_m^0$. Finally, for $C>0$ and $z\in\bbC$, we write $C(\cD-z)$ for  the collection of the images of the squares in $\cD$ under the map $w\mapsto C(w-z)$.
\end{definition}
In other words, $\wh S_m^z$ is the largest square containing $z$ which contains at most $m$ cells, with fractional parts of cells which intersect the boundary of the square  also counted.

The following is the dyadic system resampling property for random cell configurations. 
\begin{lemma}[Lemma 2.3 of~\cite{GMSinvariance}]\label{lem:dyadic-resample}
   The condition in Definition~\ref{def:translation-modulo-scaling} on random cell configurations $\cH$ is equivalent to the following.  Let $\cD$ be a uniform dyadic system independent from $\cH$ and define $\wh S_m = \wh S_m^0$ for $m>0$ as in~\eqref{eq:def-S_m^z}. For each $m>0$, conditional on $\cH$ and $\cD$, let $z$ be sampled uniformly from the Lebesgue measure on $\wh S_m$. Then there exists a random $C>0$ possibly depending on $z$, such that
   $$(C(\cH-z),C(\cD-z))\overset{d}{=} (\cH,\cD).$$
\end{lemma}
In the rest of this section, we will restrict to the ergodic modulo scaling cell configurations as in Definition~\ref{def:ergodic-modulo-scaling}.  
The following ergodicity statement will be a key tool in what follows.
\begin{lemma}[Lemma 2.7 of~\cite{GMSinvariance}]\label{lem:ergodic-cell-system}

    Let $F = F(\cH,\cD)$ be a measurable function on the space of cell configuration/dyadic system pairs which is scale invariant (i.e., $F(C\cH,C\cD) = F(\cH,\cD)$ for every $C>0$). Let $\mathcal{I}$ be the $\sigma$-algebra generated by all such functions $F$ which are also translation invariant in the sense that $F(\cH-z,\cD-z)=F(\cH,\cD)$ for every $z\in\bbC$. Then $\mathcal{I}$ is independent from $\cH$. If either $\bbE[|F|]<\infty$  or $F\geq 0$ a.s., then a.s., 
    \begin{equation}
       \lim_{k\to\infty}\frac{1}{|S_k|^2}\int_{S_k} F(\cH-z,\cD-z) \,dz = \bbE[F|\mathcal{I}]. 
    \end{equation}
\end{lemma}

We comment that our version of Lemma~\ref{lem:dyadic-resample} and Lemma~\ref{lem:ergodic-cell-system} are slightly different from the ones in~\cite{GMSinvariance} since there they only have the relation $\sim$ instead of our associated map $\cM$ in the definition of cell configurations, but the same proof carries to our setting.

We will also use the following version of Lemma~\ref{lem:ergodic-cell-system} at a fixed scale.
\begin{lemma}\label{lem:ergodic-cell-system-B}
    Let $F$ be the same as in Lemma~\ref{lem:ergodic-cell-system}. Then almost surely, 
    \begin{equation}\label{lem:ergodic-cell-system-BBB}
        \limsup_{r\to\infty} \frac{1}{r^2}\Big|\int_{[-r,r]^2} F(\cH-z,\cD-z)\,dz\Big|\leq 64\bbE\big[|F|\big|\mathcal{I}\big].
    \end{equation}
\end{lemma}

\begin{proof}
  Without loss of generality assume $F\geq 0$.  Let $\cD^*$ be a uniform dyadic system independent from $(\cH,\cD)$. Define $\wh S_m^*$ in the analogous way using $(\cH,\cD^*)$ in place of $(\cH,\cD)$ as in Definition~\ref{def:dyadicsystem} Let $S_k^*$ be the corresponding origin containing squares, and  $z_m^*$ be a point sampled uniformly from the Lebesgue measure on  $\wh S_m^*$. Then by Lemma~\ref{lem:dyadic-resample}, there is a random constant $C^*>0$ independent from $\cD$ such that  $(C^*(\cH-z_m^*),C^*(\cD^*-z_m^*))\overset{d}{=} (\cH,\cD^*).$ Using the independence between $(C^*,z_m^*)$ and $\cD$, it follows from ~\cite[Lemma 2.1]{GMSinvariance} that  $(C^*(\cH-z_m^*),C^*(\cD-z_m^*),C^*(\cD^*-z_m^*))\overset{d}{=} (\cH,\cD,\cD^*).$ 
  Let $\cJ$ be the $\sigma$-algebra generated by the functions $G$ such that $G(C(\cH-z),C(\cD-z),C(\cD^*-z)) = G(\cH,\cD,\cD^*)$ for any deterministic $C>0$, $z\in\bbC$. Note that $\cI\subseteq\cJ$. Then the above resampling property as well as the backward martingale convergence theorem as in the proof of~\cite[Lemma 2.7]{GMSinvariance} shows that almost surely,
  \begin{equation}\label{eq:ergodic-aaa}
     \frac{1}{|S_k^*|^2} \int_{S_k^*} F(\cH-z,\cD-z)\,dz \to \bbE[F|\cJ].
  \end{equation}
  Furthermore, we claim that $\bbE[F|\cJ] = \bbE[F|\cI]$. To see this, for any $A\in\cJ$, let $p_A(\cH,\cD) = \bbP[A|\cH,\cD]$, which is $\cI$-measurable by further averaging over $\cD^*$.  Then $$\bbE[F(\cH,\cD)\mathds{1}_A] = \bbE [F(\cH,\cD)p_A(\cH,\cD)] = \bbE\big[\bbE[F(\cH,\cD)|\cI]p_A(\cH,\cD) \big] = \bbE[\bbE[F(\cH,\cD)|\cI]\mathds{1}_A ],$$
  which verifies our claim.

  Now~\eqref{eq:ergodic-aaa} and the fact that $\bbE[F|\cJ] = \bbE[F|\cI]$ implies that for large enough $k$, \eqb\label{eq:ergodic-aaaaa} \frac{1}{|S_k^*|^2} \int_{S_k^*} F(\cH-z,\cD-z)\,dz\leq \bbE[F|\cI]+\e.\eqe The stopping time argument in the proof of~\cite[Lemma 2.8]{GMSinvariance} shows that for large enough $k$, ~\eqref{eq:ergodic-aaaaa} remains true if we replace $S_k^*$ with all four possible dyadic parents of $S_k^*$. Then for large enough $r$, we may choose $k$ such that $r\leq |S_k^*|\leq 2r$, and thus
  $$\int_{[-r,r]^2} F(\cH-z,\cD-z)\,dz \leq 16|S_k^*|^2(\bbE[F|\cI]+\e)\leq 64r^2(\bbE[F|\cI]+\e),$$
  and~\eqref{lem:ergodic-cell-system-BBB} follows since $\e$ is arbitrary.
\end{proof}

For $k>\ell\geq 1$, we write $\mathcal{S}_{k,\ell}$ for the collection of squares (possibly not in $\cD$) consisting of the dyadic ancestors $S$ of $S_k$ of side length at most $2^\ell|S_k|$, as well as each dyadic descendant of such $S$ of side length at least $2^{-\ell}|S_k|$. In particular, the 8 squares in $\cD$ with length $|S_k|$ and non-empty intersection with $S_k$ are in $\cS_{k,\ell}$. The results below are stated in terms of $S_k$, but they are true for $\wh S_k$ as well since each $\wh S_k$ is equal to $S_{m(k)}$ for some $m(k)$ and $m(k)\to\infty$ as $k\to\infty$.
\begin{lemma}\label{lem:dyadic-shift}
    Let $E(S) = E(S,
    \cH)$ be an event depending on a square $S\subseteq\bbC$ and our cell configuration $\cH$. Suppose $E$ is translation invariant, in the sense that for any $z\in\bbC$,   $E(S,\cH)$ happens if and only if $E(S-z,\cH-z)$ happens. Suppose that a.s.\ $E(S_k)$ occurs for each large enough $k\in\bbN$. Then for each $\ell\in \bbN$, it is a.s.\ the case that for large enough $k\in\bbN$, $E(S)$ occurs for each square $S$ in $\mathcal{S}_{k,\ell}$. 
\end{lemma}

\begin{proof}
  Consider the event $E_1(S)$ where $E(S')$ occurs for each  dyadic descendant  $S'$ of $S$ with side length at least $2^{-2\ell}|S|$. Then by~\cite[Lemma 2.8]{GMSinvariance}, a.s.\ $E_1(S_k)$  occurs for each large enough $k\in\bbN$. The claim is then straightforward by applying ~\cite[Lemma 2.8]{GMSinvariance} once more to the event $E_1(S)$.
\end{proof}

\begin{lemma}\label{lem:no-macroscopic-cell}
    Almost surely, for each $\e\in(0,1)$ and $\ell>0$, it holds that for large enough $k\in\bbN$ that 
    \begin{equation}
        \diam(H)\leq \e|S| \ \ \text{for all } H\in \cH(S) \text{ and $S\in\cS_{k,\ell}$}. 
    \end{equation}
\end{lemma}

\begin{proof}
    By~\cite[Lemma 2.9]{GMSinvariance}, where we set the conductance here to be equal to 1 and~\eqref{eq:thm-inv-circle} implies the moment bound in~\cite[Equation (1.5)]{GMSinvariance}, almost surely it holds that for large enough $k\in\bbN$ such that 
    \begin{equation*}
        \diam(H)\leq \e|S_k|/2, \ \ \text{for all   } H\in \cH(S_k). 
    \end{equation*}
    The claim then follows from Lemma~\ref{lem:dyadic-shift}.
\end{proof}

\begin{lemma}\label{lem:average-diamdeg4}
    Suppose for some $p\geq 0$, the cell configuration $\cH$ satisfies the moment bound
    \begin{equation}\label{eq:lem:moment-bound-p}
        \bbE\Big[\frac{\diam(H_0)^2}{\area(H_0)}\deg(H_0)^p \Big]<\infty.
    \end{equation}
    There is a deterministic constant $C>0$ depending only on the law of $\cH$ such that the following is true. For each $\ell\in\bbN$, it is a.s.\ the case that for large enough $k\in\bbN$ and each $S\in\mathcal{S}_{k,\ell}$, 
    \begin{equation}\label{eq:average-diamdeg4}
       \frac{1}{|S|^2} \sum_{H\in \cH(S)} \diam(H)^2\deg(H)^p <C.
    \end{equation}
    Moreover, almost surely, for each  large enough $k$ and each square $S\subseteq [-2^k,2^k]\times [-2^k,2^k]$ of side length at least $2^{k-\ell}$, ~\eqref{eq:average-diamdeg4} holds. 
\end{lemma}

\begin{proof}
    Following the same argument as in~\cite[Lemma 2.10]{GMSinvariance}, applying Lemma~\ref{lem:ergodic-cell-system} with $$F(\cH,\cD) = \frac{\diam(H_0)^2}{\area(H_0)}\deg(H_0)^p$$
    and Lemma~\ref{lem:no-macroscopic-cell}, it follows that for some deterministic $C>0$, almost surely for each large enough $k\in\bbN$,
    \eqb\label{eq:pf-average-diamdeg4}  \frac{1}{|S_k|^2}\sum_{H\in \cH(S_k)}\diam(H)^2\deg(H)^p<C/16.\eqe
    By Lemma~\ref{lem:dyadic-shift}, almost surely for large enough $k$, for each $S\in\mathcal{S}_{k,\ell+2}$,  ~\eqref{eq:pf-average-diamdeg4} holds for $S$ instead of $S_k$. This verifies the first part of the claim. The second part of the claim follows since each square $S\subseteq [-2^k,2^k]\times [-2^k,2^k]$ of side length at least $2^{k-\ell}$ can be covered using squares in $\mathcal{S}_{k,\ell+2}$ with side length $2^{-\ell-2}|S_k|$ such that their total area is less than $16|S|^2$.
\end{proof}

\begin{figure}
    \centering
    \includegraphics[scale=0.45]{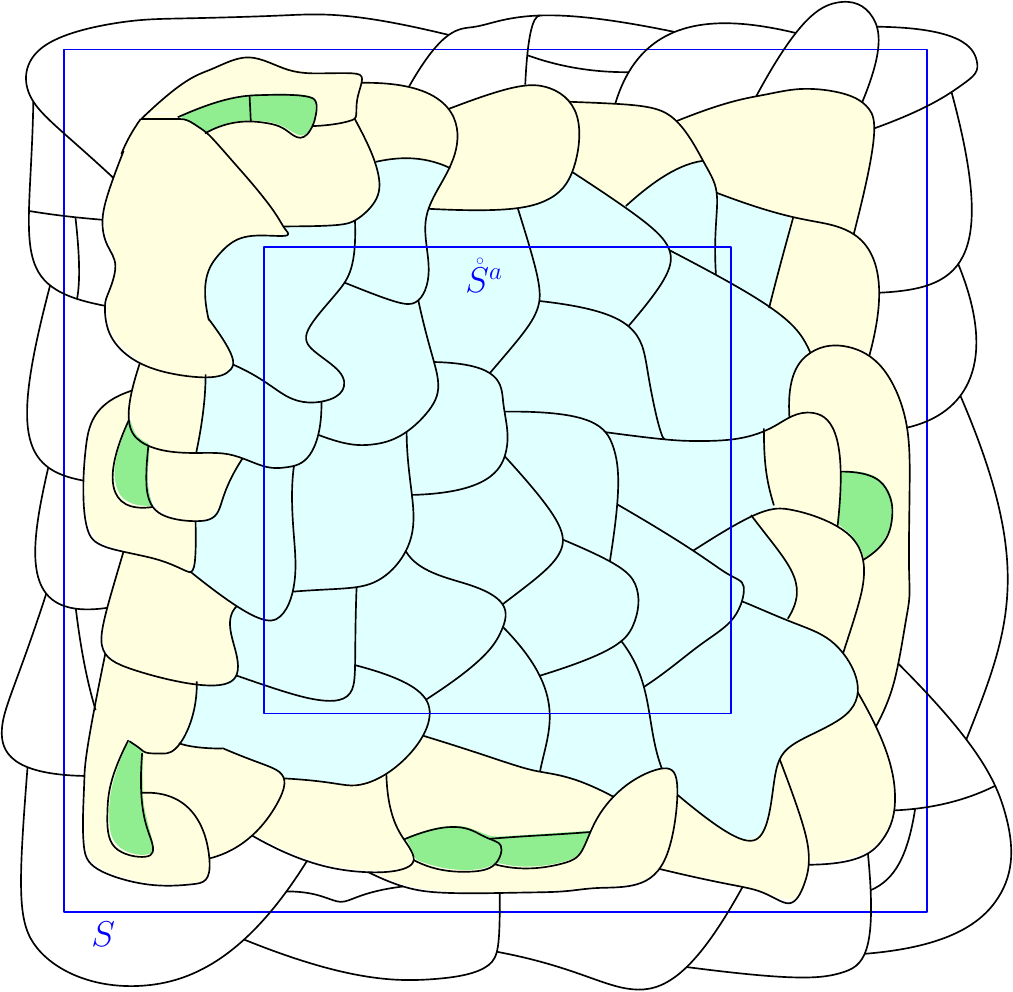}
    \caption{An illustration of the definition of $M(S;a)$ and $M(S)$. For simplicity here we assume that two cells are adjacent if they have nonempty intersection. The cells in  cyan are those intersecting $\mathring{S}^a$, and their neighbors are in  yellow. We also include the cells that are separated from infinity by the  yellow and  cyan cells, which are colored in green.  The surface $M(S;a)\subseteq M(\cH)$ is formed by vertices for these cells on $M(\cH)$ as well as the edges connecting these vertices, and the surface $M(S)$ is the largest surface of the form $M(S;a)$ such that for each vertex on $M(S;a)$, its cell is contained within $S$.}
    \label{fig:M(S)}
\end{figure}

For a square $S$, we define a disk triangulation $\cM(S)$ using cells in $\cH(S)$ as follows.
Let $\cM$ be the associated map of the cell configuration $\cH$. 
For every $a>0$, let $\mathring{S}^a$ be the square concentric to $S$ with side length $a$. We work on the event that  $\cH(\mathring{S}^a)$ is connected; otherwise we set $\cM(S;a)=\emptyset$. Let $\cM_0(S;a)\subseteq \cM$ be the submap formed by the union of faces of $\cM$ with at least one of the vertices in $\cH(\mathring{S}^a)$. 
Note that when viewed as a graph, $\cM_0(S;a)$ is 2-connected since $\cH(\mathring{S}^a)$ is connected. In fact, one can verify that, for any infinite plane triangulation $\cT$ and a connected subgraph $\cG_1\subseteq\cG$, if we let $\cG_2$ be the subgraph $\cG$ formed by all the vertices in $\cG_1$ and their neighbors, as well as all the edges connecting these vertices, then $\cG_2$ is 2-connected. We further let $\cM(S;a)$ be the smallest simply connected submap of $\cM$ containing $\cM_0(S;a)$ as a submap. Then $\cM(S;a)$  {is} a disk triangulation. 
Let $\cM(S)$ be the largest submap of the form $\cM(S;a)$ for some $a\in(0,|S|)$ such that for each vertex in $\cM(S;a)$, its cell is contained in the interior of $S$. Note that these objects are well defined since $\cH$ is locally finite. We further write $M(S)$ for the Riemann surface associated with $\cM(S)$. See Figure~\ref{fig:M(S)} for an illustration.  

\begin{lemma}\label{lem:M(S)}
     Almost surely, for each given $\e>0$ and $a\in(0,1)$, for large enough $k$, for each square $S\subseteq [-2^k,2^k]^2$ of side length at least $a2^k$, $\cM(S;(1-\e)|S|)\subseteq \cM(S)$.
\end{lemma}

\begin{proof}
    By the connectivity in Definition~\ref{def:connectedness}, for each $b\in(0,1)$, almost surely for large enough $k$, for each square $S\subseteq [-2^k,2^k]^2$ of side length at least $a2^k$, $\cH(\mathring{S}^{b|S|})$ and $\cH(\ol{\bbC\backslash{S}})$ are connected subgraphs of $\cH$. Let $H\in\cH$ be such that its vertex $v_H$ on $\cM $ is in $\cM(S;(1-\e)|S|)$. Then there are two cases: (i) $H$ is adjacent to some vertex $H'\in \cH(\mathring{S}^{(1-\e)|S|})$ or (ii) there exists a loop consisting of vertices for $H_1,...,H_n$ on $\cM $ surrounding $v_H$, such that each $H_j$ is adjacent to some cell in $\cH(\mathring{S}^{(1-\e)|S|})$. In the setting (i), by Lemma~\ref{lem:no-macroscopic-cell}, if we assume that $k$ is large, then $H\subseteq S$. For the case (ii), by Lemma~\ref{lem:no-macroscopic-cell}, for large $k$, $H_j\cap \ol{\bbC\backslash{S}}=\emptyset$ for each $j=1,...,n$. Since  $\cH(\ol{\bbC\backslash{S}})$ is infinite and connected, we see that $H\notin \cH(\ol{\bbC\backslash{S}})$ and therefore $H\subseteq S$. 
    This finishes the proof. 
\end{proof}

Let $\cH$ be the cell configuration in the setting of Theorem~\ref{thm:CellSystemPack} or Theorem~\ref{thm:invariance-uniformization}. We introduce a function $\phi_0:M(\cH)\to\bbC$ below that will later be used in Section~\ref{subsec:pre-invprinciple} and Section~\ref{sec:uniformization}. {The function $\phi_0$ sends each vertex on $M(\cH)$ to a point in the associated cell of $\cH$ and is otherwise piecewise linear. More precisely, the function $\phi_0:M(\cH)\to\bbC$ is defined as follows.} 
Given a realization of $\cH$, for each cell $H\in\cH$, we independently sample a point $\wt c(H)$ according to the probability measure proportional to the Lebesgue measure on $H$. For each triangular face $\Delta$ in $M(\cH)$ formed by the cells $H_1,H_2,H_3\in\cH$, consider the triangle $\Delta'$ in $\bbC$ formed by connecting $\wt c(H_1),\wt c(H_2),\wt c(H_3)$ using straight line segments, and define $\phi_0|_{\Delta}:\Delta\to\Delta'$ to be the unique affine map sending the vertices in $\Delta$ to their counterpart in $\Delta'$. This defines a continuous and piecewise linear function $\phi_0:M(\cH)\to\bbC$. See Figure~\ref{fig:phi0} for an illustration. 

\begin{figure}
    \centering
    \includegraphics[scale=0.65]{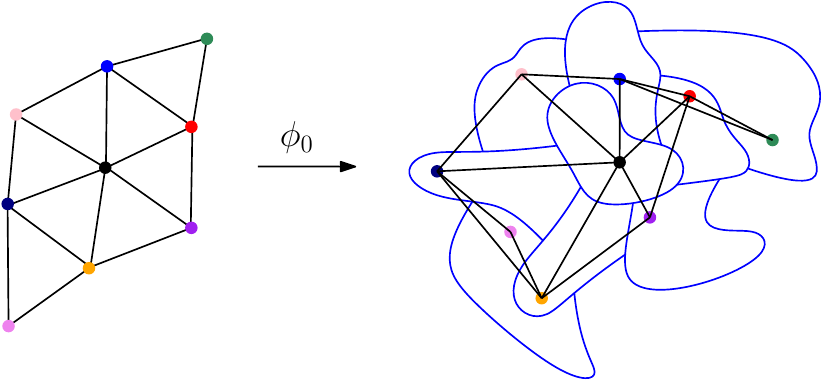}
    \caption{An illustration for the definition of the map $\phi_0:M(\cH)\to\bbC$, which maps each triangle affinely on the left panel to the corresponding triangle on the right panel. Note that the $\phi_0$-images of different equilateral triangles of $M(\cH)$ might overlap with each other.}
    \label{fig:phi0}
\end{figure}

We end this subsection with the following topological properties on $\phi_0$.
\begin{lemma}\label{lem:phi_0-bbC}
   We have $\phi_0(M(\cH))=\bbC$. 
\end{lemma}

\begin{proof}
    Pick $k$ large enough such that by  Definition~\ref{def:connectedness}, the 4 subgraphs of $\cH$ formed by the cells intersecting the 4 sides of $S_k$ are connected.  {Since a cell intersecting a vertex of $S_k$ is contained in (at least) two of these subgraphs, we get that the set of cells intersecting a side of $S_k$ forms a connected graph.} By further applying Lemma~\ref{lem:no-macroscopic-cell}, we can find a chain of cells $H_1\sim H_2\sim\cdots\sim H_n\sim H_1$  {in this connected graph} such that if we connect $\wt c(H_j)$ and $\wt c(H_{j+1})$ using straight line segments for $j=1,...,n$ with $n+1$ identified with 1, then this creates a loop $\gamma_n$ surrounding the square $\wt S_k$  {concentric} to $S_k$ with side length $|S_k|/2$, and the winding number of $ \gamma_n$ 
    around each point in $\wt S_k$ is equal to 1. If we consider the vertices for $H_1,...,H_n$ on the surface $M(\cH)$, then one can find a loop $\wt\gamma_n$ on $M(\cH)$  {connecting these vertices with straight line segments} such that $\phi_0(\wt\gamma_n)=\gamma_n$. Since $M(\cH)$ is simply connected, we can continuously shrink $\wt\gamma_n$ to a single point $x$. If there exists a point  $w\in\wt S_k\backslash \phi_0(M(\cH))$, then under $\phi_0$ image,  {all of these shrinking loops will} have winding number 1 around $w$. On the other hand, for loops sufficiently close to $x$, its image under $\phi_0$ has winding number 0 around $w$ since $w\notin \phi_0(M(\cH))$. This is a contradiction, which implies that almost surely for large enough $k$, $\wt S_k\subseteq \phi_0(M(\cH))$, which completes the proof.
\end{proof}

\begin{lemma}\label{lem:phi_0-bbC-3}
    For every compact set $K\subseteq\bbC$, $\phi_0^{-1}(K)$ has compact closure in $M(\cH)$.
\end{lemma}

\begin{proof}
    By Lemma~\ref{lem:no-macroscopic-cell}, a.s.\ for all large enough $k$, any cells in $\cH([-2^{k},2^k]^2)$ has diameter less than $2^{k-10}$, and the convex hull of the union of any three cells $H_1\sim H_2\sim H_3$ in $\cH([-2^{k},2^k]^2\backslash[-2^{k-2},2^{k-2}]^2)$ does not intersect $K$. This gives the claim.
\end{proof}

\begin{lemma}\label{lem:phi_0-bbC-1}
    Almost surely for large enough $k$, the following is true. For any cells $H_1,...,H_n\in\cH(\partial[-2^k,2^k]^2)$ such that $H_1\sim H_2\sim...\sim H_n\sim H_1$ form a simple loop $\gamma_n$ on $M(\cH)$ {by drawing a straight line between the vertices representing adjacent cells}, if $\phi_0(\gamma_n)$ has winding number 1 around 0, then the loop $\gamma_n$ encloses $M([-2^k,2^k]^2)$.
\end{lemma}

\begin{proof}
    From the definition, $M([-2^k,2^k]^2)$ has disk topology and is disjoint from $\gamma_n$. Suppose $k$ is large enough such that by Lemma~\ref{lem:phi_0-bbC-3}, $\phi_0^{-1}(0)\subseteq M([-2^{k-1},2^{k-1}]^2)$. If $\gamma_n$ does not enclose $M([-2^k,2^k]^2)$, then one can continuously shrink $\gamma_n$ to a single point $x$ on $M(\cH)$ such that all the loops are disjoint from $M([-2^k,2^k]^2)$. In particular, under the image of $\phi_0$,  {all of these shrinking loops will} have winding number 1 around 0. On the other hand, $x\notin M([-2^k,2^k]^2)$ and $\phi_0(x)\neq 0$. This is a contradiction.
\end{proof}

\subsection{Convergence of random walks on cell configurations}\label{subsec:pre-invprinciple}
In order to prove Theorem~\ref{thm:CellSystemPack} and Theorem~\ref{thm:invariance-uniformization}, we will need that for each large square $S$, $\cH(S)$ is close to \emph{some} planar embedding of the map $\cM(S)$. To achieve this, in Proposition~\ref{prop:almost-planarity} below,  we will use the convergence of random walks on cell configurations in~\cite{GMSinvariance} to prove that for large squares $S$, $\cH(S)$ is close to the Tutte embedding of $\cM(S)$.

Let $\beta_1:[0,T_{\beta_1}]\to\bbC$ and $\beta_2:[0,T_{\beta_2}]\to\bbC$ be continuous curves defined on possibly different time intervals. We define 
\begin{equation}\label{eq:curve-metric-0}
    \mathbb{d}^{\rm CMP}(\beta_1,\beta_2):=\inf_\phi \sup_{t\in[0,T_{\beta_1}]}|\beta_1(t)-\beta_2(\phi(t))|,
\end{equation}
where the infimum is over all increasing homeomorphisms $\phi:[0,T_{\beta_1}]\to [0,T_{\beta_2}]$ and ``CMP'' stands for ``curves modulo parameterization". For curves  $\beta_1:[0,\infty)\to\bbC$ and $\beta_2:[0,\infty)\to\bbC$ defined for infinite time and $r>0$, we let $T_{1,r}$ (resp.\ $T_{2,r}$) be the  first exit time of $\beta_1$ (resp.\ $\beta_2$) from the ball $B(0;r)$  (or 0 if the curve starts outside $B(0;r)$). We define
\begin{equation}\label{eq:curve-metric}
    \mathbb{d}^{\rm CMP}_{\rm loc}(\beta_1,\beta_2):= \int_1^\infty e^{-r}\min\{1, \mathbb{d}^{\rm CMP}(\beta_1|_{[0,T_{1,r}]},\beta_2|_{[0,T_{2,r}]}) \}\,dr, 
\end{equation}
so that $\mathbb{d}^{\rm CMP}_{\rm loc}(\beta_n,\beta)\to0$ if  and only if for Lebesgue a.e.\ $r>1$, $\beta_n$ stopped at its first exit time from $B(0;r)$ converges to $\beta$  stopped at its first exit time from $B(0;r)$ with respect to the metric~\eqref{eq:curve-metric-0}. 
\begin{theorem}[Theorem 3.10 of~\cite{GMSinvariance}]\label{thm:GMSinvariance}
    Let $\cH$ be a random cell configuration satisfying the conditions in Definition~\ref{def:ergodic-modulo-scaling} and Definition~\ref{def:connectedness}. Further assume $\cH$ satisfies the moment bound
    \begin{equation}\label{eq:GMSinvariance-moment}
        \bbE\bigg[\frac{\diam(H_0)^2}{\area(H_0)}\pi(H_0)\big)\bigg]<\infty, \ \  \bbE\bigg[\frac{\diam(H_0)^2}{\area(H_0)}\pi^*(H_0)\big)\bigg]<\infty
    \end{equation} 
    where $\pi(H) = \sum_{H'\sim H}\fc(H,H')$ and $\pi^*(H) = \sum_{H'\sim H}\fc(H,H')^{-1}$. Then the simple random walk on $\cH$ is a.s.\ recurrent. Moreover, for $z\in\bbC$, let $X^z$ denote the random walk on $\cH$ started from $H_z$ with conductance $\fc$. For $j\in\bbN_0$, let $\wh X_j^z$ be an arbitrarily chosen point of the cell $X_j^z$ and extend $\wh X_j^z$ from $\bbN_0$ to $[0,\infty)$ by piecewise linear interpolation. There is a deterministic covariance matrix $\Sigma$ with $\det \Sigma\neq0$ such that the following is true.  For each fixed compact set $A\subseteq\bbC$,  it is a.s.\ the case that as $\e\to0$, the maximum over all $z\in A$ of the Prokhorov distance between the conditional law of $\e\wh X^{z/\e}$ given $\cH$ and the law of Brownian motion started from $z$ with covariance matrix $\Sigma$, with respect to the topology on curves induced in~\eqref{eq:curve-metric}, tends to 0.
\end{theorem}

We remark that the original version of Theorem~\ref{thm:GMSinvariance} concerned the setting where the line connectivity property in Definition~\ref{def:connectedness} holds for every horizontal and vertical line segment, but the same proof works if we reduce to the line connectivity property for macroscopic line segments as in Definition~\ref{def:connectedness}. See also the proof of Theorem~\ref{thm:large} for a discussion of the line connectivity property in~\cite{GMSinvariance}.

\begin{corollary}\label{cor:GMS-rotation-inv}
    In the setting of Theorem~\ref{thm:GMSinvariance}, if $\cH$ is rotation invariant in law for some $\theta_0\in (0,2\pi)\backslash\{\pi\}$ in the sense of Definition~\ref{def:rotation-inv}, then $\Sigma$ is a scalar times the identity matrix.
\end{corollary}
\begin{proof}
   Write  $\mathbf{R}_{\theta_0}$ for the  matrix of rotating by angle $\theta_0$.  One can directly check by computation that for a 2D Brownian motion $(X,Y)^\rmT$ with covariance matrix $\Sigma$, if  $\mathbf{R}_{\theta_0}(X,Y)^\rmT$ is also a 2D Brownian motion with the same covariance matrix $\Sigma$, then $\Sigma$ must be a scalar times the identity matrix. This immediately implies Corollary~\ref{cor:GMS-rotation-inv}. 
\end{proof}

As studied in~\cite{GMS21tutte}, one main application of Theorem~\ref{thm:GMSinvariance} is the convergence of random planar maps under the Tutte embedding. Let $\cM_0$ be a disk triangulation and $v\in\cV \cM_0$ be an interior vertex. The Tutte embedding of $\cM_0$ in $\bbD$ centered at $v$ is defined as follows. Enumerate the boundary vertices of $\cM_0$ in counterclockwise order as $\{y_1,...,y_n\}$ with $y_1$ chosen arbitrarily. For $j=1,...,n$ we let $\mathfrak{p}(y_j)$ be the  probability   that a simple random walk started from $v$ first hits the boundary at a vertex in the boundary arc $\{y_1,...,y_j\}$.  The boundary vertices $y_1,...,y_n$ are then mapped in counterclockwise order to the unit circle $\partial\bbD$ such that $\psi(y_j):=e^{2\pi i \mathfrak{p}(y_j)}$. Then the hitting probability of the random walk started from $v$ approximates the uniform measure on the unit circle. We map the interior vertices of $\cM_0$ into the unit disk via the discrete harmonic extension of the boundary values, i.e., for any interior vertex $v$,
\begin{equation}
    \psi(v) = \frac{1}{\deg(v)} \sum_{e\sim v}\psi(v_e),
\end{equation}
where $v_e$ is the endpoint of $e$ other than $v$.
Let $\psi:\cV\cM_0\to\bbD$ be the discrete harmonic function with boundary data as prescribed above. For each edge $e$ of $\cM_0$ connecting $v_1$ and $v_2$, we also define $\psi(e)$ to be the straight line segment joining $\psi(v_1)$ and $\psi(v_2)$. This defines the Tutte 
embedding\footnote{ {Since $\cM_0$ potentially has multiple edges, the Tutte embedding might not be a planar embedding as defined at the beginning of Section \ref{subsec:intro-1} since $\psi$ might be non-injective. On the other hand, 
the embedding in~\eqref{eq:almost-planarity} below will be an actual planar embedding.
}} 
of the map $\cM_0$ within the unit disk $\ol\bbD$. Note that under this drawing, if $e_1,...,e_k$ are edges with the same endpoints  {(i.e., these edges form a multiple edge)}, then $\psi(e_1)=...=\psi(e_k)$, and  vertices  {enclosed by} 
the edges $e_1,...,e_k$ on the map $\cM_0$ are mapped to  points on the line segment $\psi(e_1)$. Another possible degeneracy is that $\mathfrak{p}(y_j)$ is not necessarily strictly increasing in terms of $j$ and some boundary vertices might be embedded at the same point. This, however, will not affect our proof of Proposition~\ref{prop:almost-planarity} below since we will use the Tutte embedding of a larger map and restrict to a subset of interior vertices. On the other hand, the drawing $\psi$ is planar in the following sense. Let  $e_1$ and $e_2$ be edges that are not contained within multiple edges of $\cM_0$, i.e., they are not part of a multiple edge or are enclosed by one. If $e_1$ and $e_2$ have disjoint endpoint sets, $\psi(e_1)\cap \psi(e_2)=\emptyset$,  and if $e_1$ and $e_2$ have exactly one common endpoint $v$, then $\psi(e_1)\cap \psi(e_2)=\{\psi(v)\}$. 

Next we prove that for large $k$, the embedding of $\cM([-2^k,2^k]^2)$ via $\cH$ is close to \emph{some} planar embedding of  $\cM([-2^k,2^k]^2)$. We will typically use the Tutte embedding here as the planar embedding. However, the embedded boundary of  $\cM([-2^k,2^k]^2)$ via $\cH$ only has close Hausdorff distance, not curve distance~\eqref{eq:curve-metric-0},  to $\partial [-2^k,2^k]^2$, which is not sufficient. Below we will construct a map $\cM_{k,\ell}$ which contains and is slightly larger than $\cM([-2^k,2^k]^2)$ and work on that map instead. Recall the function $\phi_0$ defined after Lemma~\ref{lem:M(S)}.

 \begin{lemma}\label{lem:almost-planarity}
     Let $\ell>0$. Almost surely, for all large enough $k$, there exists a submap $\cM_{k,\ell}\subseteq \cM$ with the following properties:
     \begin{enumerate}[(i)]
         \item $\cM_{k,\ell}$ has disk topology, and for each vertex $v$ on its boundary, its cell in $\cH$ intersects $\partial[-2^k,2^k]^2$;
         \item $\cM_{k,\ell}$ contains $\cM([-2^k,2^k]^2)$ as a submap;
         \item Let $M_{k,\ell}\subseteq M(\cH)$ be the Riemann surface generated from $\cM_{k,\ell}$. Then the curve distance~\eqref{eq:curve-metric-0} between $\phi_0(\partial M_{k,\ell})$ and $\partial [-2^k,2^k]^2$ is less than $2^{k-\ell}$.
     \end{enumerate}
 \end{lemma}

\begin{proof}
    Following Definition~\ref{def:connectedness} and Lemma~\ref{lem:no-macroscopic-cell},  let  $k$ be large enough such that for any line segment $L\subseteq [-2^{k+1},2^{k+1}]^2$ of length at least $2^{k-\ell-4}$, $\cH(L)$ is connected, and the maximal diameter of cells in $\cH([-2^{k+1},2^{k+1}]^2)$ is less than $2^{k-2\ell-10}$.  We partition $\partial [-2^{k},2^k]^2$ into line segments $L_1,...,L_{2^{\ell+3}}$ of length $2^{k-\ell}$, such that $L_1 = [-2^{k},-2^{k}+2^{k-\ell}]\times \{-2^k\}$, and $L_1,...,L_{2^{\ell+3}}$ are aligned counterclockwise. For $j=1,...,2^{\ell+3}$ with $2^{\ell+3}+1$ identified with 1, let $\mathring{L_j}$ be the line segments concentric to $L_j$ with length $|L_j|/2$, and let $L_{j,j+1}$ be the union of $\mathring{L_j},\mathring{L}_{j+1}$ and the middle component of $(L_j\cup L_{j+1})\backslash(\mathring{L_j}\cup\mathring{L}_{j+1})$.  We pick a cell $H_j\in \cH(\mathring{L_j})$. By the line connectivity property above, we can find a simple path  $\wt\gamma_{k,j}$ on $\cH(L_{j,j+1})$ connecting $H_j$ and $H_{j+1}$. The intersection between $\wt\gamma_{k,j}$ and $\wt\gamma_{k,j+1}$ are in $\cH(\mathring{L}_{j+1})$, and therefore we can erase all the loops in $\cup_j \wt\gamma_{k,j}$ that are made by vertices in $\cH(\mathring{L}_{j+1})$ for each $j$ to obtain a simple loop $\gamma_k$. Let $\cM_{k,\ell}\subseteq \cM$ be the planar map formed by $\gamma_k$ and the vertices enclosed by $\gamma_k$. Then properties (i) and (iii) are immediate from construction, while property (ii) is a consequence of Lemma~\ref{lem:phi_0-bbC-1}.
\end{proof}

We are now ready to prove  {that the embedding of $\cM([-2^k,2^k]^2)$ via $\cH$ is close to planar}.

\begin{proposition}\label{prop:almost-planarity}
 Let $\cH$ be a cell configuration in the context of Theorem~\ref{thm:CellSystemPack} or Theorem~\ref{thm:invariance-uniformization}.   
   There exists a collection  of  points $\{z_H^k:H\in \cV\cM([-2^k,2^k]^2)\}$  and simple curves   $\{\gamma_e^k:e\in \cE\cM([-2^k,2^k]^2)\}$ forming a planar embedding of $\cM([-2^k,2^k]^2)$, such that almost surely, 
   \begin{equation}\label{eq:almost-planarity}
       \frac{1}{2^k}\max_{H\in \cV\cM([-2^k,2^k]^2)}\max_e \dist(\gamma_e,H)\to0 \ \ \text{or  } \\ \   \frac{1}{2^k}\max_{H\in \cV\cM([-2^k,2^k]^2)}\max_e \dist(\ol \gamma_e,H)\to0,
   \end{equation}
   where $\ol\gamma_e {
   =\{\ol z\,:\, z\in\gamma_e \}
   }$ is the reflection of $\gamma_e$ over the real line and the second maximum is over edges with $H$ being an endpoint and the distance is the Hausdorff distance.
\end{proposition}

\begin{proof}
    By setting the conductance  to be equal to 1 for all edges, the moment bound~\eqref{eq:thm-inv-circle} implies the moment bound~\eqref{eq:GMSinvariance-moment}, so the conclusion in Theorem~\ref{thm:GMSinvariance} holds. Without loss of generality, throughout the proof we assume that the matrix $\Sigma$ there is the identity matrix; otherwise we choose a deterministic matrix $\mathbf{A}$ such that the random walk on $\mathbf{A}\cH$   converges to the standard planar Brownian motion in the sense of Theorem~\ref{thm:GMSinvariance} and work with some deterministic parallelogram instead of $[-1,1]^2$ below. 
    

Let $k,\ell>0$, the map $\cM_{k,\ell}$ be as in Lemma~\ref{lem:almost-planarity}. We will work with the embedding of $\cM_{k+2,\ell}$ and later restrict to $\cM([-2^k,2^k]^2)$. By Lemma~\ref{lem:almost-planarity}, there are cyclic-order-preserving identifications $\pi_k:\cV(\partial\cM_{k+2,\ell})\to \partial [-1,1]^2$ such that $\max_{H\in \cV(\partial\cM_{k+2,\ell})} |\pi_k(H)-2^{-k-2}\wt c(H)|\to 0$ almost surely as we send $\ell\to \infty$ sufficiently slower than $k\to\infty$. Fix a conformal map $\varphi_0:[-1,1]^2\to\bbD$ fixing 0. For $H\in \cM_{k+2,\ell}$, let $\mu_k^H$ be the distribution of the first boundary vertex visited by the simple random walk (with edge multiplicity counted) on $\cM_{k+2,\ell}$ started from $H$. Then the argument of ~\cite[Lemma 3.14]{GMSinvariance}, Theorem~\ref{thm:GMSinvariance} along with the boundary approximation above gives that, up to a reflection depending on the orientation of $\pi_k$, $$\sup_{H\in \cM([-2^k,2^k]^2)}d_{\mathrm{Prok}}\big((\varphi_0\circ\pi_k)_*\mu_k^H,\omega_{\bbD}^{2^{-k-2}\wt c(H)}\big)\to 0 \ \text{a.s.\ as $\ell\to\infty$ sufficiently slower than $k\to\infty$},$$ where $d_{\mathrm{Prok}}$ is the Prokhorov distance, $(\varphi_0\circ\pi_k)_*\mu_k^H$ is the pushforward of $\mu_k^H$ via $\varphi_0\circ\pi_k$ and $\omega_{\bbD}^{2^{-k-2}\wt c(H)}$ is the Euclidean harmonic measure of $\partial \bbD$ as seen from $\varphi_0(2^{-k-2}\wt c(H))$. This further gives the closeness of the cells in $\cM([-2^k,2^k]^2)$ between the a priori embedding $\cH$ and the Tutte embedding for $\cM_{k+2,\ell}$ or its reflection as we pull back via $\varphi_0^{-1}$ and scale by $2^{k+2}$. 
Furthermore, for large enough $k$, if we are in the former setting (i.e., the Tutte embedding for $\cM_{k+2,\ell}$ rather than its reflection is close to $ 2^{-k}\cH$ when restricted to $ \cM([-2^k,2^k]^2)$), then the clockwise loop formed by $\partial \cM([-2^k,2^k]^2)$ on $\cM$ is embedded as a clockwise loop via the cell configuration $\cH$, which further implies that we must be in the same setting  for $k+1$ in place of $k$. Similarly,  if we are in the other setting where the reflection of  the Tutte embedding for $\cM_{k+2,\ell}$ is close to the embedding induced by $ 2^{-k}\cH$, then the same also holds for $k+1$.
We can further dilate  the above Tutte embedding for $\cM_{k+2,\ell}$  to generate the points and curves  $\{\wt z_H^k\}$ and  curves $\{\wt\gamma_e\}$ where exactly one of the two relations in~\eqref{eq:almost-planarity} holds, and further perturb and separate the curves  $\wt\gamma_e$ with error $o(2^{k+2})$ for multiple edges  and  embed the components   contained within these multiple edges in an arbitrary way to get an actual planar embedding where exactly one of the two relations in~\eqref{eq:almost-planarity} holds.
     \end{proof} 

Throughout the rest of the paper, we will assume that the first relation rather than the second in~\eqref{eq:almost-planarity} holds, and prove Theorem~\ref{thm:CellSystemPack} and  Theorem~\ref{thm:invariance-uniformization} in the setting where $\det(\mathbf{A})=1$. In the case where the second relation in ~\eqref{eq:almost-planarity} holds, we work with the reflection of $\cH$ (with the same associated map $\cM$) and the matrix $\mathbf{A}$ in Theorem~\ref{thm:CellSystemPack} and  Theorem~\ref{thm:invariance-uniformization} will have determinant $-1$.

\section{Convergence of cell configurations under the circle packing}\label{sec:circle-packing}

In this section we prove Theorem~\ref{thm:CellSystemPack}. Throughout this section, we assume that $\cH$ is a random cell configuration that satisfies the constraints in Theorem~\ref{thm:CellSystemPack}. In Section~\ref{subsec:ringlemma}, we recap the Ring Lemma from~\cite{RodinSullivan87} and prove a variant (Lemma~\ref{lem:circle}), which roughly states that in circle packings, a disk cannot simultaneously be tangent to two large disks  {unless one of the two large disks has high degree}. In Section~\ref{subsec:Covergence-RW-cell-Dubejko}, we use Lemma~\ref{lem:circle} and the invariance principle from~\cite{GMSinvariance} to prove that, the random walk on $\cH$ with Dubejko weights converges to Brownian motion. In Section~\ref{subsec:Covergence-RW-CP-Dubejko}, we prove that the random walk on the circle packing $\cP$ for $\cH$ with Dubejko weights converges to Brownian motion. Both convergence results hold in the quenched sense, which allow us to conclude that $\cH$ is close to $\cP$ in Section~\ref{subsec:pf-Cell-system-Pack}. 

\subsection{The Ring Lemma and a three-circle variant}\label{subsec:ringlemma}

One fundamental result in circle packing is the Ring Lemma~\cite{RodinSullivan87}. Let $(D_0;D_1,...,D_d)$ be disks with disjoint interiors. We say $(D_0;D_1,...,D_d)$ form a flower with $d$ petals if $D_1,...,D_d$ are all tangent to $D_0$ and separate $D_0$ from infinity, and for $j=1,...,d$, $D_j$ and $D_{j+1(\mathrm{mod}\ d)}$ are tangent. The following is the well known Ring Lemma by Rodin and Sullivan~\cite{RodinSullivan87}.
\begin{lemma}[Ring Lemma]\label{lem:Ring-Lemma}
    Suppose $d\geq 3$ and $(D_0;D_1,...,D_d)$ form a flower with $d$ petals. Let $r_0,r_1,...,r_d$ be the radii of $D_0,D_1,...,D_d$. Then there exists some absolute constant $c>0$ such that for each $j=1,...,d$,  $r_j/r_0\geq \exp(-cd)$. 
\end{lemma}

We will also use the following result known as Descartes' theorem.
\begin{lemma}\label{thm:Descartes}
    Let $D_1,D_2,D_3$ be mutually tangent disks with disjoint interiors and radii $r_1,r_2,r_3$. There exist exactly two circles that are tangent to all of $D_1,D_2,D_3$, and the radius $r$ {of these two circles} is determined by (a straight line is understood as a circle with radius equal to infinity)
    \begin{equation}\label{eq:Descartes}
    \frac{1}{r} =\Big| \frac{1}{r_1}+\frac{1}{r_2}+\frac{1}{r_3}\pm2\sqrt{\frac{1}{r_1r_2}+\frac{1}{r_2r_3}+\frac{1}{r_1r_3}}\Big|.
    \end{equation}
\end{lemma}
If we take the plus  {(resp., minus)} sign, then the circle is contained within the bounded (resp., unbounded) component of $\bbC\backslash (D_1\cup D_2\cup D_3)$.

The bound in Lemma~\ref{lem:Ring-Lemma} is sharp. Let $D_0 = \bbD$, $D_1 = \{\mathrm{Im}\ z=1\}$, $D_2 = \{\mathrm{Im}\ z=-1\}$ (where we view parallel lines as circles tangent at infinity), $D_3 = B(2;1)$, and we recursively define $D_{n+2}$ by the disk which is (externally) tangent to both $D_0,D_n$ and $D_{n+1}$. Then one can check by Descartes' theorem (see~\eqref{eq:Descartes} above) that the ratio $r_d/r_0$ is equal to $\frac{1}{F_{2d-3}-1}$, where  $F_1=1$, $F_2=1$, and $F_n$ is the $n$-th Fibonacci number. It is shown in~\cite{Hansen88} that the bound described as above is optimal.

We will use the following variant of the Ring Lemma in Section~\ref{subsec:Covergence-RW-cell-Dubejko}, which states that in the flower $(D_0;D_1,...,D_d)$, $D_2$ cannot have too small radii compared to \emph{both} $D_0$ and $D_1$. This will eventually allow us to verify only polynomial bounds of degree (rather than exponential bounds) when applying  {(our variant of)} the invariance principle in~\cite{GMSinvariance}. 

\begin{lemma}\label{lem:circle}
    For any flower $(D_0;D_1,...,D_d)$ with $d$ petals and radii $(r_0;r_1,...,r_d)$, \eqb\label{eq:lem-circle}r_2\geq 0.01d^{-2}\min\{r_0,r_1\}.\eqe 
\end{lemma}

\begin{figure}
    \centering
    \begin{tabular}{cc}
      \includegraphics[scale=0.45]{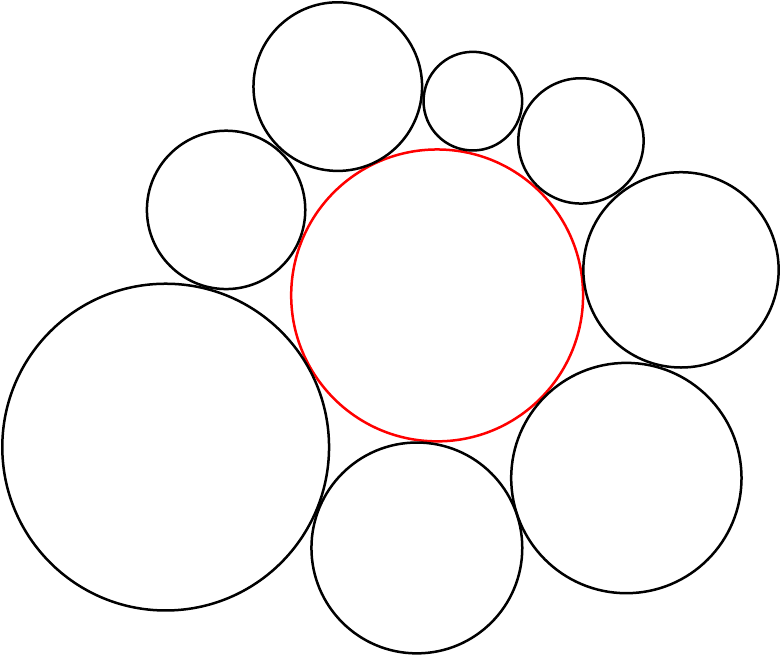}  &  \includegraphics[scale=0.45]{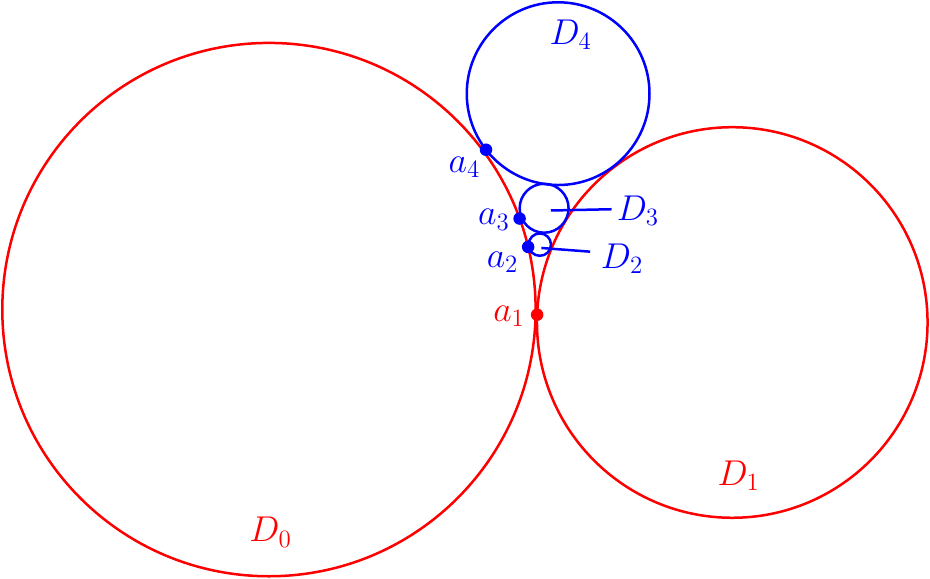}
    \end{tabular}
    \caption{\textbf{Left}: A flower with 8 petals. \textbf{Right}: An illustration for the setup in the proof of Lemma~\ref{lem:circle}  {for $d=4$}.}
    \label{fig:ringA}
\end{figure}


\begin{proof}
 Assume on the contrary that $r_2 \leq 0.01d^{-2}\min\{r_0,r_1\}$. 
     We start with the case where $D_2,...,D_d$ are all tangent to $D_1$. Suppose $D_j$ and $D_0$ are tangent at the point $a_j$, and $D_1,...,D_d$ are ordered counterclockwise. See Figure~\ref{fig:ringA} for an illustration. Using $r_2 \leq 0.01d^{-2}\min\{r_0,r_1\}$, one can easily check that the counterclockwise arc from $a_1$ to $a_2$ has length at most $\pi r_0/6$.  Then for $j=2,...,d$, since $D_{j+1}$ is tangent to $D_{j}, D_0, D_1$ while  $D_{j}, D_0, D_1$ are tangent to each other, and $D_{j+1}$ is not in the bounded component of   $\bbC\backslash(D_{j+1}\cup D_0\cup D_1)$, using  Descartes' Theorem, it follows by~\eqref{eq:Descartes} 
    \begin{equation}\label{eq:lm-descartes}
        \frac{1}{r_{j+1}} = \Big|\frac{1}{r_{j}} + \frac{1}{r_0} +\frac{1}{r_1} - 2\sqrt{(\frac{1}{r_0}+\frac{1}{r_1})\frac{1}{r_{j}}+\frac{1}{r_0r_1}}\Big|.
    \end{equation}
    Let $j_0$ be the smallest $j\geq 2$ such that $r_{j}>\min\{r_0,r_1\}/(100d)$. If no such $j_0$ exists, then we set $j_0=d$. Since  $r_2 \leq 0.01d^{-2}\min\{r_0,r_1\}$, $3\leq j_0\leq d$. 
    On the other hand, for $j=2,...,j_0-1$, from~\eqref{eq:lm-descartes}, 
    \begin{equation}
        \begin{split}
             \frac{1}{r_{j}} + \frac{1}{r_0} +\frac{1}{r_1} - 2\sqrt{2(\frac{1}{r_0}+\frac{1}{r_1})\frac{1}{r_{j}}}\geq  \frac{1}{r_{j}} + 4\big(\frac{1}{r_0} +\frac{1}{r_1}\big) - 4\sqrt{(\frac{1}{r_0}+\frac{1}{r_1})\frac{1}{r_{j}}} = \bigg(\sqrt{\frac{1}{r_j}}-2\sqrt{\frac{1}{r_0}+\frac{1}{r_1}}\bigg)^2,
        \end{split}
    \end{equation}
    where we have used $r_{j}\leq\min\{r_0,r_1\}/(100d)$. In particular, \eqb\label{eq:pf-circle-lemma-4}\sqrt{\frac{1}{r_{j+1}}}\geq \sqrt{\frac{1}{r_{j}}} - 2\sqrt{\frac{1}{r_0}+\frac{1}{r_1}},\ \ \text{and } \sqrt{\frac{1}{r_{j_0}}}\geq \sqrt{\frac{1}{r_2}} - 2d\sqrt{\frac{1}{r_0}+\frac{1}{r_1}}\geq 6d\min\{r_0,r_1\}^{-1/2},\eqe
    If $j_0<d$, then this implies $r_{j_0}\leq \min\{r_0,r_1\}/(36d^2)$, which contradicts the definition of $j_0$ since $d\geq 3$. Therefore $j_0=d$ and for each $j=2,...,d$, $r_j\leq \min\{r_0,r_1\}/(36d^2)$,  and $\dist(a_2,a_d)\leq 2(r_2+...+r_d)\leq r_0/(18d)$. In particular, the counterclockwise arc from $a_2$ to $a_d$ has length at most $\pi r_0/3$, and the counterclockwise arc from $a_1$ to $a_d$ is at most $\pi r_0/2$,  $D_1\cup D_2\cup ...\cup D_d$ cannot separate $D_0$ from infinity and $(D_0;D_1,...,D_d)$ is not a flower.

    For the general case, for $j=1,...,d$, suppose $D_0$ and $D_j$ are tangent at the point $a_j$.  We perform a M\"{o}bius transformation $\varphi$ on $\bbC$ sending $a_1$ to infinity. Then $\partial D_0, \partial D_1$ are mapped to parallel straight lines $\ell_0,\ell_1$, where the images  $\varphi(D_2),...,\varphi(D_d)$ are all disks tangent to $\ell_0$. Since $\varphi(D_2),...,\varphi(D_d)$ all have at most one intersection point with $\ell_1$, the length of the line segment from $\varphi(a_2)$ to $\varphi(a_d)$ is maximized only when  $\varphi(D_2),...,\varphi(D_d)$ are all tangent to $\ell_1$ as well. By transforming back and comparing with the previous computation, we see that  $D_1\cup D_2\cup ...\cup D_d$ cannot separate $D_0$ from infinity.  This is again a contradiction 
    and we conclude the proof.
\end{proof}

We end this section with the following corollary of the Ring Lemma. For a simple infinite triangulation $\cT$ with root vertex  $v_0$, let $B_m^{\cT}(v_0)$ be the ball of radius $m$ centered at $v_0$ in the sense of graph distance. Let $\cP$ be some circle packing of  $\cT$, where the disk corresponding to the root vertex $v_0$ is the unit disk. Also let $\cR^m_\cT$ and $\cX^m_\cT$ be the vectors of radii and centers{, respectively,} of the disks in $\cP$ corresponding to the vertices in $B_m^{\cT}(v_0)$.
\begin{proposition}\label{Prop:circle-conv}
    For simple  infinite triangulations $\cT$ and $(\cT^n)_{n\geq 1}$ with a common root vertex $v_0$, suppose that for any fixed $m\geq1$, there exists $N\geq1$ such that for any $n\geq N$, $B_m^{\cT^n}(v_0)$  agrees with $B_m^{\cT}(v_0)$ as planar maps rooted at $v_0$. Further assume that $\cT$ and $(\cT^n)_{n\geq 1}$ are all circle packing parabolic. Then for any fixed $k\geq1$, one has the limit $$\lim_{n\to\infty} \cR^k_{\cT^n} = \cR^k_\cT.$$
\end{proposition}

\begin{proof}
    Following the Ring Lemma, for any fixed $m$, $\{(\cR^m_{\cT^n},\cX^m_{\cT^n})\}_{n\geq1}$ is a tight sequence. Assume on the contrary that there is a subsequence $\{s_n\}$ such that $\lim_{n\to\infty} \cR^k_{\cT^{s_n}} \neq \cR^k_\cT.$ Then using the argument of taking diagonal subsequences, one can find a further subsequence of $\{s_n\}$  {(still denoted by $\{s_n\}$ in order to simplify notation) such} 
    that the limits of  $\{(\cR^m_{\cT^{s_n}},\cX^m_{\cT^{s_n}})\}_{n\geq1}$ exist for every $m\geq1$. Moreover, since for every $m$, $B_m^{\cT^n}(v_0)$  agrees with $B_m^{\cT}(v_0)$ for large enough $n$, it is clear that this limit satisfies the circle packing conditions for $\cT$. In other words, we have constructed a circle packing for $\cT$ with the disk corresponding to $v_0$ being the unit disk and a different vector of radii for vertices in $B_k^{\cT}(v_0)$. This contradicts the uniqueness of the circle packing for $\cT$ up to M\"{o}bius transformations.
\end{proof}

\subsection{Convergence of random walk  with Dubejko weights on cell configurations}\label{subsec:Covergence-RW-cell-Dubejko}

Recall the random walk on circle packing and the Dubejko weights as defined in Section~\ref{pre:triangulation}. In this section, we prove that, for cell configurations $\cH$ in the context of Theorem~\ref{thm:CellSystemPack}, if we assign the conductance to be equal to the Dubejko weights, then the random walk on $\cH$ with the Dubejko weights converges to Brownian motion modulo time parameterization.

We start with the following immediate consequence of Theorem~\ref{thm:GMSinvariance} and Theorem~\ref{thm:circle-packing-KAT}.

\begin{corollary}\label{cor:parabolic}
    Let $\cH$ be the cell configuration as in Theorem~\ref{thm:CellSystemPack}. Then $\cH$ is a.s.\ circle packing parabolic.
\end{corollary}
\begin{proof}
    The condition~\eqref{eq:thm-inv-circle} implies the condition~\eqref{eq:GMSinvariance-moment} for cell configurations with unit conductance, and therefore by Theorem~\ref{thm:GMSinvariance}, the simple random walk on $\cH$ is recurrent. Therefore combined with Theorem~\ref{thm:circle-packing-KAT}, we immediately infer that $\cH$ is a.s.\ circle packing parabolic.
\end{proof}

The rest of the section is devoted to the proof of the following result. 
\begin{proposition}\label{prop:inv-circle}
    Let $\cH$ be a cell configuration as in Theorem~\ref{thm:CellSystemPack}, except that the conductance is given by the Dubejko weights as in~\eqref{eq:Dubejko}. Then the convergence of random walk on $\cH$ with Dubejko weights to Brownian motion described as in Theorem~\ref{thm:GMSinvariance} holds.  If $\cH$ is rotation invariant in law for some $\theta_0\in(0,2\pi)\backslash\{\pi\}$ as in Definition~\ref{def:rotation-inv}, the matrix $\Sigma$ is a scalar times the identity matrix. Furthermore, there exists $\phi_\infty:\cH\to\bbC$ that is harmonic with respect to the conductance $\mathfrak{c}$ such that $\frac{1}{r}\sup_{H\in\cH(B(0;r))}|c(H)-\phi_\infty(H)|\to0$ almost surely as $r\to\infty$.
\end{proposition}

\begin{proof}
To begin with, we check the conditions in Theorem~\ref{thm:GMSinvariance}. For $\cH$ with the conductance  given by the Dubejko weights, the translation invariance modulo scaling property in Definition~\ref{def:translation-modulo-scaling} still holds by applying the same condition for $\cH$ with unit conductance together with the convergence of Dubejko weights as in Proposition~\ref{Prop:circle-conv}. The ergodicity modulo scaling and line connectivity from Definition~\ref{def:ergodic-modulo-scaling} and Definition~\ref{def:connectedness} indeed carry, and the first equation in the moment bound~\eqref{eq:GMSinvariance-moment} holds since the Dubejko conductance is always less than or equal to 1. All we need to show is that, the conclusion of Theorem~\ref{thm:GMSinvariance} is still true with the second equation in the condition~\eqref{eq:GMSinvariance-moment} replaced by~\eqref{eq:thm-inv-circle}, since if we have the rotational invariance in law for some $\theta_0\in(0,2\pi)\backslash\{\pi\}$, by Corollary~\ref{cor:GMS-rotation-inv} the covariance matrix $\Sigma$ must be a scalar times the identity matrix.

    {The only place in~\cite{GMSinvariance} where the second equation in the condition~\eqref{eq:GMSinvariance-moment} is used is~\cite[Lemma 2.17]{GMSinvariance}. We will first prove a slightly weaker version of this lemma (with condition~\eqref{eq:GMSinvariance-moment} replaced by~\eqref{eq:thm-inv-circle}) and then we check that the rest of the proof of~\cite[Theorem 3.10]{GMSinvariance} still carries through with the weaker version of the lemma.} Throughout the proof, $c$ stands for some deterministic constant which may vary line by line.

    Fix $k\in\mathbb{N}$, $r\in[0,1]$. Let $\cD$ be the uniform dyadic system independent of $\cH$ defined as in Definition~\ref{def:dyadicsystem}, $S_k$ be the square in $\cD$ containing 0 with side length in $[2^{k-1},2^{k})$, and let $P_k(r)$ be the horizontal line segment connecting the left and right boundaries of the square $S_k$ whose distance from the lower boundary of $S_k$ is $r|S_k|$.    We prove that for each large enough $k\in\bbN$ and each function $f:\cH(\wt S_k)\to\bbC$, 
    \begin{equation}\label{eq:pf-inv-circle}
        \int_0^1 \sum_{\{H,H'\}\in\cE\cH(P_k(r))} |f(H)-f(H')|dr\leq c\Big(\sum_{\{H,H'\}\in \cH(\wt S_k)}\fc(H,H')|f(H)-f(H')|^2\Big)^{1/2},
    \end{equation} 
    where $\wt S_k$ is the union of the four dyadic parents of $S_k$. Note that this is identical to the statement of~\cite[Lemma 2.17]{GMSinvariance} except that we have $\wt S_k$ instead of $S_k$ on the right side.

    {Before proving \eqref{eq:pf-inv-circle} let us explain why this estimate is sufficient to conclude the proof of the theorem. The only place in the proof of~\cite[Theorem 1.16]{GMSinvariance} where~\cite[Lemma 2.17]{GMSinvariance} is used is in the proof of~\cite[Equation (2.44)]{GMSinvariance}. The latter equation is proven exactly as before.}

     For $\{H,H'\}\in\cE\cH(S_k)$, let $L_k(H,H')$ be the Lebesgue measure of the set of $r\in[0,1]$ for which $H$ and $H'$ both belong to $\cH(P_k(r))$. Then{, as argued in~\cite[Lemma 2.17]{GMSinvariance},} 
    \begin{equation}\label{eq:pf-inv-circle-1}
        \int_0^1 \sum_{\{H,H'\}\in\cE\cH(P_k(r))} |f(H)-f(H')|dr = \sum_{\{H,H'\}\in \cH(S_k)} L_k(H,H')|f(H)-f(H')|.
    \end{equation}
    and \eqb\label{eq:inv-circle-Lk} L_k(H,H')\leq\frac{2}{|S_k|}\min\{\diam(H),\diam(H')\}.\eqe
    By Lemma~\ref{lem:average-diamdeg4}, there is a deterministic constant $c_0>0$ such that a.s.\ for large enough $k\in\bbN$, 
    \begin{equation}\label{eq:pf-circle-inv-52}
        \frac{1}{|S_k|^2}\sum_{H\in\cH(S_k)} \diam(H)^2\deg(H)^{4}\leq c_0.
    \end{equation}
    \begin{figure}
        \centering
        \includegraphics[scale=0.4]{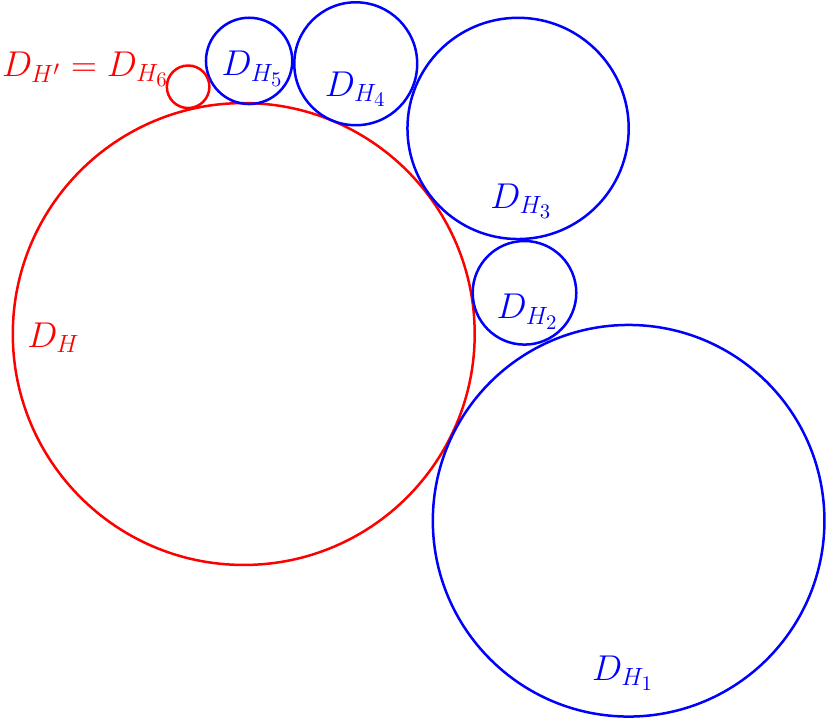}
        \caption{Construction of the path $P_{H,H'}$  {in the proof of Proposition \ref{prop:inv-circle}} when $r_{H'}\leq 0.01d^{-1}r_{H}$. Here for $j=2,...,6$, $r_{H_j}<0.01d^{-1}r_H$ while $0.01d^{-1}r_H\leq r_{H_1}<r_H$.}
        \label{fig:ringB}
    \end{figure}
    Now for two cells $H\sim H'$, we define as follows a set of edges $P_{H,H'}$ forming a path from $H$ to $H'$. We fix a circle packing $\cP$ for $\cH$. For $H\in \cH$, we write $D_H$ for the corresponding circle in the packing and $r_H$ for the radius of $D_H$.  
    We assume that $r_H\geq r_{H'}$; the other case is dealt with by setting $P_{H,H'}=P_{H',H}$. 
    We write $d = \deg(H)$. 
    If $r_{H'}\geq 0.01d^{-1}r_{H}$, then we define $P_{H,H'}$ to be the single edge connecting $H$ and $H'$. Then following Lemma~\ref{lem:circle}, for any $H''$ adjacent to both $H$ and $H'$, one must have $r_{H''}\geq10^{-4}d^{-3}r_H$, and therefore by the Dubejko weight formula~\eqref{eq:Dubejko} $\fc(H,H')\geq10^{-4}d^{-2}$. If $r_{H'}< 0.01d^{-1}r_{H}$, we define the set $P_{H,H'}$ as follows. First observe that if we draw straight lines connecting the centers of disks which are tangent to each other and $D_H$, then we form a  {closed curve} around $D_H$ with length $2\sum_{H''\sim H}r_{H''}\geq 2\pi r_H$, and there must exist some $H''\sim H$ such that $r_{H''}\geq \pi r_H/d> 0.01d^{-1}r_H$. Let $D_{H_1}$ be the first such disk when viewed clockwise from $D_{H'}$, and we let $D_{H_1}$, $D_{H_2}$,..., $D_{H_n}$ be the disks tangent to $D_H$ and ordered counterclockwise where we write $H_n = H'$. If $r_{H_1}\geq r_H$, define $P_{H,H'}=\{ \{H,H_2\},\{H_2,H_3\},...,\{H_{n-1},H_n\}\}$; otherwise set   $P_{H,H'}=\{ \{H,H_1\},\{H_1,H_2\},...,\{H_{n-1},H_n\}\}$.   See also Figure~\ref{fig:ringB} for an illustration. Then by definition, $r_{H_2}, ..., r_{H_n}<0.01d^{-1}r_H$ and $r_{H_1}\geq 0.01d^{-1}r_H$. By Lemma~\ref{lem:circle}, $r_{H_2}> 0.01d^{-2}\min\{r_{H},r_{H_1}\}$. If $r_{H_1}\geq r_H$, then $r_{H_2}\geq 0.01d^{-2}r_H$, and by the Dubejko weight formula~\eqref{eq:Dubejko}, $\fc(H,H_2)\geq 10^{-4}d^{-1}$. If $r_{H_1}<r_H$, then using $r_{H_1}\geq 0.01d^{-1}r_H$, we have $\fc(H,H_1)\geq 10^{-4}d^{-2}$. In this case, $0.01d^{-2}r_{H_1}\leq r_{H_2}\leq r_{H_1}<r_H$, and $\fc(H_1,H_2)\geq 10^{-4}d^{-1}$.  
    For $j=2,...,n-1$,  $r_{H_j}<r_H$, and by Lemma~\ref{lem:circle}, we have $0.01d^{-2} < r_{H_{j+1}}/r_{H_{j}} <100d^2$. Then again using the formula~\eqref{eq:Dubejko}, one has $\fc(H_j,H_{j+1})>10^{-4}d^{-1}$. In other words, for each edge $e\in P_{H,H'}$, if $H$ is an endpoint of $e$, then $\fc(e)\geq 10^{-4}d^{-2}$, and if $H$ is not an endpoint of $e$, then $\fc(e)\geq 10^{-4}d^{-1}$. In particular, for any $\{H,H'\}\in\cE\cH$ with $r_H\geq r_{H'}$, using $n\leq \deg(H) = d$, we have
    \begin{equation}\label{eq:pf-inv-circle-29}
        \sum_{e\in P_{H,H'}}\fc(e)^{-1}\leq 2\cdot10^{4}\deg(H)^{2}.
    \end{equation}
  
   For an edge $e = \{H,H'\}$, we set $|\nabla f(e)| = |f(H')-f(H)|$. Then from the triangle inequality,  $|f(H)-f(H')|\leq \sum_{e\in P_{H,H'}}|\nabla f(e)|$. Now by~\eqref{eq:pf-inv-circle-1} and~\eqref{eq:inv-circle-Lk}, 
    \begin{equation}\label{eq:pf-inv-circle-4}
        \begin{split}
             \int_0^1 &\sum_{\{H,H'\}\in\cE\cH(P_k(r))} |f(H)-f(H')|dr \leq \frac{2}{|S_k|}\sum_{\{H,H'\}\in \cH(S_k)}\min\{\diam(H),\diam(H')\}|f(H)-f(H')|\\
             & \leq \frac{4}{|S_k|}\sum_{\{H,H'\}\in \cH(S_k)}\sum_{e\in P_{H,H'}}\mathds{1}_{r_H\geq r_{H'}}\min\{\diam(H),\diam(H')\}|\nabla f(e)|\\
             & \leq \frac{4}{|S_k|}\Big(\sum_{\{H,H'\}\in \cH(S_k)}\sum_{e\in P_{H,H'}}\mathds{1}_{r_H\geq r_{H'}}|\nabla f(e)|^2\fc(e)\frac{1}{\deg(H)}  \Big)^{1/2}\\&\times\Big(\sum_{\{H,H'\}\in \cH(S_k)}\sum_{e\in P_{H,H'}}\mathds{1}_{r_H\geq r_{H'}} \diam(H)^2\deg(H)\fc(e)^{-1} \Big)^{1/2}. 
        \end{split}
    \end{equation}
    
    {We will now conclude the proof of~\eqref{eq:pf-inv-circle} by bounding the right side of \eqref{eq:pf-inv-circle-4}, starting with the first factor.} For each edge $e$, there exist at most two different $H$ 
    such that $e\in P_{H,H'}$ for some $H'$ with $r_H\geq r_{H'}$ and $H$ is not an endpoint of $e$; indeed, by our definition of $P_{H,H'}$ the disks of two end-points of $e$ and $D_H$ must all pairwise touch and form a triangular face incident to the edge $e$.  There are also at most two $H$ such that $e\in P_{H,H'}$ for some $H'$ with $r_H\geq r_{H'}$ and $H$ is an endpoint of $e$. In either case, the corresponding number of $H'$ is at most $\deg(H)$. Moreover, it follows from Lemma~\ref{lem:no-macroscopic-cell} that for large enough $k$,  {if $\{H,H'\}\in \cH(S_k)$ and $H''$ is an end-point of an edge in $P_{H,H'}$ then $H''\in \cH(\wt S_k)$}. 
    Therefore
    \begin{equation}\label{eq:pf-inv-circle-4-1}
        \sum_{\{H,H'\}\in \cH(S_k)}\sum_{e\in P_{H,H'}}\mathds{1}_{R_H\geq R_{H'}}|\nabla f(e)|^2\fc(e)\frac{1}{\deg(H)} \leq 4\sum_{e\in\cE\cH(\wt S_k)}|\nabla f(e)|^2\fc(e).
    \end{equation}
    On the other hand, by~\eqref{eq:pf-inv-circle-29}, we have 
    \begin{equation}\label{eq:pf-inv-circle-4-2}
    \begin{split}
        \sum_{\{H,H'\}\in \cH(S_k)}\sum_{e\in P_{H,H'}}&\mathds{1}_{R_H\geq R_{H'}} \diam(H)^2\deg(H)\fc(e)^{-1}\\& \leq 2\cdot 10^4  \sum_{\{H,H'\}\in \cH(S_k)}  \diam(H)^2\deg(H)^3\leq 2\cdot 10^4\sum_{H\in\cH(S_k)}\diam(H)^2\deg(H)^{4}.
        \end{split}
    \end{equation}
    Thus~\eqref{eq:pf-inv-circle} follows by plugging~\eqref{eq:pf-circle-inv-52}, \eqref{eq:pf-inv-circle-4-1}, \eqref{eq:pf-inv-circle-4-2} into~\eqref{eq:pf-inv-circle-4}. 
\end{proof}

\subsection{Convergence of random walk  with Dubejko weights on circle packing}\label{subsec:Covergence-RW-CP-Dubejko}


Consider a circle packing $\cP$ for $\cH$ such that the disk for the cell $H_0$ is the unit disk, where we recall that $\cH$ is circle packing parabolic by Corollary~\ref{cor:parabolic}. For $z\in\bbC$, consider the random walk $(\wh X_j^{z,\rm circle})_{j\in\bbN_0}$ moving along centers of disks in $\cP$ with Dubejko weights, starting from the disk with closest distance to $z$. We extend $\wh X_j^{z,\rm circle}$ from $\bbN_0$ to $[0,\infty)$ by piecewise linear interpolation. The aim of this section is to prove the following.
\begin{proposition}\label{prop:random-walk-on-packing}
    For each compact set $A\subseteq\bbC$, it is a.s.\ the case that as $\e\to0$, the maximum over all $z\in A$ of the Prokhorov distance between the conditional law of $\e\wh X^{z/\e, \mathrm{circle}}$ given $\cP$ and the law of standard Brownian motion started from $z$, with respect to the topology on curves induced in~\eqref{eq:curve-metric}, tends to 0.
\end{proposition}

One critical input for the proof of Proposition~\ref{prop:random-walk-on-packing} is the following, which is an immediate consequence of~\cite[Theorem A]{bougwynne2024random}.
\begin{proposition}[\cite{bougwynne2024random}]\label{prop:BouGwynne}
    Let $(\cP^n)_{n\geq 1}$ be a sequence of circle packings for simple circle packing parabolic infinite triangulations $(\cT^n)_{n\geq1}$. Suppose for each $r>0$, the maximal diameter of disks in $\cP^n$ intersecting $B(0;r)$ as well as the maximal diameter of connected components of $B(0;r)$ minus the (closed) disks of $\cP^n$ converges to 0. Let $\wh X^{z,n,\rm circle}$ be the linearly interpolated random walk on $\cP^n$ defined as above. Then  for each compact set $A\subseteq\bbC$, it is a.s.\ the case that as $n\to\infty$, the maximum over all $z\in A$ of the Prokhorov distance between the  law of $\wh X^{z,n, \mathrm{circle}}$ and the law of standard Brownian motion started from $z$, with respect to the topology on curves induced in~\eqref{eq:curve-metric}, tends to 0.
\end{proposition}

Proposition~\ref{prop:random-walk-on-packing} is now immediate from Proposition~\ref{prop:BouGwynne} and the following.
\begin{proposition}\label{prop:max-diam-0-cone}
    Let $r>0$ be fixed. Then the maximal diameter of the circles in $\e\cP$ intersecting $B(0;r)$ converges to 0 almost surely as $\e\to0$.
\end{proposition}

The rest of this subsection is devoted to the proof of Proposition~\ref{prop:max-diam-0-cone}. In Section~\ref{subsec:VEL}, we prove Lemma~\ref{lem:VEL-packing}, which relate the maximal diameter of the disks in circle packings with the vertex extremal length of graphs as in~\cite{GDN19VEL}. In Section~\ref{subsec:pf-prop:max-0-diam-cone} we prove Proposition~\ref{prop:max-diam-0-cone} using Lemma~\ref{lem:VEL-packing} and Lemma~\ref{lem:average-diamdeg4}.  We refer to Remark \ref{rmk-max-diam} for an alternative proof strategy for Proposition~\ref{prop:max-diam-0-cone}.

\subsubsection{Vertex extremal length and macroscopic circles}\label{subsec:VEL}

Consider a planar map $\cM$ with vertex set $\cV\cM$. Given a function $m:\mathcal{VG}\to [0,\infty)$ and a path $\ell$, define the length $\mathrm{len}_m(\ell) = \sum_{v\in \ell}m(v)$. For a family $\Gamma$ of paths, define $\mathrm{len}_m(\Gamma) = \inf_{\ell\in\Gamma} \mathrm{len}_m(\ell)$. The vertex extremal length of $\Gamma$, as introduced by Cannon~\cite{CannonVEL}, is defined as $$\VEL(\Gamma) = \sup_{m:\mathcal{VG}\to [0,\infty)}\frac{\mathrm{len}_m(\Gamma)^2}{\area(m)},$$
where $\area(m) = \sum_{v\in \cV\cM}m(v)^2$.

As studied in~\cite{GDN19VEL}, vertex extremal length has close relation with macroscopic circles in the circle packing. 
We will need the following result.
\begin{lemma}\label{lem:VEL-packing}
    Consider a simple infinite triangulation $\cT$. Let $v_0$ be a vertex. Consider a circle packing $\cP$ of $\cT$ with the circle representing the vertex $v_0$ centered at origin. For any vertex $v$ different from $v_0$, let $\Gamma_{v,v_0}$ be the collection of finite paths of $\cT$ starting and ending at $v$ which makes a simple loop around $v_0$. Let $D_v$ be the circle representing $v$, $r_v$ be the radius of $D_v$, and $z_v$ be the center of $D_v$. Then
    $$\VEL(\Gamma_{v,v_0})\leq 4|z_v|r_v^{-1}.$$
\end{lemma}

\begin{proof}
    The proof is identical to that of~\cite[Theorem 1.1,  equation (1)]{GDN19VEL},  {which proves a similar result for circle packings in a disk.} We include  {the proof} here for completeness. For $|z_v|-r_v\leq t\leq |z_v|+r_v$, let $\gamma_t$ be the loop in $\cT$ formed by the vertices whose disk in $\cP$ intersects $\{|z|=t\}$. Indeed $\gamma_t\in \Gamma_{v,v_0}$. Now for any $m:\mathcal{VG}\to [0,\infty)$, 
    \begin{equation}\label{eq:pf:VEL-3}
    \begin{split}
        \int_{|z_v|-r_v}^{|z_v|+r_v}\mathrm{len}_m(\gamma_t) dt&=\sum_{w\in\cV\cT}\int_{|z_v|-r_v}^{|z_v|+r_v}\mathds{1}_{w\in\gamma_t} m(w)\,dt = \sum_{w\in\cV\cT} m(w)\int_{|z_v|-r_v}^{|z_v|+r_v}\mathds{1}_{w\in\gamma_t}\,dt.
        \end{split}
    \end{equation}
    For each $w\in \cV\cT$, let $z_w$ and $r_w$ be the center and radius of $D_w$, and $\theta_w = \arg z_w$. Let $\underline{a_w} = \max\{|z_w|-r_w, |z_v|-r_v\}$, and $\ol a_w=\min\{|z_w|+r_w, |z_v|+r_v\}$. Then \eqb\label{eq:pf:VEL-4}\int_{|z_v|-r_v}^{|z_v|+r_v}\mathds{1}_{w\in\gamma_t}\,dt = \max\{\ol a_w - \underline{a_w},0\}.\eqe
    For each $w$ with $\ol a_w - \underline{a_2}>0$, we draw the disk $D_w'$ centered at $(\ol a_w + \underline{a_w})e^{i\theta_w}/2$ with radius $(\ol a_w - \underline{a_w})/2$. Then the disks $D_w'$ are contained within $\{|z_v|-r_v\leq |z|\leq |z_v|+r_v\}\cap D_w$ and hence have disjoint interiors. See also~\cite[Figure 3]{GDN19VEL} for an illustration. In particular, 
    \eqb\label{eq:pf:VEL-5}\sum_{w\in\cV\cT}\pi\big(\max\{\ol a_w - \underline{a_w},0\}\big)^2/4\leq \pi\big((|z_v|+r_v)^2-(|z_v|-r_v)^2\big) = 4\pi|z_v|r_v.\eqe
    Now combining~\eqref{eq:pf:VEL-3}, ~\eqref{eq:pf:VEL-4} and~\eqref{eq:pf:VEL-5}, by Cauchy-Schwarz inequality,
    \begin{equation}
        \begin{split}
            \int_{|z_v|-r_v}^{|z_v|+r_v}\mathrm{len}_m(\gamma_t) dt\leq \big(\sum_{w\in\cV\cT}\big(\max\{\ol a_w - \underline{a_w},0\}\big)^2\big)^{1/2}\big(\sum_{w\in\cV\cT} m(w)^2)^{1/2}\leq (16|z_v|r_v)^{1/2}\big(\sum_{w\in\cV\cT} m(w)^2)^{1/2}.
        \end{split}
    \end{equation}
    In particular, there exists $t\in(z_v-r_v,z_v+r_v)$, such that $\mathrm{len}_m(\gamma_t)\leq 2|z_v|^{1/2}r_v^{-1/2}\big(\sum_{w\in\cV\cT} m(w)^2)^{1/2}$ and $\mathrm{len}_m(\gamma_t)^2/\area(m)\leq 4|z_v|r_v^{-1}$. Since $m$ is arbitrary, this verifies the claim.
\end{proof}

\subsubsection{Proof of Proposition~\ref{prop:max-diam-0-cone}}\label{subsec:pf-prop:max-0-diam-cone}

Recall the notion of vertex extremal length introduced in Section~\ref{subsec:VEL}. 
For each $H\in\cH$, let $\Gamma_H$ be the collection of finite paths on $\cM$ starting from $H$ and making a simple loop around ${H_0}$.
\begin{proposition}\label{prop:finite-VEL}
    Let $q_0>0$. Then almost surely, there are only finitely many $H\in\cH$, such that $\mathrm{VEL}(\Gamma_H)<q_0$.
\end{proposition}

\begin{proof} 
   Without loss of generality assume $q_0>100$ and  set $q = 10^3q_0$. We will adapt the proof of~\cite[Theorem 1.4]{GDN19VEL} together with the input from Section~\ref{subsec:ergodic-cell-system}.   
    For $c>0$ and any square $S$, define the event $E_1(c,S)$ where \eqb\label{eq:pf-VELGMS-1}  \sum_{H\in\cH(S)} \diam(H)^2\leq c |S|^2.\eqe 
Let $a\in(0,1)$, and set $E_{k,1}(c,a)$ be the event where $E_1(c,S)$ holds  for each square $S$ in $\cD$ with side length in $(a|S_k|,|S_k|]$ which intersects $[-3|S_k|,3|S_k|]^2$.
For $\e\in(0,1)$ and any square $S$, define the event $E_{k,2}(\e):=\{\diam(H)\leq \e|S_k|, \ \forall H\in \cH([-3|S_k|,3|S_k|]^2)\}$. Then by Lemmas~\ref{lem:dyadic-shift}, ~\ref{lem:no-macroscopic-cell} and~\ref{lem:average-diamdeg4}, for some deterministic constant $c_0>0$ and every $a,\e\in(0,1)$, almost surely, $E_{k,1}(c_0,a)\cap E_{k,2}(\e)$ happens for all large enough $k$.

    Let $n$ be the smallest integer greater than $c_0q$.  We set $a = 2^{-8n}$ and $\e = 2^{-32n}$.
   We assume $k$ is large enough and  work on the event $E_{k,1}(c_0,a)\cap E_{k,2}(\e)$. Consider a cell $H^*$ intersecting $[-2|S_k|,2|S_k|]^2\backslash [-|S_k|,|S_k|]^2$. Let $S^*_0$ be a square in $\cD$ with side length $2^{-8n}|S_k|$ intersecting $H^*$. For $j=1,2,...$, we recursively define $S_j^*$ to be the square concentric to $S_{j-1}^*$, and having side length $3|S_{j-1}^*|$. Then it is clear that $|S_n^*|<2^{-n}|S_k|$ and $\dist(0,S_n^*)\geq |S_k|/2$. We define a mass function $m^*$ as follows. For each $j=2,...,n$ and $H\subseteq S_{j}^*\backslash S_{j-1}^*$, we set $m^*(H) = |S_j^*|^{-1}\diam(H)$, and we set $m^*(H)=0$ for all other $H$. Then for each $j$,  $S_j^*$ is contained within the union of 4 squares in $\cD$ with side length at most $2|S_j^*|$, and therefore by the definition of the event $E_{k,1}(c_0,a)$, we have
    \begin{equation}\label{eq:pf-VELGMS-2}
        \sum_{H\subseteq S_{j}^*\backslash S_{j-1}^*} m^*(H)^2 = |S_j^*|^{-2}\sum_{H\subseteq S_{j}^*\backslash S_{j-1}^*} \diam(H)^2  \leq 16c_0,
    \end{equation}
and therefore $\area( {m^*})\leq 16c_0(n-1)$. On the other hand, using the definition of $E_{k,2}(\e)$, $H^*\subseteq S_1^*$ and $H_0\cap S_n^*=\emptyset$. For each path $\ell\in\Gamma_{H^*}$, using $H\cap H'\neq\emptyset$ for $H\sim H'$  for all large enough $k$, together with the almost planarity condition for $\cM([-2^{k+3},2^{k+3}]^2)$ from Proposition~\ref{prop:almost-planarity}, 
$\cup_{H\in\ell}H$ must cross all the regions $S_{j}^*\backslash S_{j-1}^*$,  {which implies} 
\begin{equation}\label{eq:pf-VELGMS-3}
        \sum_{H\subseteq\ell\cap\cH({ S_{j}^*\backslash S_{j-1}^*})} m^*(H) =  |S_j^*|^{-1}\sum_{H\subseteq\ell\cap\cH({ S_{j}^*\backslash S_{j-1}^*})} \diam(H)  \geq  |S_j^*|^{-1}(\frac{1}{3}|S_j^*|-2\cdot2^{-32n}|S_k|)\geq \frac{1}{6},
    \end{equation}
and  $\mathrm{len}_{m^*}(\ell)\geq \frac{1}{6}(n-1)$. Therefore  by~\eqref{eq:pf-VELGMS-2} and ~\eqref{eq:pf-VELGMS-3}, 
$$\VEL(\Gamma_{H^*})\geq \frac{n-1}{36\cdot16c_0}\geq\frac{q}{10^3} = q_0.
$$
    Since $H^*$ is any arbitrary cell intersecting $[-2|S_k|,2|S_k|]^2\backslash [-|S_k|,|S_k|]^2$ and almost surely $E_{k,1}(c_0,a)\cap E_{k,2}(\e)$ occurs for all large enough $k$, we conclude our proof.
\end{proof}

\begin{proof}[Proof of Proposition~\ref{prop:max-diam-0-cone}]
 Without loss of generality assume $r=1$.   Assume on the contrary that there exists $\delta_0>0$ such that, with probability  {at least $\delta_0$}, there exists $\e_n\to0$ 
  and disks $D_{H_n}$ in the packing $\e_n\cP$ intersecting $\bbD$ with diameter at least $\delta_0$. Let $z_{H_n}$ and $r_{H_n}\geq\delta_0/2$ be the center and radius, respectively, of $D_{H_n}$. Since $D_{H_n}$ intersects $\bbD$, $|z_{H_n}|-r_{H_n}<1$, and by Lemma~\ref{lem:VEL-packing}, 
  $$\VEL(\Gamma_{H_n})\leq 4|z_{H_n}|r_{H_n}^{-1}\leq 4(r_{H_n}+1)r_{H_n}^{-1}\leq 4(1+2\delta_0^{-1}).$$
  Furthermore, since for any given disk in $\cP$, its diameter converges to 0 in the packing $\e\cP$ as $\e\to0$, we see that the vertices in  $\{v_1,v_2,...\}$ are infinite. In other words, with probability at least $\delta_0$, there exists infinitely many cells $\{H_n\}$ in $\cH$ such that $\VEL(\Gamma_{H_n})\leq 4(1+2\delta_0^{-1})$. This  contradicts Proposition~\ref{prop:finite-VEL}.
\end{proof}

\subsection{Proof of Theorem~\ref{thm:CellSystemPack}}\label{subsec:pf-Cell-system-Pack}

Our proof of Theorem~\ref{thm:CellSystemPack} is mostly based on the convergence of the same random walk on planar maps with different embeddings to the Brownian motion as in Proposition~\ref{prop:inv-circle} and Proposition~\ref{prop:random-walk-on-packing}, as well as the consideration of wrapping around events for 2D Brownian motion. In the rest of this section, we use the following notation. For a continuous-time walk $(X_t)_{t\geq0}$ on a probability space $(\Omega,\cF,\bbP)$ and a bounded set $A\subseteq\bbC$ where $z\notin A$, we write $\bbP^z$ for the law of $(X_t)_{t\geq0}$ started from $z$, and $\tau_{A}$ for the first time $T$ when the trace of $(X_t)_{0\leq t\leq T}$ separates $A$ from infinity. For a domain $D\subseteq\bbC$, we let $\sigma_D = \inf\{t\geq 0:X_t\notin D\}$ denote the exit time of $D$. We write $\tau_{w}$ instead of $\tau_{\{w\}}$ if $A=\{w\}$. For $\ell>10$, we write $\sigma_{\ell}$ for the first time when the walk exits $[-2^{\ell-1},2^{\ell-1}]^2$.
\begin{proposition}\label{prop:BM-f}
   Let $f:\bbC\to\bbC$ be a homeomorphism and $(W_t)_{t\geq0}$ be the standard Brownian motion in $\bbC$ starting from 0. Suppose $f(0)=0,f(1)=1$, and for every $z\in\bbC$, $\big(f(z+W_t)\big)_{t\geq0}$ has the law of a standard Brownian motion starting from $f(z)$ modulo time change, then either $f$ is the identity function,  or $f(z)=\ol z$ for every $z\in\bbC$.
\end{proposition}

\begin{proof}
Let $g:\bbR^2\to\bbR$ be any harmonic function.  Let $U\subseteq B(0;r)$ and $z\in U$. By our assumption, we have 
$$g(f(z)) = \bbE^{f(z)}[g(W_{\sigma_{f(U)}})] = \bbE^z[g(f(W_{\sigma_U}))]$$
  and $(W_t)_{t\geq0}$ is the standard 2D Brownian motion.     As we condition on the first time when $(X_t)_{t\geq0}$ exits $B(z;\e)$ where $B(z;\e)\subseteq U$, the strong Markov property implies that $g\circ f$ satisfies the mean value property and is therefore harmonic. In particular, if we write $f = u+iv$ where $u,v$ are both real-valued, then $u,v,u^2-v^2$ and $uv$ are all harmonic. This further implies that $u,v$ are smooth and by elementary computations \eqb\label{eq:CREQ} |\nabla u|^2=|\nabla v|^2, \ \ \nabla u\cdot\nabla v=0.\eqe Let $u_1,v_1$ be analytic functions such that $\mathrm{Re}\, u_1=u$ and $\mathrm{Re}\, v_1=v$. Then $f =h_1+\ol h_2$ where both $h_1=\frac{u_1+iv_1}{2}$ and $h_2=\frac{u_1-iv_1}{2}$ are analytic. The Cauchy-Riemann equation along with~\eqref{eq:CREQ} implies that $h_1'(z)\ol {h_2'(z)}=0$ for every $z\in\bbC$, and the identity theorem implies that either $h_1'$ or $h_2'$ vanishes on the whole plane. This implies that $f$ is holomorphic or antiholomorphic. Since $f$ is also a homeomorphism of $\bbC$, $f$ has to be affine. By our choice $f(0)=0$, $f(1)=1$, we must have $f(z)=z$ or $f(z)=\ol z$. This finishes the proof. 
\end{proof}
\begin{proposition}\label{prop:0-thm:CellSystemPack}
    Theorem~\ref{thm:CellSystemPack} holds up to a random scaling constant, i.e., the conclusions of Theorem~\ref{thm:CellSystemPack} hold if we replace~\eqref{eq:thm:CellsystemPack} with
        \begin{equation}\label{eq:thm:CellsystemPack-C}
        \lim_{r\to\infty}\frac{1}{r}\max_{H\in \cH(B(0;r))}|\mathbf{A}c(H)-{z(r)o(H)}| = 0, \ \  a.s.,
    \end{equation}
    where $z(r)\in\bbC\backslash\{0\}$ is some random scaling factor.
\end{proposition}

\begin{proof}
The first claim for no macroscopic circles is an immediate consequence of Proposition~\ref{prop:max-diam-0-cone}. Now
  let $\Sigma$ be the matrix in Proposition~\ref{prop:inv-circle}. We pick a deterministic $2\times 2$ matrix $\mathbf{A}$, such that the trace of a 2D Brownian motion with covariance $\Sigma$ under the linear transform $\mathbf{A}$ has the same law as that of a standard 2D Brownian motion. If $\cH$ is rotation invariant in law for some $\theta_0\in(0,2\pi)\backslash\{\pi\}$, then we set $\mathbf{A}$ to be the identity matrix. We fix a circle packing $\cP$ for $\cH$ such that the disk for the cell $H_0$ is the unit disk. Let $H_1^\e$ be the cell in $\e\mathbf{A}\cH$ containing 1. Then we pick $\wt\e>0$ and apply a rotation $z\mapsto e^{i\theta_{\wt \e}}z$, such that in the packing $ e^{i\theta_{\wt \e}}\wt\e\cP$, the disk for $H_1^\e$ is centered at 1. For notational simplicity, we assume $\theta_{\wt \e}=0$. Clearly $\wt\e\to0$ as $\e\to0$ and vice versa. We write $\wt\e_k$ for the value of $\wt\e$ for $\e_k=2^{-k}$. Let $k,\ell>10$. Consider the collection of points $\{z_H^{k+\ell}:H\in\cV\cM([-2^{k+\ell},2^{k+\ell}]^2)\}$ and simple curves $\{\gamma_{e}^{k+\ell}: e\in\cE\cM([-2^{k+\ell},2^{k+\ell}]^2)\}$ as   in Proposition~\ref{prop:almost-planarity}. Let $f^{k,\ell}$ be a homeomorphism of $\bbC$ sending each point $2^{-k} \mathbf{A}z_H^{k+\ell}$ to the center of the disk for the cell $H$ in $\wt\e_k\cP$, and sending each curve $2^{-k}\mathbf{A}\gamma_{\{H,H'\}}^{k+\ell}$ to the line segment joining the centers of the disks for the cells $H$ and $H'$ in $\wt\e_k\cP$. Below we will prove that almost surely, for each fixed $\wt\ell>10$, for $k\to\infty$ and  a sequence $\ell:=\ell(k)\to\infty$ sufficiently slowly, $f^{k,\ell}$  converges to the identity function uniformly on $[-2^{\wt \ell-5},2^{\wt \ell-5}]^2$, which implies~\eqref{eq:thm:CellsystemPack-C} by further sending $\wt\ell\to\infty$. 

Let $\wh X^{z,\mathrm{circle},k,\ell}$ be the linearly interpolated random walk on $\wt\e_k\cP$ described in 
Proposition~\ref{prop:random-walk-on-packing}, and let $\wh X^{z,k,\ell}$ be interpolated random walk on $2^{-k}\mathbf{A}\cH$ as in Theorem~\ref{thm:GMSinvariance},  
except that for $H\in \cV\cM([-2^{k+\ell},2^{k+\ell}]^2)$ we use the points $\{2^{-k}\mathbf{A}z_H^{k+\ell}\}$ and for $e\in\cE\cM([-2^{k+\ell},2^{k+\ell}]^2)$ we interpolate along the edges $\{2^{-k}\mathbf{A}\gamma_{e}^{k+\ell}:e\in \cE\cM([-2^{k+\ell},2^{k+\ell}]^2)\}$. 
   By Lemma~\ref{lem:M(S)}, we can further modify the starting vertex of $\wh X^{z,k,\ell}$ for $z\notin \{\mathbf{A}z_H^{k+\ell}\}$ such that almost surely for large enough $k$, the traces of $f^{k,\ell}(\widehat X^{z,k,\ell})$ and $\widehat X^{f^{k,\ell}(z),\mathrm{circle},k,\ell}$ are within $f^{k,\ell}([-2^{ \ell-1},2^{\ell-1}]^2)$ can be coupled to be the same. Indeed, this coupling exists since $\e\cH$ and $\wt\e\cP$ can be viewed as 
   two embeddings of the associated map for $\cH$ and have identical conductance, and the modification of starting vertex induces $o_k(1)$ error since there are no macroscopic edges as in~\eqref{eq:almost-planarity}. We work on this coupling. Then by Proposition~\ref{prop:inv-circle} and Proposition~\ref{prop:random-walk-on-packing}, for fixed $\ell$, as $k\to\infty$ the conditional law of $\widehat{X}^{z,k,\ell}$ and $\widehat X^{z,\mathrm{circle},k,\ell}$ given $\e_k\cH$ and $\wt\e_k\cP$ converges to a constant multiple of planar Brownian motion   with respect to the topology on curves induced in~\eqref{eq:curve-metric}.  Throughout the proof, we write $\bbP_{\rm BM}$ for the law of standard Brownian motion, $\bbP_{k,\ell}$ for the conditional law of $\widehat X^{k,\ell}$  given  $\e\cH$, and $\wt\bbP_{k,\ell}$  for the conditional law of $\widehat X^{\mathrm{circle},k,\ell}$ given $\wt\e_k\cP$. We fix a sequence $\ell:=\ell(k)$  such that $\ell(k)\to\infty$ as $k\to\infty$ sufficiently slowly such that the above convergence to Brownian motion continues to hold for the sequence $\ell(k)$, and in the planar embedding supplied by Proposition~\ref{prop:almost-planarity} we have
   $$ \frac{1}{2^k}\max_{H\in \cV\cM([-2^{k+\ell},2^{k+\ell}]^2)}\max_e \dist(\gamma_e,H)\to0$$ in the sense of~\eqref{eq:almost-planarity}. Consider the point $2^{-k}\mathbf{A}z_{H_0}^{k+\ell}$ and the point $2^{-k}\mathbf{A}z_{\wh H}^{k+\ell}$ for the cell $\wh H$ corresponding to   $H_1^{\e_k}$ is located at 1. Indeed, by our choice of $\ell$, $(2^{-k}\mathbf{A}z_{H_0}^{k+\ell},2^{-k}\mathbf{A}z_{\wh H}^{k+\ell})\to(0,1)$ almost surely. We can pick a homeomorphism $g^{k,\ell}$ of $\bbC$ such that $g^{k,\ell}(z)=z$ for $|z|\geq 2$ and sending $(2^{-k}\mathbf{A}z_{H_0}^{k+\ell},2^{-k}\mathbf{A}z_{\wh H}^{k+\ell})$ to $(0,1)$, while $g^{k,\ell}$ and $(g^{k,\ell})^{-1}$ both converge to the identity uniformly as $k\to\infty$. Then we modify the points $\{2^{-k}\mathbf{A}z_{H}^{k+\ell}\}$ via the map $g^{k,\ell}$ and therefore we may assume that $f^{k,\ell}$ fixes 0 and 1. Fix $\wt\ell>10$, let $K = [-2^{\wt\ell-5},2^{\wt\ell-5}]^2$ and $d_K = 2^{\wt\ell-4}$. We will verify that  $\{ f^{k,\ell}{\,:\,k\in\bbN, \ell(k)>\wt\ell+1} \}$ is uniformly bounded and equicontinuous on $K$ {for any fixed $\wt\ell$}, and we will later send $\wt\ell\to\infty$.
   
    \textbf{Step 1: Uniform boundedness on $K$.} For $z\neq 0,1$, we work on the event $\{\tau_z\leq\tau_1<\sigma_\ell \}$, i.e., when the trace separates 1 from infinity, it separates $z$ from infinity and has not exited $B(0;2^{\ell-1}$). Let $\gamma_K$ be a deterministic curve formed by the union of the line segment from 0 to $(2d_K+3)i$ and the line segment from $(2d_K+3)i$ to {$2+(d_K+2)i$}, and finally the circle $\{|z|=|2+(d_K+2)i|\}$, such that it is disjoint from $B(0;d_K+1)$, crosses itself and separates $B(0;{d_K+1})$ from infinity. Then for any $\delta>0$, with positive probability, the planar Brownian motion started from 0 stays within the $\delta/2$-neighborhood of $\gamma_K$ and form a loop around $K\cup\{1\}\subset B(0;d_K+1)$  during time $[0,2]$. In particular, a.s.\ for fixed $\ell$ and large enough $k$, $\bbP_{k,\ell}^0\big[\{\tau_z\leq\tau_1\}\cap\{\sup_{0\leq t\leq \tau_1} |\wh X_t|\leq 2^{\ell-1}\}\big]$ is uniformly positive for $z\in K$. Since $f^{k,\ell}(\widehat X^{z,k,\ell})$ and $\widehat X^{f^{k,\ell}(z),\mathrm {circle},k,\ell}$ are coupled to have the same trace until leaving $ f^{k,\ell}([-2^{\ell-1},2^{\ell-1}]^2)$ 
   for each $z\in K$, there exists $\e'>0$ such that $\wt\bbP_{k,\ell}^0[\tau_{f^{k,\ell}(z)}\leq\tau_1]>3\e'$ for $k$ sufficiently large. On the other hand, for each $\e'>0$, there exists $M,\delta_1>0$, such that the probability that a Brownian motion from 0 separates $B(1;\delta_1)$ from infinity before leaving $B(0;M)$ is at least $1-\e'$. Using the convergence in law of $\widehat{X}^{\mathrm{circle}, k,\ell}$ to standard Brownian motion, for $k$ sufficiently large, with probability at least $1-2\e'$, the trace of $\widehat{X}^{0,\mathrm{circle}, k,\ell}$ separates 1 from infinity before leaving $B(0;2M)$. In other words, for $k$ sufficiently large,  $\wt\bbP_{k,\ell}^0[\tau_{z}\leq\tau_1]\leq2\e'$ for $|z|>2M$. 
   Thus $|f^{k,\ell}(z)|\leq 2M$ for all $z\in K$, large enough $k$, and $\ell>\wt\ell+1$. 
   
   By a similar argument, for each compact set $K'\subseteq\bbC$, there exists $\e''>0$ such that for some  $\ell_0$, it holds that for all $\ell>\ell_0$ there exists $k_\ell$ such that for $k>k_\ell$, $\inf_{z'\in K'}\wt\bbP_{k,\ell}^0[\tau_{z'}\leq\tau_1]> 2\e''$. Next we  pick $\wt\ell>0$ such that for large $\ell,k$ chosen in the similar way, $\bbP_{k,\ell}^0[\tau_1<\sigma_{\wt\ell-5}]>1-\e''$ . The latter inequality implies that for large enough $\ell$ and   $k$, $\sup_{z'\notin  f^{k,\ell}(K)}\wt\bbP_{k,\ell}^0[\tau_{z'}\leq\tau_1]<  \e''$, which further implies that almost surely,  there exists some $\ell_0$, such that for all $\ell>\ell_0$, there is a  $k_\ell>0$ such that $K'\subseteq f^{k,\ell}(K_{\wt \ell})$ for all $k>k_\ell$.

   \textbf{Step 2: Equicontinuity near 0.}  
   Let $\delta>0$. For $z\neq 0,1$, we work on the event $\{\tau_z=\tau_1\}$, i.e.,   when the trace separates 1 from infinity, it separates $z$ from infinity at the same time.   Using the same $\gamma_K$ and Brownian motion argument as in the previous step, one can see that  there exists $\delta'>0$ such that   for $k$ large enough $\bbP_{k,\ell}^0[\tau_z=\tau_1<\sigma_\ell]>\delta'$ and $\wt\bbP_{k,\ell}^0[\tau_z=\tau_1]>\delta'$ when $z\in K$ and $|z|>\delta$. On the other hand, there exists $\delta''>0$ such that the probability that a planar Brownian motion separates $B(0;2\delta'')$ from infinity before leaving $B(0;1/2)$ is at least $1-\delta'/2$, and thus using the convergence in law of $\widehat X^\e$ to Brownian motion in local uniform topology, for $|z|<\delta''$ and all sufficiently large $k$, $\bbP_{k,\ell}^0[\tau_z<\tau_1;\tau_z<\sigma_\ell]>1-\delta'$ and $\wt\bbP_{k,\ell}^0[\tau_z<\tau_1]>1-\delta'$. One can see this by applying~\cite[Theorem 3.2 and Theorem 5.1]{Van14}. Since $f^{k,\ell}(\widehat X^{k,\ell})$ and $\widehat X^{\mathrm {circle},k,\ell}$ are coupled to have the same trace until leaving $f^{k,\ell}(K)$, we have a.s.\ for large $k$, $\wt\bbP_{k,\ell}^0[\tau_{f^{k,\ell}(z)}=\tau_1]>\delta'$ for $|z|>\delta$ and $\wt\bbP_{k,\ell}^0[\tau_{f^{k,\ell}(z)}<\tau_1]>1-\delta'$ for $|z|<\delta''$. We conclude  that there a.s.\ exists $k_0$, such that for any $k>k_0$, if $|z|<\delta''$ then $|f^{k,\ell}(z)|<\delta$, and if $|z|\geq \delta$ then $|f^{k,\ell}(z)|\geq\delta''$. This verifies the equicontinuity near 0. 

  \textbf{Step 3: Equicontinuity at other points in $K$.} This step is almost identical to the argument in step 2 by considering walks started from a fixed $z\in K$ with $z\neq 0$. 
  From the previous cases, for fixed $\ell$, uniformly in $k$, $\{f^{k,\ell}(z)\}$ is bounded away from 0 and $\infty$. Pick $0<\delta_z<\min\{|z|/2,1\}$,  fix $0<\delta<\frac{1}{2}\min \{|f^{k,\ell}(w)|:k\geq 10, |w-z|<\delta_z \} $, $M\geq \sup_{k>0}|f^{k,\ell}(K)|$ and $z''\in B(0;2M)\backslash B(0;2\delta)$ with $|z''-f^{k,\ell}(z)|>\delta$. Let $\gamma_{f^{k,\ell}(z),z''}$ be a finite path starting from $f^{k,\ell}(z)$ that first separates $B(0;\delta/4)$ and   then $B(z'';\delta/4)$ from infinity, but does not separate any points in $\bbC\backslash( B(z'';{\delta/2})\cup B(0;\delta/2))$ from infinity.  Then using a similar argument in previous cases, where the Brownian motion has positive probability to stay sufficiently close to $\gamma_{f^{k,\ell}(z),z''}$, and the random walks converge in law to Brownian motion in local uniform topology, we see that for each $\delta>0$ there exists $\delta'>0$   such that a.s.\  for sufficiently large $k$, $\wt\bbP_{k,\ell}^{f^{k,\ell}(z)}[ \tau_0<\tau_{z''}]>\delta'$ if $|z''-f^{k,\ell}(z)|>\delta$ and $2\delta<|z''|< 2M$. On the other hand, for some $\delta''>0$, for all $k$ large enough, $\bbP_{k,\ell}^z[\tau_{z'}<\tau_0;\tau_{z'}< \sigma_\ell]>1-\delta'$ for all $|z'-z|<\delta''$,  which further implies that $\wt\bbP_{k,\ell}^{f^{k,\ell}(z)}[ \tau_{f^{k,\ell}(z')}<\tau_{0}]>1-\delta'$. This implies a.s.\ for large enough $k$, $|f^{k,\ell}(z')-f^{k,\ell}(z)|<\delta$ whenever $|z-z'|\leq \min\{\delta'',\delta_z\}$ and $z'\in K$. 
  
  This finishes the proof for the tightness of $\{\{f^{k,\ell}\}|_{z\in K}:k\in\bbN,\ell=\ell(k)\}$ for each $\wt\ell>10$ and $K=[-2^{\wt\ell-5},2^{\wt\ell-5}]^2$.  By the Arzel\`a-Ascoli theorem, for any subsequence of $k$, there exists a further subsequence $k_n,\ell_n\to\infty,$ such that for every $\ell>10$, $f^{k_n,\ell_n}$ converges uniformly on $[-2^{\ell-5},2^{\ell-5}]^2$  to 
    a continuous function $f$ on $\bbC$ such that $f(0)=0,f(1)=1$. By the last sentence of Step 1, we have $f(\bbC)=\bbC$. Using an identical argument as Step 2 and Step 3, for each compact set $K'$, for each large enough $\ell$, $\{(f^{k,\ell})^{-1}|_{K'}\}$ is also tight. 
    By taking a further subsequence we may assume that $f$ is a homeomorphism. Then if $(W_t)_{t\geq0}$ is a Brownian motion from $z$,  $f(W_t)_{t\geq0}$ has the law of 
   a planar Brownian motion from $f(z)$ modulo time change. Then by Proposition~\ref{prop:BM-f}, $f$
must be either the identity function or the reflection $f(z)=\ol z$. To rule out the $f(z)=\ol z$ case, consider cells $H_1,...,H_m$ in $\cH{(\partial [-2^{k},2^k]^2)}$ such that the edges $\gamma_e^{k+1}$ connecting the points $z_{H_j}^{k+1}$ form a clockwise loop which is close to $\partial [-2^{k},2^k]^2$. Then the corresponding loop formed by the disks for $H_1,...,H_m$ is also clockwise, which implies that we cannot have $f(z)=\ol z$.
\end{proof}

Next we finish the proof of Theorem~\ref{thm:CellSystemPack} by proving that one can choose the constant $z(r)$ in~\eqref{eq:thm:CellsystemPack-C} to be equal to 1. 
Further note that for a fixed circle packing $\cP$, the map $H\mapsto o(H)$ is also harmonic with respect to the Dubejko conductance $\mathfrak{c}$.  Theorem~\ref{thm:CellSystemPack} then follows once we prove that the two functions are linearly correlated. To achieve this, we use a compactness-based improvement-of-flatness argument of Avellaneda-Lin type as in~\cite{AL87}  and also~\cite[Section 3.3]{GNO1414}. In the following, we fix a circle packing $\cP$ for $\cH$ and write $o(H)$ for the center of the disk for $H$ in $\cP$. For $r>0$ and a function $g:\cH\to\bbC$, we write $\|g\|_r = \max_{H:o(H)\in B(0;r)}|g(H)|$. 

\begin{lemma}\label{lem:harmonic-tight-1}
    Let $\{g_n\}$ be any collection of functions $\cH\to\bbR$ that are harmonic with respect to $\fc$ such that $\|g_n\|_{2^n}\leq 1$. Let $u_n:\bbC\to\bbC$ be defined by setting $u_n(2^{-n}o(H)) = g_n(H)$ and linearly interpolating on triangular faces under the circle packing $2^{-n}\cP$. Then the set of  functions $\{u_n|_{B(0;1/2)}\}$ is tight, and any subsequential limit is harmonic. 
\end{lemma}

\begin{proof}
    Let $2^{-n}\wh X^{2^nz,\mathrm{circle}}$ be the random walk with Dubejko weights as in Proposition~\ref{prop:random-walk-on-packing} and $\sigma_n^z$ be the first time when $2^{-n}\wh X^{2^nz,\mathrm{circle}}$ hits the center of a disk in $2^{-n}\cP$ that is not contained in $B(0;5/6)$. Then the harmonicity of $g_n$ together with Proposition~\ref{prop:max-diam-0-cone} imply that for each $H$ with $z=2^{-n}o(H)\in B(0;2/3)$, $u_n(z) = \bbE[u_n(2^{-n}\wh X^{2^nz,\mathrm{circle}}_{\sigma_n^z})|\cP]$. Furthermore, for $z'=2^{-n}o(H)\in B(0;2/3)$, by~\cite[Lemma 3.12]{GMSinvariance}\footnote{This lemma is originally stated for the simple random walk, but the same argument based on Wilson's algorithm~\cite{Wil96} applies in the weighted setting as well.}, given the packing $\cP$, the total variation distance between $2^{-n}\wh X^{2^nz,\mathrm{circle}}_{\sigma_n^z}$ and $2^{-n}\wh X^{2^nz',\mathrm{circle}}_{\sigma_n^{z'}}$ is no greater than the probability that $2^{-n}\wh X^{2^nz,\mathrm{circle}}$ does not separate $\partial B(0;3/4)$ from $z'$ before time $\sigma_n^z$. On the other hand, the same wrapping around argument as in the proof of Proposition~\ref{prop:0-thm:CellSystemPack} implies that uniformly over $z\in \ol{B(0;1/2)}$, the latter probability converges to 0 as $n\to\infty$ and $z'\to z$. This shows the equicontinuity of $u_n$, and by definition and Proposition~\ref{prop:max-diam-0-cone}, $|u_n(z)|\leq 1$ for $|z|\leq 1/2$, which further proves the tightness of $\{u_n\}$ by Arzel\`a-Ascoli theorem. The harmonicity of the subsequential limit is then a consequence of the convergence of the walk $2^{-n}\wh X^{2^nz,\mathrm{circle}}$ to Brownian motion along with a similar mean-value property as in the proof of Proposition~\ref{prop:BM-f}.
\end{proof}

\begin{lemma}\label{lem:harmonic-tight-2}
    Let $g:\cH\to\bbR$ be  harmonic with respect to $\fc$ such that $\|g\|_{2^n}\leq 2^{(1+\delta)n}$ for some $0<\delta<1$ and every $n$ large enough. Then there exists (random) $a,b_1,b_2\in\bbR$ such that $g(H) = a+b_1\mathrm{Re}\, o(H)+b_2\mathrm{Im}\,o(H)$.
\end{lemma}

\begin{proof}
 Fix $\alpha\in (\delta,1)$.   Define $$E_g(r) = \inf_{a,b_1,b_2\in\bbR}\max_{H: o(H)\in B(0;r)}|g(H) - a-b_1 \mathrm{Re}\, o(H)-b_2\mathrm{Im}\,o(H)|.$$
 Clearly $E_g(r)\leq \|g\|_r$. 
    We claim that there exists deterministic $\ell\in\bbZ\cap[3,\infty)$ and random  $N_0>0$ such that $E_g(2^{n-\ell})\leq 2^{-\ell(1+\alpha)}\|g\|_{2^n}$ for every $\fc$-harmonic $g$ and $n\geq N_0$. By standard properties of harmonic functions, there exist absolute constants $\ell\in\bbN,C>0$ such that for any ordinary harmonic function $u_*$ on $B(0;1/2)$ with $|u_*|\leq 1$, \eqb\label{eq:pf-lem:harmonic-tight-2}\sup_{z\in B(0;2^{-\ell+1})}|u_*(z)-u_*(0)-\nabla u_*(0)\cdot z|\leq C2^{-2\ell}.\eqe
    Now we choose $\ell$ such that $(2+C)2^{-2\ell}\leq 2^{-(1+\alpha)\ell}$. Assume that our claim is not true. Then there exists $j_n\to\infty$, such that there exists $\fc$-harmonic $g_{j_n}$ such that $E_{g_{j_n}}(2^{j_n-\ell})> 2^{-\ell(1+\alpha)}\|g_{j_n}\|_{2^{j_n}}$ for every $n$, and without loss of generality we assume $j_n=n$ for simplicity. We set $\wt g_n = g_n/\|g_n\|_{2^n}$ and let $\wt u_n:\bbC\to\bbC$ be defined via $\wt u_n(2^{-n}o(H)) = \wt g_n(H)$ and linear interpolation as in Lemma~\ref{lem:harmonic-tight-1}. Then by Lemma~\ref{lem:harmonic-tight-1}, without loss of generality we may assume that $\wt u_n$ converges to a harmonic function $u_*$ uniformly on $B(0;1/2)$, and in particular $|\wt u_n(z)-u_*(z)|\leq 2^{-2\ell}$ for all sufficiently large $n$. Then by Proposition~\ref{prop:max-diam-0-cone} and~\eqref{eq:pf-lem:harmonic-tight-2} along with the definition of $\wt u_n$, for sufficiently large enough $n$, 
    $$\max_{H:o(H)\in B(0;2^{n-\ell})}\Big|g_n(H)-\|g_n\|_{2^n}u_*(0)-2^{-n}\|g_n\|_{2^n}\nabla u_*(0)\cdot o(H) \Big|\leq (C+1)2^{-2\ell}\|g_n\|_{2^n}.$$
    This contradicts our assumption that $E_{g_n}(2^{n-\ell})> 2^{-\ell(1+\alpha)}\|g_n\|_{2^n}$ and therefore verifies our claim, and further by substituting $g$ with $g(H) - a-b_1 \mathrm{Re}\, o(H)-b_2\mathrm{Im}\,o(H)$ we have  $E_g(2^{n-\ell})\leq 2^{-\ell(1+\alpha)}E_g({2^n})$. Now fix $n\geq N_0$. Then by an iteration, for large enough $j\in\bbN$, $E_g(2^n)\leq  2^{-\ell(1+\alpha)j}E_g(2^{n+j\ell})\leq 2^{-\ell(1+\alpha)j}\|g\|_{2^{n+j\ell}}\leq 2^{-(1+\alpha)j\ell+(1+\delta)(n+j\ell)}$ using our assumption on $g$. Since $j$ is arbitrary and $\alpha>\delta$, we conclude that $E_g(2^n)=0$ for every large enough $n$. This further implies that $g$ is linear with respect to $o(H)$ when restricted to $o(H)\in B(0;2^n)$ for every $n$ and further concludes the proof.
\end{proof}

\begin{proof}[Proof of Theorem~\ref{thm:CellSystemPack}]
    By the last claim of Proposition~\ref{prop:inv-circle}, for some $\phi_\infty:\cH\to\bbC$ that is harmonic with respect to the conductance $\mathfrak{c}$, we have $\sup_{H\in\cH(B(0;r))}\frac{1}{r}|c(H)-\phi_\infty(H)|\to0$ almost surely as $r\to\infty$. By Proposition~\ref{prop:0-thm:CellSystemPack}, as we compare~\eqref{eq:thm:CellsystemPack-C} at the scale  $2^n$ and $2^{n+1}$, using Lemma~\ref{lem:no-macroscopic-cell} along with Proposition~\ref{prop:max-diam-0-cone}, we have $z(2^{n+1})/z(2^n)\to1$ as $n\to\infty$, from which one can prove that $\|\phi_\infty\|_{2^{n+1}}\leq 2^{(1+\delta/2)}\|\phi_\infty\|_{2^n}$ for some $\delta\in (0,1/4)$ and every large enough $n$.  This gives $\|\phi_\infty\|_{2^n}\leq 2^{(1+\delta)n}$ for some $0<\delta<1$ and every $n$ large enough, and $\phi_\infty$ has to linear in terms of $o(H)$ thanks to Lemma~\ref{lem:harmonic-tight-2}. Again by the closeness of $o(H)$ and $H$ up to scaling as in Proposition~\ref{prop:0-thm:CellSystemPack} and the closeness between $\phi_\infty$ and the a priori embedding $\cH$ from Proposition~\ref{prop:inv-circle}, we can further infer that $\mathbf{A}\phi_\infty(H)$ must be equal to $o(H)$.  This finishes the proof of~\eqref{eq:thm:CellsystemPack} and Theorem~\ref{thm:CellSystemPack}. 
\end{proof}

\section{Convergence of cell configurations under the uniformization map}\label{sec:uniformization}

In this section we prove Theorem~\ref{thm:invariance-uniformization}. Throughout this section, we assume that $\cH$ is a random cell configuration that satisfies the constraints in Theorem~\ref{thm:invariance-uniformization}. 
In Section~\ref{subsec:regularity}, we collect several regularity estimates for the Riemann surface $M(\cH)$ under a uniformization map $\varphi:M(\cH)\to\bbC$, and prove that $M(\cH)$ is a.s.\ parabolic. 
In Sections~\ref{subsec:GMS2.3} and~\ref{subsec:GMS2.4}, we use a  {continuum} variant of the  {discrete} argument from~\cite[Section 2.3, Section 2.4]{GMSinvariance}{, along with results from Section~\ref{subsec:regularity},} to prove that there is a harmonic function $\phi_\infty:M(\cH)\to\bbC$ such that {on} a large scale, it is close to the function mapping vertices on $M(\cH)$ to their corresponding cells in $\cH$. In Section~\ref{subsec:pf-thm:invariance-uniformization}, we conclude the proof of Theorem~\ref{thm:invariance-uniformization} by showing that the function $\phi_\infty\circ\varphi^{-1}$ must be a scalar times a deterministic linear transform $\mathbf{A}$ and a rotation.  This argument is different from \cite{GMSinvariance} and gives a general  {and elementary} method for proving that a particular class of random harmonic functions are of this form.

\subsection{Regularity estimates for the Riemann surface}\label{subsec:regularity}

In this section, we prove several regularity estimates for the Riemann surface $M(\cH)$ under a uniformization map $\varphi:M(\cH)\to\bbC$ that will be applied in the proof of Theorem~\ref{thm:invariance-uniformization}. In Lemmas~\ref{lem:area/length-1} and~\ref{lem:area/length-2} 
we prove that if we have a microscopic set in $\cH$ (resp.\ $\varphi(M(\cH))$), then the corresponding set in $\varphi(M(\cH))$ (resp.\ $\cH$) will also be microscopic. Let $D_k$ be the union of    $\varphi(P_H)$ of the cells $H\in\cH$ intersecting $[-2^k,2^k]^2$.  In Lemma~\ref{lem:area/length-4}, we prove that the diameter of $D_k$ has a growth rate faster than $2^{k/l}$ for some deterministic $l$. In Lemma~\ref{lem:area/length-5}, we prove that for large $k$, $D_k$ contains a disk with diameter at least $c\diam(D_k)$. These lead to Lemma~\ref{lem:parabolic}, where we prove that $M(\cH)$ is a.s.\ parabolic. Finally in Lemma~\ref{lem:cell-polynomial-growth}, we prove that the function mapping points in $\varphi(M(\cH))$ to their corresponding cells in $\cH$ has a polynomial growth, which will be applied in Section~\ref{subsec:pf-thm:invariance-uniformization} in the proof of Theorem~\ref{thm:invariance-uniformization}. 

The key idea in these proofs is as follows. For each cell $H\in\cH$ and its vertex $v_H$ on $M(\cH)$, we associate them with a semi-flower $P_H\subseteq M(H)$ as in Figure~\ref{fig:Koebe-distortion}. By Lemma~\ref{lem:Koebe-distortion} below, the geometry of a semi-flower under a uniformization map $\varphi$ can be bounded in terms of its degree. In~\cite[Equation (2.8)]{HeSchramm}, He and Schramm used a length-area argument in the context of circle packing, which originates from the Length-Area Lemma in~\cite{RodinSullivan87}. Roughly speaking, one can control  the diameter of a union of circles by estimating the area. Using a variant of this argument  in the setting of the cell configuration $\cH$ and semi-flowers, we can prove various bounds using ergodic averages of the degree of the cell configurations as in Lemma~\ref{lem:average-diamdeg4}, which leads to the regularity results in this subsection.

Let  $\cT$ be a disk triangulation or an infinite plane triangulation. Let $v$ be a vertex of $\cT$ with degree $n$. The \emph{flower} $F_v\subseteq M(\cT)$ is the union of the vertex $v$, the interior of the edges with an endpoint being $v$, and the interior of the equilateral triangles in $M(\cT)$ that meet at $v$, so $F_v$ has open disk topology. Define the \emph{semiflower} 
$P_v$ to be the collection of points on $M(\cT)$  {for which $v$ is the closest vertex,} 
i.e., $P_v = \{x\in M(\cT):d_M(x,v')\geq d_M(x,v), \forall v'\in\cV\cT\}$. Then $P_v\subseteq F_v$ is an $n$-gon (resp.\ $(n+2)$-gon) if $v$ is an interior (resp.\ boundary) vertex of degree $n$, and $\cup_{v\in\cV\cT} P_v = M(\cT)$. 

In~\cite{GRRiemann}, the authors used the term \emph{half-flower} to describe the shape $\frac{1}{2}F_v$, which is slightly different from our semi-flower. They worked on the interstices between the half-flowers, which always have bounded degree and bounded geometry (i.e., bounded ratio between outradius and inradius), to prove the parabolicity of Riemann surfaces associated with  {random triangulations obtained as the Benjamini-Schramm limit of disk triangulations.} 
In Lemma~\ref{lem:Koebe-distortion} below, we prove that the geometry of semi-flowers can be controlled in terms of the square of the degrees, and we will combine this with the averaging over degrees as in Lemma~\ref{lem:average-diamdeg4} to prove the other results in this subsection. 

\begin{figure}
    \centering
    \includegraphics[scale=0.6]{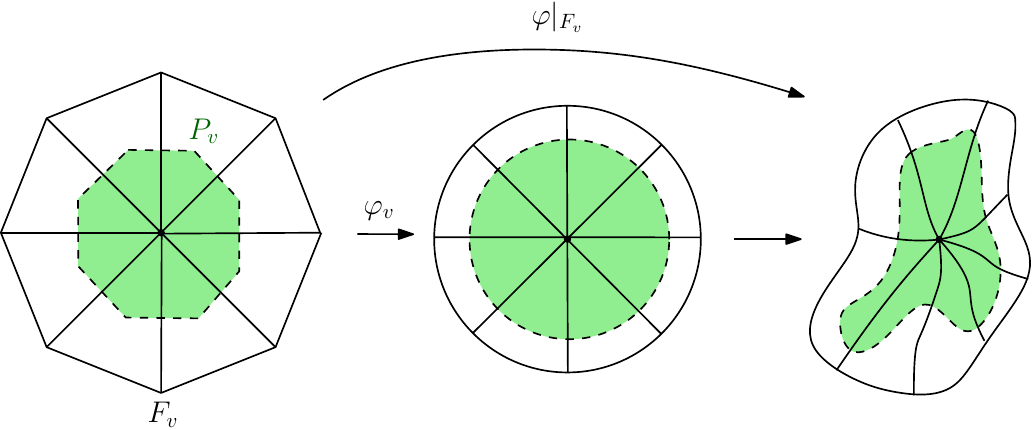}
    \caption{An illustration of the semi-flower $P_v$ (green), its image under the conformal map $\varphi_v:F_v\to \bbD$ and a conformal map $\varphi:M(\cT)\to D\subseteq\bbC$.}
    \label{fig:Koebe-distortion}
\end{figure}

\begin{lemma}\label{lem:Koebe-distortion}
    There exists a universal constant $c_0>0$ such that the following is true. Let  $\cT$ be a disk triangulation or infinite plane triangulation and $\varphi$ be a conformal map from $M(\cT)$ to some domain $D\subseteq\bbC$. Let $v$ be an interior vertex of $\cT$. Then \eqb\label{eq:lem:Koebe-distortion-1} B\Big(\varphi(v);c_0\deg(v)^{-2}\diam(\varphi(P_v))\Big)\subseteq \varphi(P_v).\eqe
    If $\cT$ is a disk triangulation, $D=\bbD$, $v$ is a boundary vertex of $\cT$ such that $0\notin \varphi(F_v)$, then \eqb\label{eq:lem:Koebe-distortion-2} B\Big(\varphi(v);c_0\deg(v)^{-2}\diam(\varphi(P_v))\Big)\cap\bbD\subseteq \varphi(P_v).\eqe
\end{lemma}
\begin{proof}
    Let $n=\deg(v)$. We start with~\eqref{eq:lem:Koebe-distortion-1}. Consider the conformal map $\varphi_0$ from the equilateral triangle $\Delta$ with vertices $0,1,e^{\pi i/3}$ to $\wt\Delta:=\{re^{i\theta}:0<r<1,0<\theta<\pi/3\}$ fixing $0,1,e^{\pi i/3}$. Let $\Delta_0 = \{z\in\Delta:\dist(z,0)\leq\min\{\dist(z,1),\dist(z,e^{i\pi/3}\}\}$. Then there exists constants $a_1,a_2\in (0,1)$ such that $B(0,a_1)\cap\wt\Delta\subseteq \varphi_0(\Delta_0)\subseteq B(0,a_2)\cap \wt\Delta$. Let $\varphi_v:F_v\to\bbD$ be a conformal map sending $v$ to 0. Then using the map $z\mapsto z^{6/n}$ and symmetry, we have $B(0,a_1^{6/n})\subseteq \varphi_v(P_v)\subseteq B(0,a_2^{6/n})$. Now $\varphi\circ\varphi_v^{-1}$ is conformal on $\bbD$, and by the Koebe distortion theorem, we have \eqb\label{eq:pf-lem:Koebe-distortion-1} B\Big(\varphi(v),\frac{a_1^{6/n}}{(1+a_1^{6/n})^2}|(\varphi\circ\varphi_v^{-1})'(0)|\Big)\subseteq \varphi(P_v)\subseteq B\Big(\varphi(v),\frac{a_2^{6/n}}{(1-a_2^{6/n})^2}|(\varphi\circ\varphi_v^{-1})'(0)|\Big).\eqe
    One can check that $\frac{a_2^{6/n}}{(1-a_2^{6/n})^2}\big/ \frac{a_1^{6/n}}{(1+a_1^{6/n})^2}\leq cn^2$ for some constant $c$ depending only on $a_1,a_2$, and this verifies~\eqref{eq:lem:Koebe-distortion-1}. 
    
    For~\eqref{eq:lem:Koebe-distortion-2}, since $M(\cT)$ is simply connected, there are precisely two edges $e$ and $e'$ on $M(\cT)$ with $v$ being one of endpoints which are adjacent to only one equilateral triangle of $M(\cT)$. 
    Consider the conformal map from $\xi:\bbD\to\bbH$ sending $(0,\varphi(v))$ to $(i,0)$, and a conformal map $\varphi_v:F_v\to\bbD\cap\bbH$ sending the edges $e$ and $e'$ to $[-1,0]$ and $[0,1]$, respectively.  Then $g:=\varphi\circ\varphi_v^{-1}$ is a conformal map on $\bbD\cap\bbH$ sending $(-1,1)$ to a segment of $\partial D$, and we extend it to $\bbD$ by setting $g(z) = \frac{1}{\ol{g(\ol z)}}$ for $z\in \ol\bbD$ with $\mathrm{Im }\,z<0$. This gives a holomorphic univalent extension, and the desired claim follows again from Koebe distortion.
\end{proof}

For a square $S$, recall the disk triangulation $\cM(S)\subseteq \cM$ defined in Lemma~\ref{lem:M(S)}, and the Riemann surface $M(S)$ associated with $\cM(S)$. 
For each $k>0$, let $\varphi_k$ be the uniformization map sending $M([-2^k,2^k]^2)$ to $2^k\bbD$ such that the vertex on $M([-2^k,2^k]^2)$ for the cell  containing 0 is sent to 0. Recall the definition of the semiflower in Lemma~\ref{lem:Koebe-distortion}. For each $H\in\cH$, we write $P_H$ for the semiflower for the vertex $H$ on the surface $M(\cH)$.

\begin{lemma}\label{lem:area/length-1}
    There exists a deterministic constant $c\in(0,\infty)$ such that the following is true. Almost surely, for each $\delta\in(0,0.001)$, for large enough $k$, for each square $S\subseteq [-2^k,2^k]^2$ with side length $\delta2^k$, 
    \begin{equation}\label{eq:lem:area/length-1}
        \diam\Big(\bigcup_{H\in\cH(S)} \varphi_k(P_H) \Big)\leq c(-\log\delta)^{-1/2}2^k,
    \end{equation}
    where we use the convention $\varphi_k(P_H)=\varphi_k(P_H\cap M([-2^k,2^k]^2))$ if $P_H\nsubseteq M([-2^k,2^k]^2)$. 
     {Furthermore, for any $\delta>0$,}
    almost surely for large enough $k$, for each $H\in\cH([-2^k,2^k]^2)$, 
     {$\diam(\varphi_k(P_H))<\delta2^k$}.
\end{lemma} 

\begin{figure}
    \centering
    \includegraphics[scale=0.6]{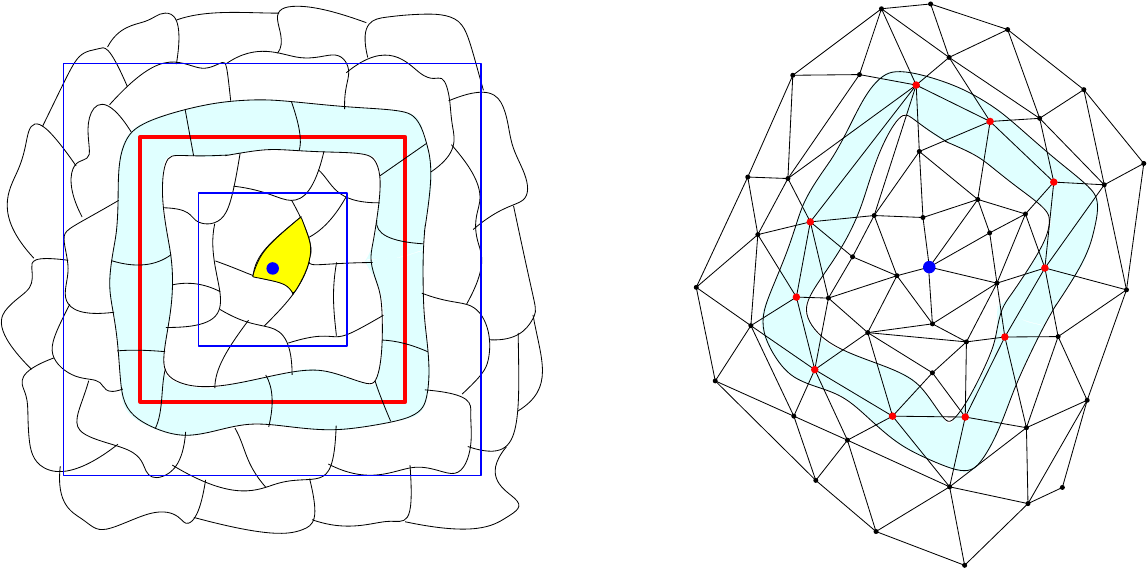}
    \caption{An illustration of the proof of Lemma~\ref{lem:area/length-1}, where the left panel is a part of our cell configuration and the right panel is the corresponding image of the part of the surface $M(\cH)$ under a conformal map. The squares concentric to $S$  of side lengths $r_0$ and $r_1$ are colored in blue on the left panel. In~\eqref{eq:pf-area/length-1-6}, we prove that one can find some $\rho\in[ r_0,r_1]$, such that under a conformal map  the union of the semi-flowers for the cells intersecting the square   concentric to $S$ of side length $\rho$ (shaded in light blue on the right panel) has diameter less than $c(-\log\delta)^{-1/2}2^k$. Further using   the line connectivity property and almost planarity condition from  Proposition~\ref{prop:almost-planarity} gives the claim.}
    \label{fig:ringpf}
\end{figure}

\begin{proof}
   We adapt the argument from~\cite[Equation (2.8)]{HeSchramm},  {which studies maps between two circle packings}. We sum over the diameter of $\varphi(P_H)$ for $H$ intersecting $\partial \mathring{S}^\rho$, integrate over $\rho$, apply a Cauchy-Schwarz inequality, and the terms in this Cauchy-Schwarz inequality can be described in terms of area. One key difference is that the image  $\varphi(P_H)$ is no longer exactly a round disk, and there is no direct way to bound the sum over $\diam(\varphi(P_H))^2$. To circumvent this issue, we use the information from Lemma~\ref{lem:Koebe-distortion} that $\varphi(P_H)$ contains a disk with radius $c\deg(H)^{-2}\diam(\varphi(P_H))$, and add an additional $\deg(H)^{-2}$ term in the Cauchy-Schwarz inequality~\eqref{eq:pf-area/length-1-2} below. Then the sum over  $\diam(\varphi(P_H))^2\deg(H)^{-4}$ can be bounded by the total area of a collection of disjoint disks, and the sum over $\deg(H)^4\diam(H)^2$ can be controlled by the ergodic average from Lemma~\ref{lem:average-diamdeg4}.  
   
   Let $r_0 = \delta2^{k+1}$ and $r_1 = 2^\ell r_0$ such that $\ell\in\bbN$ and $ r_1\in [2^{k-2},2^{k-1})$. For $\rho\in [r_0,r_1]$, set 
   \begin{equation}
       L(\rho) = \sum_{H\in\cH(\partial \mathring{S}^\rho)} \diam(\varphi_k(P_H)).
   \end{equation}
   See also Figure~\ref{fig:ringpf} for an illustration. Then
   \begin{equation}\label{eq:pf-area/length-1-1}
       \inf\{L(\rho):r_0\leq \rho\leq r_1\} \log(r_1/r_0)\leq \int_{r_0}^{r_1} \frac{L(\rho)}{\rho}\,d\rho = \sum_{H\in\cH(\mathring{S}^{r_1}\backslash \mathring{S}^{r_0})}\diam(\varphi_k(P_H)) \int_{r_0}^{r_1} \frac{1}{\rho}\mathds{1}_{H\in \cH(\partial \mathring{S}^\rho)}\,d\rho.
   \end{equation}
   For $H\in \cH(\mathring{S}^{r_1}\backslash \mathring{S}^{r_0})$, let $a_H = \inf\{\rho:H\in \cH(\partial \mathring{S}^\rho)\}$. Then
   $$ \int_{r_0}^{r_1} \frac{1}{\rho}\mathds{1}_{H\in \cH(\partial \mathring{S}^\rho)}\,d\rho\leq \frac{2\diam(H)}{a_H}.$$
   Therefore by the Cauchy-Schwarz inequality, the right side of~\eqref{eq:pf-area/length-1-1} is less  {than} or equal to
   \begin{equation}\label{eq:pf-area/length-1-2}
       \Big( \sum_{H\in\cH(\mathring{S}^{r_1}\backslash \mathring{S}^{r_0})}\diam(\varphi_k(P_H))^2\deg(H)^{-4}\Big) {^{1/2}}\Big(\sum_{H\in\cH(\mathring{S}^{r_1}\backslash \mathring{S}^{r_0})}\frac{4\diam(H)^2}{a_H^2}\deg(H)^4\Big)^{1/2}.
   \end{equation}
   By Lemma~\ref{lem:Koebe-distortion}, for some  constant $c_0$, each set $\varphi_k(P_H)$ contains a disk of radius $c_0\diam(\varphi_k(P_H))\deg(H)^{-2}$  if the vertex for $H$ is in the interior of $M([-2^k,2^k]^2)$ and contains a cap of the form $c_0\diam(\varphi_k(P_H))\deg(H)^{-2}\cap (2^k\bbD)$ if the vertex for $H$ is on the boundary of $M([-2^k,2^k]^2)$. Assuming $c_0<1/2$, $c_0\diam(\varphi_k(P_H))\deg(H)^{-2}\cap (2^k\bbD)$ further contains a disk of radius $c_0\diam(\varphi_k(P_H))\deg(H)^{-2}/2$.  Since the sets $\{\varphi_k(P_H):H\in\cH\}$ have disjoint interiors and are contained within $2^k\bbD$, 
   \begin{equation}\label{eq:pf-area/length-1-3}
        \sum_{H\in\cH(\mathring{S}^{r_1}\backslash \mathring{S}^{r_0})}\diam(\varphi_k(P_H))^2\deg(H)^{-4}\leq c_0^{-2} 4\area(2^k\bbD) = 4c_0^{-2}\pi 2^{2k}.
   \end{equation}
   On the other hand, for each $j=0,...,\ell-1$ and $H\in\cH(\mathring{S}^{2^{j+1}r_0}\backslash \mathring{S}^{2^jr_0})$, by Lemma~\ref{lem:no-macroscopic-cell}, for large enough $k$, $a_H\geq 2^{j-2}r_0$, and by Lemma~\ref{lem:average-diamdeg4}, 
   \begin{equation}\label{eq:pf-area/length-1-4}
       \sum_{H\in\cH(\mathring{S}^{2^{j+1}r_0}\backslash \mathring{S}^{2^jr_0})}\frac{\diam(H)^2}{a_H^2}\deg(H)^4\leq \sum_{H\in\cH(\mathring{S}^{2^{j+1}r_0}\backslash \mathring{S}^{2^jr_0})} (2^{j-2}r_0)^{-2} \diam(H)^2\deg(H)^4\leq 64C,
   \end{equation}
   where $C$ is the deterministic constant in~\eqref{eq:average-diamdeg4}. Summing over $j=0,...,\ell-1$, 
   \begin{equation}\label{eq:pf-area/length-1-5}
       \sum_{H\in\cH(\mathring{S}^{r_1}\backslash \mathring{S}^{r_0})}\frac{\diam(H)^2}{a_H^2}\deg(H)^4 \leq 64C\ell.
   \end{equation}
   Combining~\eqref{eq:pf-area/length-1-3} and~\eqref{eq:pf-area/length-1-5}, we conclude from~\eqref{eq:pf-area/length-1-2} that for some deterministic constant $c_2>0$, almost surely, for all large enough $k$, 
   \begin{equation}\label{eq:pf-area/length-1-6}
        \inf\{L(\rho):r_0\leq \rho\leq r_1\} \leq c_2\ell^{-1/2}2^k.
   \end{equation}
   In particular, there exists $\rho\in [r_0,r_1]$ such that $L(\rho)\leq 4c_2(-\log\delta)^{-1/2}2^k$. 
   
Now consider the embedding $\{\gamma_e^{k+3}:e\in \cE\cM([-2^{k+3},2^{k+3}]^2)\}$ in Proposition~\ref{prop:almost-planarity}.  If $\mathring{S}^{\rho}\subseteq [-7\cdot2^{k-3},7\cdot 2^{k-3}]^2$, using the conditions from Definition~\ref{def:connectedness} and Proposition~\ref{prop:almost-planarity}, for large enough $k$, there exists cells $H_1\sim H_2\sim...\sim H_n\sim H_1$ intersecting $\partial \mathring{S}^{\rho}$, such that there are edges $e_1,...,e_n$ forming a loop with these cells, and in the planar embedding $\{\gamma_e^{k+3}:e\in\cE\cM([-2^{k+3},2^{k+3}]^2)\}$ the edges $\gamma_{e_1}^{k+3},...,\gamma_{e_n}^{k+3}$  form a loop around $\mathring{S}^{1.5\delta2^k}$, and thus encloses all the points $z_H$ for $H\in \cH({S})$. This implies that $\{\varphi_k(P_H):H\in \cH({S})\}$ is separated from infinity by $\cup_{j=1}^n\varphi_k(P_{H_j})$, which together with~\eqref{eq:pf-area/length-1-6} implies~\eqref{eq:lem:area/length-1}. 
   
   If $\mathring{S}^{\rho}\nsubseteq [-7\cdot2^{k-3},7\cdot 2^{k-3}]^2$, since $\rho<r_1\leq 2^{k-1}$,  $\mathring{S}^{\rho}$ is disjoint from the cell $H_0$. Using  Definition~\ref{def:connectedness} and the embedding above, for large enough $k$, there exists cells $H_1\sim H_2\sim...\sim H_n\sim H_1$ either intersecting $\partial \mathring{S}^{\rho}$ or with their vertex on $M(\cH)$ on the boundary of the surface $M([-2^k,2^k]^2)$, such that there are edges $e_1,...,e_n$ forming a loop with these cells and the edges $\gamma_{e_1}^{k+3},...,\gamma_{e_n}^{k+3}$ 
   enclose a loop around all the points $z_H$ where $H\in\cH(S)$ and the vertex for $H$ is on $M([-2^k,2^k]^2)$. This together with~\eqref{eq:pf-area/length-1-6} implies that $\{\varphi_k(P_H):H\in \cH({S})\}$ can be separated from 0 by $\partial(2^k\bbD)$ and some curves of total length $4c_2(-\log\delta)^{-1/2}2^k$. 
   This finishes the proof of~\eqref{eq:lem:area/length-1}. 
   
   In order to get the second assertion of the lemma, let $\wt\delta\in(0,0.01)$ be such that $c(-\log\wt\delta)^{-1/2}<\delta$ and let $S\subseteq [-2^k,2^k]^2$ be a square of side length $\wt\delta2^k$ such that $H\cap S\neq\emptyset$. Then apply~\eqref{eq:lem:area/length-1} with $\wt\delta$ instead of $\delta$. 
\end{proof}

The following is the analog of Lemma~\ref{lem:area/length-1} in the reverse direction.

\begin{lemma}\label{lem:area/length-2}
      There exists a deterministic constant $c\in(0,\infty)$ such that the following is true. Almost surely, for each $\delta\in(0,0.001)$, $\lambda\in(0,1-16\delta)$ and $\e>0$, for large enough $k$, for $z\in B(0;\lambda2^k)$ and $r = \delta2^k$, 
      \begin{equation}\label{eq:lem:area/length-2}
         \diam\Big(\bigcup_{H:\varphi_k(P_H)\cap B(z;r)\neq\emptyset} H \Big)\leq c(\log(1-\lambda)-\log16\delta)^{-1/2}2^k.
     \end{equation}
     \end{lemma}
\begin{proof}
    Fix $z\in\bbC$. Let  $r_0=\delta2^{k+1}$ and $r_1 = 2^\ell r_0$ such that $\ell\in\bbN$ and $ r_1\in [(1-\lambda)2^{k-2},(1-\lambda)2^{k-1})$. For $\rho\in[r_0,r_1]$, we write 
    $$L^*(\rho) = \sum_{H:\varphi_k(P_H)\cap \partial B(z;\rho)\neq\emptyset}\diam(H).$$
    Then similar to   the proof of Lemma~\ref{lem:area/length-1},
     \begin{equation}\label{eq:pf-area/length-2-1}
     \begin{split}
       \inf\{L^*(\rho):&r_0\leq \rho\leq r_1\} \log(r_1/r_0)\leq \int_{r_0}^{r_1} \frac{L^*(\rho)}{\rho}\,d\rho \\&= \sum_{H:\varphi_k(P_H)\cap (B(z;r_1)\backslash B(z;r_0))\neq\emptyset}\diam(H) \int_{r_0}^{r_1} \frac{1}{\rho}\mathds{1}_{\varphi_k(P_H)\cap \partial B(z;\rho)\neq\emptyset}\,d\rho.
       \end{split}
   \end{equation}
   Let $a_H^* = \inf\{\rho>0:\varphi_k(P_H)\cap \partial B(z;\rho)\neq\emptyset\}$, then 
   $$\int_{r_0}^{r_1} \frac{1}{\rho}\mathds{1}_{\varphi_k(P_H)\cap \partial B(z;\rho)\neq\emptyset}\,d\rho\leq \frac{\diam(\varphi_k(P_H))}{a_H^*}.$$
   Therefore by the Cauchy-Schwarz inequality, the right side of~\eqref{eq:pf-area/length-2-1} is less  {than} or equal to
   \begin{equation}\label{eq:pf-area/length-2-2}
      \Big( \sum_{H:\varphi_k(P_H)\cap (B(z;r_1)\backslash B(z;r_0))\neq\emptyset}\diam(H)^2\deg(H)^4\Big)^{1/2}\Big( \sum_{H:\varphi_k(P_H)\cap (B(z;r_1)\backslash B(z;r_0))\neq\emptyset} \frac{\diam(\varphi_k(P_H))^2}{(a_H^*)^2\deg(H)^4}\Big)^{1/2}.
   \end{equation}
   By Lemma~\ref{lem:average-diamdeg4}, for some deterministic constant $C>0$ and large enough $k$, 
   \begin{equation}\label{eq:pf-area/length-2-3}
       \sum_{H:\varphi_k(P_H)\cap (B(z;r_1)\backslash B(z;r_0))\neq\emptyset}\diam(H)^2\deg(H)^4 \leq \sum_{H\in\cH([-2^k,2^k]^2)}\diam(H)^2\deg(H)^4\leq C2^{2k}.
   \end{equation}
   By the last claim of Lemma~\ref{lem:area/length-1}, a.s.\ for large enough $k$,  for $j=0,1,...,\ell-1$ and $H$ such that $\varphi_k(P_H)\cap (B(z;2^{j+1}r_0)\backslash B(z;2^{j}r_0))\neq\emptyset$,  $a^*_H\geq 2^{j-1}r_0$ and $\varphi_k(P_H)\subseteq B(z;2^{j+2}r_0)$. Thus by an application of Lemma~\ref{lem:Koebe-distortion}, for some constant $c_0$, 
   \begin{equation}\label{eq:pf-area/length-2-4}
       \sum_{H:\varphi_k(P_H)\cap (B(z;2^{j+1}r_0)\backslash B(z;2^jr_0))\neq\emptyset} \frac{\diam(\varphi_k(P_H))^2}{(a_H^*)^2\deg(H)^4} \leq c_0^{-2}(2^{j-1}r_0)^{-2}\area(B(z;2^{j+2}r_0)) = 64{\pi}c_0^{-2}.
   \end{equation}
   Summing over $j=0,...,\ell-1$, 
   \begin{equation}\label{eq:pf-area/length-2-5}
        \sum_{H:\varphi_k(P_H)\cap (B(z;r_1)\backslash B(z;r_0))\neq\emptyset} \frac{\diam(\varphi_k(P_H))^2}{(a_H^*)^2\deg(H)^4}\leq 64c_0^{-2}\ell.
   \end{equation}
   Combining~\eqref{eq:pf-area/length-2-3} and~\eqref{eq:pf-area/length-2-5}, we conclude from~\eqref{eq:pf-area/length-2-2} that for some constant $c_2$, for large enough $k$,
   \begin{equation}\label{eq:pf-area/length-2-6}
       \inf\{L^*(\rho):r_0\leq \rho\leq r_1\} \leq c_2(\log(1-\lambda)-\log 16\delta)^{-1/2}2^k.
   \end{equation}
    In particular, there exists $\rho\in [r_0,r_1]$ such that $L^*(\rho)\leq 4c_2(\log(1-\lambda)-\log 16\delta)^{-1/2}2^k$. Recall the points $\{z_H^{k+3}:H\in\cV\cM([-2^{k+3},2^{k+3}]^2)\}$ and the curves $\{\gamma_e^{k+3}:e\in \cE\cM([-2^{k+3},2^{k+3}]^2)\}$ in the planar embedding from Proposition~\ref{prop:almost-planarity}. Note that by our choice of $r_1$, $B(z;\rho)\subseteq 2^k\bbD$. One can find edges $e_1,...,e_n$ forming a loop in $\cH$  such that $\gamma_{e_1}^{k+3},...,\gamma_{e_n}^{k+3}$ make a loop surrounding all the vertices $z_H^{k+3}$ with $\varphi_k(P_H)\cap B(z;r)\neq \emptyset$, and the union of theses edges has diameter less than $20 c_2(\log(1-\lambda)-\log16\delta)^{-1/2}2^k$. Applying ~\eqref{eq:almost-planarity} once more gives~\eqref{eq:lem:area/length-2}.  
\end{proof}

\begin{lemma}\label{lem:area/length-4}
    There exist deterministic constants $\alpha,c_1\in(0,1/2)$ such that the following is true. Almost surely, for large enough $k$, 
    \begin{equation}\label{eq:lem:area/length-4}
      B(0;c_12^k) \subseteq   \varphi_k(M([-\alpha 2^k,\alpha2^k]^2)) \subseteq B(0;2^{k-1}).
    \end{equation}
\end{lemma}

\begin{proof}
    Let $c>0$ be the maximum of the two constants $c$ in Lemma~\ref{lem:area/length-1} and Lemma~\ref{lem:area/length-2}. We pick $\alpha\in(0,0.001)$ such that $c(-\log \alpha)^{-1/2}<1/4$. Then by~\eqref{eq:lem:area/length-1} for $S=[-\alpha 2^k,\alpha2^k]^2$, $$\bigcup_{H\in\cH([-\alpha 2^k,\alpha2^k]^2)} \varphi_k(P_H) \subseteq B(0;2^{k-1}).$$
    On the other hand, if we pick $c_1\in(0,0.001)$ such that $c(-\log c_1)^{-1/2}<\alpha/16$, then by~\eqref{eq:lem:area/length-2} for $B(z;r)= B(0;c_12^k)$, 
    $$\bigcup_{H:\varphi_k(P_H)\cap B(0,c_12^k)\neq\emptyset} H \subseteq [-\alpha2^{k-2},\alpha 2^{k-2}]^2.$$
    This along with Lemma~\ref{lem:M(S)} verifies~\eqref{eq:lem:area/length-4}.
\end{proof}

 With slight abuse of language, we write 0 for the vertex on $M(\cH)$ representing the cell $H_0$. Consider a uniformization map $\varphi:M(\cH)\to\bbC$ sending 0 to 0.   Let $\alpha$ be the constant as in Lemma~\ref{lem:area/length-4}. We define
\begin{equation}\label{eq:def-d_k-alpha}
    d_k:=\diam\big( \varphi(M([-\alpha 2^k,\alpha2^k]^2) ) \big).
\end{equation}

\begin{lemma}\label{lem:area/length-6}
  Let $a\in(0,1)$.  There exists a deterministic $\ell>0$ such that almost surely, for large enough $k$, $d_{k-\ell}<ad_k$. 
\end{lemma}

\begin{proof}
    The proof is almost identical to that of Lemma~\ref{lem:area/length-1}. Fix $\ell\in\bbN$, $r_0=\alpha 2^{k-\ell+1}$, $r_1=\alpha 2^{k-2}$, and for $\rho\in[r_0,r_1]$, set 
     \begin{equation}
       L(\rho) = \sum_{H\in\cH(\partial [-\rho,\rho]^2)} \diam(\varphi(P_H)).
   \end{equation}
    Then
   \begin{equation}\label{eq:pf-area/length-6-1}
   \begin{split}
       &\inf\{L(\rho):r_0\leq \rho\leq r_1\} \log(r_1/r_0)\leq \int_{r_0}^{r_1} \frac{L(\rho)}{\rho}\,d\rho \leq \sum_{H\in\cH([-r_1,r_1]^2\backslash [-r_0,r_0]^2)}\diam(\varphi(P_H))\frac{\diam(H)}{a_H} \\ &\leq \Big(\sum_{H\in\cH([-r_1,r_1]^2\backslash [-r_0,r_0]^2)}\diam(\varphi(P_H))^2\deg(H)^{-4}\Big)^{1/2}\Big(\sum_{H\in\cH([-r_1,r_1]^2\backslash [-r_0,r_0]^2)}\frac{\diam(H)^2}{a_H^2}\deg(H)^4\Big)^{1/2}
       \end{split}
   \end{equation}
   where $a_H = \inf\{\rho>0:H\in\cH(\partial[-\rho,\rho]^2)\}$.
   By Lemma~\ref{lem:M(S)} and the definition of $d_k$, $\bigcup_{H\in\cH([-r_1,r_1]^2)}\varphi(P_H)\subseteq B(0;d_k)$.  By Lemma~\ref{lem:Koebe-distortion},    for some constant $c_0$, 
   \eqb\label{eq:pf-area/length-6-2}\sum_{H\in\cH([-r_1,r_1]^2\backslash [-r_0,r_0]^2)}\diam(\varphi(P_H))^2\deg(H)^{-4}\leq \pi c_0^{-2}d_k^2.\eqe
   By a similar argument as in~\eqref{eq:pf-area/length-1-5}, for some deterministic constant $C>0$,
   \begin{equation}\label{eq:pf-area/length-6-3}
       \sum_{H\in\cH([-r_1,r_1]^2\backslash [-r_0,r_0]^2)}\frac{\diam(H)^2}{a_H^2}\deg(H)^4\leq 64C\ell.
   \end{equation}
   Thus by~\eqref{eq:pf-area/length-6-1}, ~\eqref{eq:pf-area/length-6-2} and~\eqref{eq:pf-area/length-6-3}, for some deterministic constant $c_2>0$, almost surely for large enough $k$,
   \begin{equation}
       \inf\{L(\rho):r_0\leq \rho\leq r_1\} \leq c_2\ell^{-1/2}d_k.
   \end{equation}
   We pick $\ell$ such that $c_2\ell^{-1/2}<a/4$. The lemma then follows using the same argument as in the last paragraph of the proof of Lemma~\ref{lem:area/length-1}.
\end{proof}

\begin{lemma}\label{lem:area/length-5}
    There exist  deterministic constants $\alpha,c_2\in(0,1)$ and $\ell>0$ such that the following is true. Almost surely, for large enough $k$,
    \begin{equation}\label{eq:lem:area/length-5}
    \varphi(M([-\alpha 2^{k-\ell},\alpha2^{k-\ell}]^2)) \subseteq B\big(0; c_2 d_k  \big)\subseteq  \varphi(M([-\alpha 2^k,\alpha2^k]^2)) . 
    \end{equation}
\end{lemma}
\begin{proof}
    The second half of~\eqref{eq:lem:area/length-5} is immediate by applying the Koebe distortion theorem to the conformal map $\varphi\circ \varphi_k^{-1}$ along with  {the first half of}~\eqref{eq:lem:area/length-4}. The first half of~\eqref{eq:lem:area/length-5} then follows from Lemma \ref{lem:area/length-6} with $a=c_2$.
\end{proof}

\begin{lemma}\label{lem:parabolic}
    Almost surely, $\lim_{k\to\infty} d_k = \infty$. Furthermore, $\varphi(M(\cH))=\bbC$ and thus $M(\cH)$ is parabolic.
\end{lemma}
\begin{proof}
    The first claim is immediate from Lemma~\ref{lem:area/length-6}. The second claim follows since, by Lemma~\ref{lem:area/length-5}, for some constant $c_2>0$, for all large enough $k$, $B(0;c_2d_k)\subseteq \varphi(M(\cH))$.
\end{proof}

\begin{lemma}\label{lem:cell-polynomial-growth}
    There is a deterministic constant $\beta\geq1$ 
    such that the following is true. Almost surely, for some random constant $\wt C$ depending on $\cH$ and $\varphi$, for every large enough $r$, \eqb\label{eq:lem:cell-polynomial-growth}\{H:\varphi(P_H)\cap B(0;r)\neq\emptyset \}\subseteq B(0;\wt Cr^\beta).\eqe
\end{lemma}

\begin{proof}
    Let $\ell$ be the constant in Lemma~\ref{lem:area/length-6} for $a=1/2$. Then there exists some random constant $\wt C_0$, such that $d_k\geq \wt C_02^{k/\ell}$. By Lemma~\ref{lem:area/length-5},  for some constant $c_2>0$ and $k$ large enough, if $\varphi(P_H)\cap B(0;c_2\wt C_02^{k/\ell})\neq\emptyset$, then $H\in\cH([-\alpha 2^k,\alpha2^k]^2)$ or $H$ is adjacent to some $H'\in \cH([-\alpha 2^k,\alpha2^k]^2)$. In particular, by Lemma~\ref{lem:no-macroscopic-cell}, $H\subseteq [-2^k,2^k]^2$. Now for each $r>0$ large enough, choose $k$ such that $c_2\wt C_02^{k/\ell}<r<c_2\wt C_02^{(k+1)/\ell}$. If $\varphi(P_H)\cap B(0;r)\neq\emptyset$, then $H\subseteq [-2^{k+1},2^{k+1}]^2\subseteq B(0;4(c_2\wt C_0)^{-\ell}r^\ell)$. This verifies the claim for $\beta=\ell$.
\end{proof}

\begin{remark}\label{rmk-max-diam}
    We note that the proofs of Lemmas~\ref{lem:area/length-1}-\ref{lem:area/length-6} can be adapted to the setting of circle packings. Therefore, by repeating the arguments in this subsection, it might be possible to give  
    an alternative proof of Proposition~\ref{prop:max-diam-0-cone} without using the vertex extremal length in~\cite{GDN19VEL}. 
\end{remark}

\subsection{Existence of the harmonic corrector}\label{subsec:GMS2.3}
Let $\cH$ be a random cell configuration satisfying all the conditions in Theorem~\ref{thm:invariance-uniformization} and let $M(\cH)$ be the Riemann surface constructed from the map associated with $\cH$. Recall the function $\phi_0$ defined in Section~\ref{subsec:ergodic-cell-system}.  {In this section we will define a harmonic function $\phi_\infty:M(\cH)\to\bbC$ as the limit of functions $\phi_m:M(\cH)\to\bbC$ for $m\in\bbN$, which are defined as certain harmonic extensions of the function $\phi_0:M(\cH)\to\bbC$. This idea is similar to \cite[Section 2.3]{GMSinvariance} and the key difference is that we work with (continuum) harmonic functions on $M(\cH)$ instead of discrete harmonic functions on the graph $\cH$.} In the remainder of this section, we write $\mathcal{I}$ for the $\sigma$-algebra in Lemma~\ref{lem:ergodic-cell-system}.

Recall the definition $\wh S_m^z$ in~\eqref{eq:def-S_m^z} and define the collection of squares $${\mathfrak{S}}_m :=\{\wh S_m^z:z\in\bbC \text{ and $z$ is not on the boundary of any square in $\cD$} \}.$$
If $z'$ is in the interior of $\wh S_m^z$, then $\wh S_m^z = \wh S_m^{z'}$. It follows that the squares in $\mathfrak{S}_m$ intersect only along their boundaries, and they  together cover all of $\bbC$. In particular, the preimages $\{\phi_0^{-1}(S):S\in \mathfrak{S}_m\}$ cover $M(\cH)$. Note that $\phi_0^{-1}(S)$ might not necessarily be connected, and their boundaries consist of line segments.

For $m>0$, we define the function $\phi_m {: M(\cH)\to\bbC}$ as 
 {follows}. For each $S\in \mathfrak{S}_m$ and $\phi_0^{-1}(S)$, we define $\phi_m$ to be equal to $\phi_0$ on $\partial \phi_0^{-1}(S)$, and define $\phi_m$ to be the harmonic extension of $\phi_0$ in the interior of $\phi_0^{-1}(S)$. Equivalently, one can take {a} uniformization map $\varphi:M(\cH)\to\bbC$ and work on the ordinary harmonic extension of $\phi_0\circ\varphi|_{\partial (\varphi\circ\phi_0^{-1})(S)}$ in $(\varphi\circ\phi_0^{-1})(S)$.

The main goal of this subsection is to prove the following. {We do not know at this point whether $\phi_\infty(x)$ is a measurable function of $\cH$ and $\cD$ due to the randomness of the points $\wt c(H)$. In the following proposition we work with $\bbE\big[\phi_\infty(x)|\cH,\cD \big]$, where the randomness that we average over in the expectation is coming from the randomness of the points $\wt c(H)$.} Recall that we write $0$ for the vertex on $M(\cH)$ representing the cell $H_0$.
\begin{proposition}\label{prop:harmonic-exist}
 There exists a harmonic function $\phi_\infty:M(\cH)\to\bbC$  such that $\bbE\big[|\phi_\infty(x)||\cH,\cD \big]<\infty$ a.s.\   for each $x\in M(\cH)$ and
    \begin{equation}\label{eq:prop:harmonic-exist}
        \phi_m - \phi_m(0) \to \phi_\infty \  \ \text{in probability as $m\to\infty$},
    \end{equation}
    with respect to the topology of uniform convergence on compact subsets of $M(\cH)$, i.e., for each compact set $\cA\subseteq M(\cH)$ (chosen in some measurable manner) and each $\e>0$, 
    $$\lim_{m\to\infty} \bbP\big[\sup_{z\in \cA}|\phi_m(z)-\phi_m(0)-\phi_\infty(z)|>\e\big] = 0.$$
\end{proposition}

To prove Proposition~\ref{prop:harmonic-exist}, we use the similar outline as in~\cite[Section 2.3]{GMSinvariance}. However, instead of working on discrete Dirichlet energy and discrete harmonic functions $\cH\to\bbC$ as in~\cite{GMSinvariance}, we work with continuous harmonic functions and their Dirichlet energy. Our particular choice of $\phi_0$ above will play an important role in order for the argument to work; it will be important that the images of vertices on $M(\cH)$ under $\phi_0$ are close to the corresponding cells in $\cH$ and that we can bound the Dirichlet energy per area for $\phi_0$ in terms of $\diam(H)^2\deg(H)$ via Lemma~\ref{lem:average-diamdeg4} (see Lemma~\ref{lem:energy-phi0} below). Different from the setup of \cite{GMSinvariance}, where the function $\phi_0$ is defined on $\cH$, our function $\phi_0$ and its harmonic extensions $\phi_m$ are defined on the surface $M(\cH)$. 
First we gather some topological properties of the map $\phi_0$.  Then we show that  $\phi_0$ (and hence $\phi_m$) has uniformly bounded Dirichlet energy on $\phi_0^{-1}(S_n)$ when rescaled by $\frac{1}{|S_n|^2}$. Then we use the orthogonality of the increments $\phi_m-\phi_{m-1}$ with respect to the Dirichlet inner product  {(see the proof of Lemma \ref{lem:dirichlet-B})} to show that for large enough $m$, the expected ``Dirichlet energy per unit area"  is small
uniformly over all $m'\geq m$. Together with the standard properties of ordinary harmonic functions, this will allow us to show that the functions $\phi_m-\phi_m(0)$ are Cauchy in probability, and thereby establish Proposition~\ref{prop:harmonic-exist}.

For a function $f:M(\cH)\to\bbC$ and some domain $D\subseteq M(\cH)$, using the conformal invariance of Dirichlet energy, we set
$$\mathrm{Energy}(f;D) = \int_D |\nabla f(x)|^2dx = \int_{\varphi(D)}|\nabla(f\circ\varphi^{-1})(z)|^2dz,$$
where $\varphi:M(\cH)\to\bbC$ is some conformal map.

We start with the following simple lemma.
\begin{lemma}\label{lem:Dirichlet-linear}
    For any linear map $f$ from an equilateral triangle $\Delta$ with unit side length to another triangle $\Delta'$ with side length $\ell_1,\ell_2,\ell_3$, 
    $$\int_\Delta |\nabla f(z)|^2dz\leq 3(\ell_1^2+\ell_2^2+\ell_3^2).$$
\end{lemma}

\begin{proof}
    Without loss of generality, by translation and rotation, assume the vertices of $\Delta'$ are positioned at $(0,0), (x_1,0),(x_2,y_2)$. One can check that the matrix for the linear map $f$ is given by 
    $$\begin{pmatrix}
        x_1 & -\frac{1}{\sqrt{3}}x_1+\frac{2}{\sqrt{3}}x_2 \\ 0 & \frac{2}{\sqrt{3}}y_2
    \end{pmatrix}.$$
    Then $|\nabla f|^2\leq 3(x_1^2+x_2^2+y_2^2) \leq 3\ell_1^2+3\ell_2^2+3\ell_3^2$.
\end{proof}

\begin{lemma}\label{lem:phi_0-bound}
    Almost surely, for large enough $k$, the following holds. If $H_1,H_2,H_3\in\cH$ are adjacent to each other, such that the triangle formed by $\wt c(H_1),\wt c(H_2),\wt c(H_3)$ 
    intersects $S_k$, then $H_1,H_2,H_3$ are contained within the square $S_k'$ concentric to $S_k$ with side length $1.5|S_k|$. 
\end{lemma}

\begin{proof}
    This is a direct consequence of Lemma~\ref{lem:no-macroscopic-cell}.  
\end{proof}

We call a closed set $D\subseteq M(\cH)$ a polygon if $\partial D$ is a simple closed curve and can be written as a union of a finite number of line segments on faces of $M(\cH)$.
\begin{lemma}\label{lem:topology-cell-square}
    Almost surely, we have
    \begin{equation}\label{eq:topology-cell-square}
        \{\wt c(H):H\in\cH\}\cap \bigg(\bigcup_{S\in\cD}\partial S\bigg) = \emptyset, \qquad  \{\wt c(H):H\in\cH\}\cap \bigg(\bigcup_{H_1,H_2\in\cH,H_1\neq H_2}(H_1\cap H_2)\bigg) = \emptyset.
    \end{equation}
    Furthermore, for each $S\in\cD$, $\phi_0^{-1}(S)$ can be written as a union of a finite number of polygons on $M(\cH)$. For $S_1,S_2\in\cD$ with disjoint interiors, $\phi_0^{-1}(S_1)\cap\phi_0^{-1}(S_2)$ has zero area on $M(\cH)$. Similarly, for {distinct} $H_1,H_2\in\cH$, $\phi_0^{-1}(H_1)\cap\phi_0^{-1}(H_2)$ has zero area on $M(\cH)$.
\end{lemma}

\begin{proof}
   The first claim is straightforward since the sets  {$\partial S$ (for $S\in\cD$) and $H_1\cap H_2$ (for $H_1,H_2\in\cH$ distinct)} in~\eqref{eq:topology-cell-square}  have zero Lebesgue measure  {and since the points $\wt c(H)$ are sampled uniformly from sets of positive Lebesgue measure conditioned on the former sets}. The other claim{s} follow  from~\eqref{eq:topology-cell-square}, together with the fact that almost surely, for each $H_1,H_2,H_3$, the triangle formed by $\wt c(H_1),\wt c(H_2)$ and $\wt c(H_3)$ is nondegenerate.
\end{proof}

\begin{lemma}\label{lem:energy-phi0}
    For each $n\geq m\geq 1$, we have
    \begin{equation}\label{eq:energy-phi0-1}
        \mathrm{Energy}(\phi_m;\phi_0^{-1}(\wh S_n))\leq \mathrm{Energy}(\phi_0;\phi_0^{-1}(\wh S_n)).
    \end{equation}
    Furthermore, there is a deterministic constant $C>0$ such that a.s.
    \begin{equation}\label{eq:energy-phi0-2}
        \limsup_{k\to\infty} \frac{1}{|\wh S_k|^2}\mathrm{Energy}(\phi_0;\phi_0^{-1}(\wh S_k))<C.
    \end{equation}
\end{lemma}

\begin{proof}
    The relation~\eqref{eq:energy-phi0-1} is an immediate consequence of the fact that  harmonic functions minimize Dirichlet energy subject to specified boundary data. Suppose $H_1,H_2,H_3\in\cH$ are adjacent such that the triangle $\Delta_{H_1,H_2,H_3}'$ formed by $\wt c(H_1),\wt c(H_2),\wt c(H_3)$ 
    intersects $\wh S_k$. Let $\wh S_k'$ be the square concentric to $\wh S_k$ with side length $1.5|\wh S_k|$. Then by Lemma~\ref{lem:phi_0-bound}, $H_1,H_2,H_3\subseteq \wh S_k'$  {for sufficiently large $k$}. Furthermore, if we write $\Delta_{H_1,H_2,H_3}$ for the triangle in $M(\cH)$,  by Lemma~\ref{lem:Dirichlet-linear}, 
\begin{equation}\label{eq:pf-energy-phi0-1}
    \begin{split}
    \mathrm{Energy}(\phi_0;\Delta_{H_1,H_2,H_3})&\leq 3\big((\diam(H_1)+\diam(H_2))^2+(\diam(H_1)+\diam(H_3))^2+(\diam(H_2)+\diam(H_3))^2\big)\\&\leq 12(\diam(H_1)^2+\diam(H_2)^2+\diam(H_3)^2).
    \end{split}
    \end{equation}
    Now we sum  ~\eqref{eq:pf-energy-phi0-1} over all possible choices of $H_1,H_2,H_3\subseteq \wh S_k'$ and note that each term $\diam(H)^2$  {appears} 
    at most ${12}\deg(H)$ times in the summation,  {which gives}
    \begin{equation*}
        \mathrm{Energy}(\phi_0;\phi_0^{-1}(\wh S_k))\leq 12\sum_{H\in\cH(\wh S_k')} \diam(H)^2\deg(H).
    \end{equation*}
    The relation~\eqref{eq:energy-phi0-2} then follows from Lemma~\ref{lem:average-diamdeg4}.
\end{proof}

Following the notations of~\cite{GMSinvariance}, for $z\in\bbC$, we write $H_z$ for the a.s.\ unique cell in $\cH$ containing $z$, and for $m_1<m_2$, we set 
\eqb\label{eq:def-SEm1m2}\mathrm{SE}_{m_1,m_2}^z = \frac{\int_{\phi_0^{-1}(H_z\cap \wh S_{m_2}^z)}|\nabla\phi_{m_2}(w)-\nabla\phi_{m_1}(w)|^2\,dw}{\area(H_z\cap \wh S_{m_2}^z)},\eqe
and we write $\mathrm{SE}_{m_1,m_2}$ for $\mathrm{SE}_{m_1,m_2}^0$.

\begin{lemma}\label{lem:dirichlet-A}
    For each fixed $m_1<m_2$, almost surely
    \begin{equation}
       \lim_{n\to\infty}\frac{1}{|\wh S_n|^2}\bbE\Big[\mathrm{Energy}(\phi_{m_2}-\phi_{m_1};\phi_0^{-1}(\wh S_n))\big|\cH,\cD\Big] = \bbE[\mathrm{SE}_{m_1,m_2}|\mathcal{I}],
    \end{equation}
    and this expectation is bounded above by a finite constant which does not depend on $m_1,m_2$.
\end{lemma}

\begin{proof}
Observe that for $n\geq m_2$, 
Lemma~\ref{lem:topology-cell-square}  {gives that} almost surely, 
\eqb\label{eq:pf-energy-phi0-3}\frac{1}{|\wh S_n|^2}\int_{\wh S_n}\mathrm{SE}_{m_1,m_2}^z \,dz = \frac{1}{|\wh S_n|^2}\int_{\phi_0^{-1}(\wh S_n)}|\nabla\phi_{m_2}(w)-\nabla\phi_{m_1}(w)|^2\,dw. \eqe
If we dilate the cell configuration $\cH$ and the dyadic system $\cD$ by some constant $C$, then $\phi_0^{-1}(H_0\cap \wh S_{m_2})$ is invariant, $\phi_m$ is multiplied by $C$, and $\area(H_0\cap \wh S_{m_2})$ is multiplied by $C^2$, and thus $\bbE[\mathrm{SE}_{m_1,m_2}|\cH,\cD]$ is scale invariant. By Lemma~\ref{lem:energy-phi0} and a triangle inequality, there is a deterministic constant $C_0$ such that ~\eqref{eq:pf-energy-phi0-3} is less than $C_0$ for all $n$ large enough. The lemma then follows by taking conditional expectation of~\eqref{eq:pf-energy-phi0-3} over $(\cH,\cD)$ and setting $F(\cH,\cD)=\bbE[\mathrm{SE}_{m_1,m_2}|\cH,\cD]$ in Lemma~\ref{lem:ergodic-cell-system} (observe that $F(\cH-z,\cD-z) = \bbE[\mathrm{SE}_{m_1,m_2}^z|\cH,\cD]$ for each $z\in\bbC$ since given $\cH,\cD$, up to translation, the $\phi_m$-functions for $(\cH,\cD)$ and $(\cH-z,\cD-z)$ have the same law).
\end{proof}

\begin{lemma}\label{lem:dirichlet-B}
    For each $\e>0$, there exists $m_\e\in\bbN$ such that \begin{equation}
        \bbE[\mathrm{SE}_{m_\e,m} ]<\e, \  \text{ for all }m\geq m_\e.
    \end{equation}
\end{lemma}

\begin{proof}
    For each square $S\in\mathfrak{S}_m$, and $m'\leq m$, $\phi_{m'}-\phi_{m'-1}$ vanishes on $\partial\phi_0^{-1}(S)$ following the definition of $\phi_m$, and $\phi_m$ is harmonic in the interior of $ \phi_0^{-1}(S)$. Thus using integration by parts, almost surely, for each $S\in \mathfrak{S}_m$,
    $$\int_{\phi_0^{-1}(S)} \nabla \phi_m(w)\cdot (\nabla\phi_{m'}(w)-\nabla\phi_{m'-1}(w))\,dw = 0.$$
    
    By a summation over the squares, one can check that for $n\geq m\geq m'$, almost surely
    $$\frac{1}{|\wh S_n|^2}\mathrm{Energy}(\phi_{m}-\phi_{m'}; \phi_0^{-1}(\wh S_n)) = \sum_{j=m'+1}^m \frac{1}{|\wh S_n|^2}\mathrm{Energy}(\phi_j-\phi_{j-1}; \phi_0^{-1}(\wh S_n)). $$
     Now taking $n\to\infty$, by Lemma~\ref{lem:dirichlet-A}, we have 
     \eqb\label{eq:lem:dirichlet-B-66} \bbE[\mathrm{SE}_{m',m}|\mathcal I] = \sum_{j=m'+1}^m \bbE[\mathrm{SE}_{j-1,j}|\mathcal I],\eqe
     and the sum is uniformly bounded from above by the last part of Lemma~\ref{lem:dirichlet-A}. In particular, there exists $m_\e$ such that $\sum_{j=m_\e}^\infty \bbE[\mathrm{SE}_{j-1,j}]<\e$, and thus $\bbE[\mathrm{SE}_{m_\e,m}]<\e$ for all $m\geq m_\e$.
\end{proof}

\begin{proof}[Proof of Proposition~\ref{prop:harmonic-exist}]
  Fix $\ell\in\bbN$.   Let $\e>0$ and $m_\e$ be as in Lemma~\ref{lem:dirichlet-B}. Without loss of generality assume $\ell<m_\e$.  Then almost surely, using Lemma~\ref{lem:topology-cell-square}, for $m\geq m_\e$,
    \eqb\begin{split}\int_{\wh S_\ell}& \mathrm{SE}_{m_\e,m}^z\,dz = \sum_{H\in\cH(\wh S_\ell)} \frac{\area(H\cap \wh S_\ell)}{\area(H\cap \wh S_m)}\int_{\phi_0^{-1}(H\cap\wh S_m)} |\nabla\phi_m(w)-\nabla\phi_{m_\e}(w)|^2\,dw 
    \\&\geq \sum_{H\in\cH(\wh S_\ell)\backslash\cH(\partial \wh S_\ell)} \int_{\phi_0^{-1}(H)}|\nabla\phi_m(w)-\nabla\phi_{m_\e}(w)|^2\,dw = \int_{\phi_0^{-1}(\cH(\wh S_\ell)\backslash\cH(\partial \wh S_\ell))}|\nabla\phi_m(w)-\nabla\phi_{m_\e}(w)|^2\,dw.
    \end{split}\eqe 
    On the other hand, by Lemma~\ref{lem:dyadic-resample} and the definition of $m_\e$,
    \begin{equation}
       \bbE\Big[ \frac{1}{|\wh S_\ell|^2}\int_{\wh S_\ell} \mathrm{SE}_{m_\e,m}^z\,dz \Big] = \bbE[\mathrm{SE}_{m_\e,m}]<\e.
    \end{equation}
    Let $\varphi:M(\cH)\to\bbC$ be a conformal map sending the point 0 on $M(\cH)$  {(which corresponds to the a.s.\ unique origin-containing cell of $\cH$)} to 0 chosen in a way measurable with respect to $\cH$. Using the conformal invariance of Dirichlet energy, for each $m \geq m_\e$,
    \begin{equation}\label{eq:pf-prop:harmonic-exist-4}
        \bbE\Big[\frac{1}{|\wh S_\ell|^2} \int_{{\wt D_\ell}}|\nabla(\phi_m\circ\varphi^{-1})(z)-\nabla(\phi_{m_\e}\circ\varphi^{-1})(z)|^2 \,dz\Big]<\e,
    \end{equation}
  {where} $\wt D_\ell = \varphi(\phi_0^{-1}(\cH(\wh S_\ell)\backslash\cH(\partial \wh S_\ell)))$.    Observe that $\phi_m\circ\varphi^{-1}$ and $\phi_{m_\e}\circ\varphi^{-1}$ are ordinary harmonic functions on $\wt D_\ell$ since we assumed that $m_\e>\ell$. Let $\delta\in(0,1)$, $\mathring{D}_\ell^\delta$ be the set of points $z_0$ such that for each point $z$ on the line segment $L$ containing 0 and $z_0$, $B(z;\delta)\subseteq \wt D_\ell$. For any harmonic function $f$ on  $\wt D_\ell$, $\nabla f$ is also harmonic, and thus using the mean value property of harmonic functions, for each $z_0=r_0e^{i\theta_0}\in \mathring{D}_\ell^\delta$,
 \begin{equation}\label{eq:pf-prop:harmonic-exist-5}
 \begin{split}
     \big|f(z{_0})-f(0)\big| &= \Big|\int_0^{r_0} \bigg(\frac{\partial f}{\partial x}(re^{i\theta_0})\cos \theta_0+\frac{\partial f}{\partial y}(re^{i\theta_0})\sin \theta_0 \bigg)\,
      {dr}\Big|\\
     &= \frac{1}{\pi\delta^2}\Big|\int_0^{r_0} \int_{\delta\bbD} \big(\frac{\partial f}{\partial x}(re^{i\theta_0}+w)\cos\theta_0 + \frac{\partial f}{\partial y}(re^{i\theta_0}+w)\sin\theta_0\big)\,dw {\,dr}\Big| \\
     &\leq \frac{10}{\delta^2}\int_{ {\wt D}_\ell} |\nabla f(w)|\,dw \leq \frac{10}{\delta^2}\, \area( {\wt D}_\ell)^{1/2} \big(\int_{ {\wt D}_\ell} |\nabla f(w)|^2\,dw \big)^{1/2}.
     \end{split}
 \end{equation}
 Combined with~\eqref{eq:pf-prop:harmonic-exist-4} we see that the sequence $\{\phi_m\circ\varphi^{-1}-\phi_m(0)\}$ is Cauchy in probability with respect to the topology of uniform convergence on $\mathring{D}_\ell^\delta$. Since $\cH$ is locally finite, by varying $\ell$ and $\delta$, one can find a function $\phi_\infty$ on $M(\cH)$ such that $\{\phi_m\circ\varphi^{-1}-\phi_m(0)\}$ converges to $\phi_\infty\circ \varphi^{-1}$ in probability with respect to the topology of uniform convergence  on compact subsets of $\varphi(M(\cH))$. This verifies~\eqref{eq:prop:harmonic-exist}. The function $\phi_\infty$ is a.s.\ harmonic since each $\phi_m$ is harmonic on $\phi_0^{-1}(\cH(\wh S_m)\backslash\cH(\partial \wh S_m))$, and the uniform limit of harmonic functions is harmonic. To prove the claim $\bbE\big[|\phi_\infty(x)||\cH,\cD \big]<\infty$ for each $x\in M(\cH)$,  
 since uniform convergence of harmonic functions implies uniform convergence of their gradients as well, Lemma~\ref{lem:dirichlet-A}, Fatou's lemma and the above arguments imply that 
 for some constant $C$,
  \begin{equation}\label{eq:pf-prop:harmonic-exist-14}
     \bbE[Z_\ell]:=   \bbE\Big[\frac{1}{|\wh S_\ell|^2} \int_{{\wt D_\ell}}|\nabla(\phi_\infty\circ\varphi^{-1})(z)-\nabla(\phi_{\ell}\circ\varphi^{-1})(z)|^2 \,dz\Big]<C.
    \end{equation}
    For each $R>0$, we can fix some sufficiently large $\ell$ that only depends on $(\cD,\cH)$ such that   $B(0;2R)\subseteq \wt D_\ell$. Then the argument as in~\eqref{eq:pf-prop:harmonic-exist-5} implies that $\sup_{z\in B(0;R)}|\phi_\infty\circ\varphi^{-1}(z)-\phi_{\ell}\circ\varphi^{-1}(z)+\phi_{\ell}\circ\varphi^{-1}(0)|^2\leq C_0(1+Z_\ell)|\wh S_\ell|^2$, while for large $\ell$, $|\phi_{\ell}\circ\varphi^{-1}(z)-\phi_{\ell}\circ\varphi^{-1}(0)|$ is also bounded by $8|\wh S_\ell|$  thanks to the maximum principle, the definition of $\phi_\ell$ and Lemma~\ref{lem:no-macroscopic-cell}. Therefore for some large $\ell$ depending only on $(\cD,\cH)$, $$\bbE\Big[\sup_{z\in B(0;R)}|\phi_\infty\circ\varphi^{-1}(z)|\Big|\cH,\cD\Big]\leq C|\wh S_\ell|(1+\bbE[Z_\ell|\cG,\cD])<\infty.$$ This finishes the proof.
\end{proof}

\subsection{Sublinearity of the corrector}\label{subsec:GMS2.4}

 {The aim of this subsection is to prove Proposition~\ref{prop:harmonic-sublinear} below.  
 We start with the fact that the Dirichlet energy per unit area of $\phi_m-\phi_\infty$ is small when $m$ is large (Lemma \ref{lem:dirichlet-D}), and for fixed $m$, the maximal value of $|\phi_m(z)-\phi_0(z)|$  {on $M([-2^k,2^k])$} has order $o_k(2^k)$ as $k\to\infty$  {(Lemma \ref{lem:sublinear-A})}. Then we transfer the bound on the Dirichlet energy per unit area of $\phi_m-\phi_\infty$ to a pointwise bound for $|\phi_0-\phi_\infty|$. The structure of the proof is similar as in \cite[Section 2.4]{GMSinvariance}, but we need to use a different setup when we conclude the proof of Proposition~\ref{prop:harmonic-sublinear} where we 
work with the uniformization map $\varphi_k:M([-2^k,2^k])\to 2^k\bbD$ as in Lemmas~\ref{lem:area/length-1} and~\ref{lem:area/length-2}, and we also need to use regularity results from Section~\ref{subsec:regularity}. We refer to the paragraph above the latter proof for a proof outline.}

\begin{proposition}\label{prop:harmonic-sublinear}
    Almost surely, in the setting of Proposition~\ref{prop:harmonic-exist},
    \begin{equation}\label{eq:prop:harmonic-sublinear}
        \lim_{k\to\infty} \frac{1}{2^k}\sup_{x\in M([-2^k,2^k]^2)}|\phi_0(x) - \bbE[\phi_\infty(x)|\cH,\cD]|=0.
    \end{equation}
\end{proposition}

Extending the definition~\eqref{eq:def-SEm1m2}, we set
\begin{equation}
    \mathrm{SE}_{m,\infty}^z = \frac{\int_{\phi_0^{-1}(H_z)}|\nabla\phi_m(x)-\nabla\phi_\infty(x)|^2\,dx}{\area(H_z)},
\end{equation}
and define $\mathrm{SE}_{m,\infty}=\mathrm{SE}_{m,\infty}^0$.

\begin{lemma}\label{lem:dirichlet-C}
    For each $m\in\bbN$, $\bbE[\mathrm{SE}_{m,\infty}]<\infty$, and almost surely,
    \eqb\label{eq:lem:dirichlet-C}\lim_{m\to\infty}\bbE[\mathrm{SE}_{m,\infty}|\mathcal{I}] = 0.\eqe
\end{lemma}

\begin{proof}
    For each fixed $m$, by Proposition~\ref{prop:harmonic-exist}, along with the fact that the convergence of harmonic functions $\{f_n\}$   to $f$ in local uniform topology implies the convergence of $\{\nabla f_n\}$ to $\nabla f$ in local uniform topology, $\mathrm{SE}_{m,m'}\to \mathrm{SE}_{m,\infty}$ in probability {as $m'\to\infty$}. By the last sentence of Lemma~\ref{lem:dirichlet-A} and Fatou's lemma, we have $\bbE[\mathrm{SE}_{m,\infty}]<\infty$. The convergence~\eqref{eq:lem:dirichlet-C} then follows from  Fatou's lemma, the summability from~\eqref{eq:lem:dirichlet-B-66} and the boundedness from Lemma~\ref{lem:dirichlet-A}.
\end{proof}

\begin{lemma}\label{lem:dirichlet-D}
    Almost surely, 
    \begin{equation}
   \lim_{m\to\infty}\limsup_{n\to\infty}2^{-2n}\bbE\big[\mathrm{Energy}(\phi_\infty-\phi_m;M([-2^n,2^n]^2))\,|\,\cH,\cD\big] = 0.
    \end{equation}
\end{lemma}

\begin{proof}
Throughout the proof we write $Q_n=[-2^n,2^n]^2$.    By Lemma~\ref{lem:topology-cell-square} and Lemma~\ref{lem:no-macroscopic-cell}, for large enough $n$,
    \begin{equation*}
    \begin{split}
        2^{-2n}\int_{Q_{n+1}} \mathrm{SE}_{m,\infty}^z\,dz &= 2^{-2n} \sum_{H\in\cH(Q_{n+1})} \frac{\area(H\cap Q_{n+1})}{\area(H)}\int_{\phi_0^{-1}(H)} |\nabla\phi_m(w)-\nabla\phi_{\infty}(w)|^2\,dw \\\geq  2^{-2n}& \int_{\phi_0^{-1}(\cH(Q_{n+1})\backslash\cH(\partial  Q_{n+1}))} |\nabla\phi_m(w)-\nabla\phi_{\infty}(w)|^2\,dw 
        \geq 2^{-2n} \mathrm{Energy}(\phi_\infty-\phi_m;M(Q_n)),
        \end{split}
    \end{equation*}
   {where, in the last inequality, we used the fact that if $H_1,H_2,H_3$ form a triangular face on $M(Q_n)$, then for large enough $n$, any cell intersecting  the convex hull of   $H_1\cup H_2\cup H_3$ are within the interior of $Q_{n+1}$}. Now we set $F(\cH,\cD) = \bbE[\mathrm{SE}_{m,\infty}|\cH,\cD]$.   Then using the definition of $\phi_m$ and $\phi_\infty$, $F(\cH-z,\cD-z) = \bbE[\mathrm{SE}_{m,\infty}^z|\cH,\cD]$, and $F(C\cH,C\cD) = F(\cH,\cD)$. By Lemma~\ref{lem:ergodic-cell-system-B} applied to the map $F$, it follows from Lemma~\ref{lem:dirichlet-C} that
    $$
    \limsup_{m\to\infty}\limsup_{n\to\infty}2^{-2n}\int_{Q_{n+1}} \bbE[\mathrm{SE}_{m,\infty}^z|\cH,\cD]\,dz  \leq 2^{32}\limsup_{m\to\infty}\bbE[ \mathrm{SE}_{m,\infty}|\mathcal{I}] = 0.
    $$
    This finishes the proof.
\end{proof}

\begin{lemma}\label{lem:sublinear-A}
    For each fixed $m\in\bbN$, almost surely,
    \begin{equation}\label{eq:lem:sublinear-A}
       \lim_{k\to\infty}2^{-k}\sup_{x\in M([-2^k,2^k]^2)}|\phi_m(x)-\phi_0(x)| = 0.
    \end{equation}
\end{lemma}

\begin{proof}
    Following the same argument as in the proof of~\cite[Equation (2.42)]{GMSinvariance} (where we take the conductance $\fc(H,H')$ there to be the number of edges between $H$ and $H'$,  {which guarantees that their moment bound  \eqref{eq:GMSinvariance-moment} is satisfied}), a.s.\ 
    \begin{equation}\label{eq:pf-sublinear-A-1}
      \lim_{k\to\infty} \frac{1}{2^{k}}\max_{\wh S\cap [-2^k,2^k]^2\neq\emptyset,\wh S\in\mathfrak{S}_m}\diam(\wh S) = 0.
    \end{equation}
    By definition, $\phi_0$ agrees with $\phi_m$ on $\partial \phi_0^{-1}(\wh S)$ for each $\wh S\in \mathfrak{S}_m$. Therefore~\eqref{eq:lem:sublinear-A} follows from~\eqref{eq:pf-sublinear-A-1} along with the maxim{um} principle for the harmonic function $\phi_m|_{\phi_0^{-1}(\wh S)}$.
\end{proof}

We are now ready to prove Proposition~\ref{prop:harmonic-sublinear} by transferring the bound on the Dirichlet energy from Lemma \ref{lem:dirichlet-D} to a pointwise bound for $|\phi_0-\phi_\infty|$. The bound on the Dirichlet energy for $\phi_m-\phi_\infty$ will allow us to control $|\phi_0\circ\varphi_k^{-1}-\phi_\infty\circ\varphi_k^{-1}|$ on a dense collection $\cP_k\subseteq 2^k\bbD$ of horizontal and vertical line segments. The oscillation of $\phi_\infty\circ\varphi_k^{-1}$ on $2^{k-1}\bbD\backslash \cP_k$ is controlled by the maximum principle for harmonic functions, while we use the continuity result from Lemma~\ref{lem:area/length-2} to bound the oscillation of $\phi_0\circ\varphi_k^{-1}$ on $2^{k-1}\bbD\backslash \cP_k$. Combining these with Lemma~\ref{lem:area/length-1}, for some deterministic constant $c_2>0$, the oscillation of $|\phi_0-\phi_\infty|$ on $M([-c_22^k,c_22^k]^2)$ has order $o_k(2^k)$, and the bound can be further extended to $M([-2^k,2^k]^2)$ via Lemma~\ref{lem:dyadic-shift}. The first half of the proof is similar as~\cite{GMSinvariance}, except that we work on continuum Dirichlet energy rather than discrete Dirichlet energy and that the lines $\ell_j,\wt\ell_j$ are contained in the image of certain conformal maps $\varphi_k$ instead of living in the same plane as the cell system. The second half is more different, where  the function $\phi_0$ in~\cite{GMSinvariance} is roughly identity, while   we need to use Lemma~\ref{lem:area/length-1} to control the oscillation of $\phi_0$ in our setting. 
\begin{proof}[Proof of Proposition~\ref{prop:harmonic-sublinear}]
Fix $\e\in(0,10^{-6})$.   By Lemma~\ref{lem:dirichlet-D}, we may pick $m_\e>0$ such that 
$$\limsup_{k\to\infty}2^{-2k}\bbE\big[\mathrm{Energy}(\phi_\infty-\phi_{m_\e};M([-2^k,2^k]^2))|\cH,\cD\big]<e^{-1/\e^8}.$$
Let $\varphi_k:M([-2^k,2^k]^2)\to 2^k\bbD$ be the uniformization map defined as in Lemma~\ref{lem:area/length-1} and Lemma~\ref{lem:area/length-2}. Using the conformal invariance of Dirichlet energy, 
\begin{equation}\label{eq:pf-harmonic-sublinear-1}
    \limsup_{k\to\infty}\frac{1}{2^{2k}} \int_{2^k\bbD} \bbE\Big[|\nabla (\phi_\infty\circ\varphi_k^{-1})(w) -  \nabla (\phi_{m_\e}\circ\varphi_k^{-1})(w)|^2\big|\cH,\cD \Big]\, dw <e^{-1/\e^8}.
\end{equation}
   From~\eqref{eq:pf-harmonic-sublinear-1}, for large enough $k$, there exists a collection of horizontal straight lines $\ell_1,...,\ell_n$ such that for each $j=1,...,n$, 
   \begin{equation}\label{eq:pf-harmonic-sublinear-2}
       \frac{1}{2^{k}} \int_{\ell_j\cap 2^k\bbD }\bbE\Big[|\nabla (\phi_\infty\circ\varphi_k^{-1})(y) -  \nabla (\phi_{m_\e}\circ\varphi_k^{-1})(y)|^2\big|\cH,\cD \big]\, dy <\e^8,
   \end{equation}
   and for $j=0,1,...,n$, the distance between $\ell_j$ and $\ell_{j+1}$ is less than $e^{-1/\e^4}2^k$, with $\ell_0 = \{z:\mathrm{Im}\,z = 2^k\}$ and $\ell_{n+1} = \{z:\mathrm{Im}\,z = -2^k\}$. Similarly, one can find a collection of vertical straight lines $\wt \ell_1,...,\wt \ell_{\wt n}$ such that for each $j=1,...,\wt n$, ~\eqref{eq:pf-harmonic-sublinear-2} holds for $\wt\ell_j$ instead of $\ell_j$, and the distance $\wt\ell_j$ and $\wt\ell_{j+1}$ is less than $e^{-1/\e^4}2^k$, with $\wt\ell_0 = \{z:\mathrm{Re}\,z = 2^k\}$ and $\wt\ell_{\wt n+1} = \{z:\mathrm{Re}\,z = -2^k\}$. We assume that these lines are selected in some way that is measurable with respect to $(\cH,\cD)$. By the Cauchy-Schwarz inequality, for each $\ell\in\{\ell_1,...,\ell_n,\wt\ell_1,...,\wt\ell_{\wt n}\}$,
    \begin{equation}\label{eq:pf-harmonic-sublinear-3}
       \frac{1}{2^k} \int_{\ell\cap 2^k\bbD }\bbE\Big[|\nabla (\phi_\infty\circ\varphi_k^{-1})(y) -  \nabla (\phi_{m_\e}\circ\varphi_k^{-1})(y)|\big|\cH,\cD \big]\, dy <4\e^4.
   \end{equation}
   Integrating this, for each $w,w'\in\ell\cap 2^k\bbD$, 
   \begin{equation*}
       \begin{split}
           \frac{1}{2^k}\bbE\Big[\big|\big(&\phi_\infty\circ\varphi_k^{-1}(w)-\phi_{m_\e}\circ\varphi_k^{-1}(w)\big)-\big(\phi_\infty\circ\varphi_k^{-1}(w') - \phi_{m_\e}\circ\varphi_k^{-1}(w')\big)\big|\Big|\cH,\cD\Big] 
           <4\e^4.
       \end{split}
   \end{equation*}
   It follows that if we set 
   $\cP_k = (\ell_1\cup ...\cup \ell_n\cup\wt\ell_1\cup...\cup\wt\ell_{\wt n})\cap 0.9\cdot2^k\bbD$, then for large enough $k$, 
   \begin{equation}\label{eq:pf-harmonic-sublinear-4}
      \frac{1}{2^k} \sup_{w,w'\in \cP_k}\bbE\Big[\big|\big(\phi_\infty\circ\varphi_k^{-1}(w)-\phi_{m_\e}\circ\varphi_k^{-1}(w)\big)-\big(\phi_\infty\circ\varphi_k^{-1}(w') - \phi_{m_\e}\circ\varphi_k^{-1}(w')\big)\big|\Big|\cH,\cD\Big]<8\e^4< \e^2,
   \end{equation}
   and by~\eqref{eq:lem:sublinear-A}, for large enough $k$,
   \begin{equation}\label{eq:pf-harmonic-sublinear-5}
     \frac{1}{2^k}    \sup_{w,w'\in \cP_k}\bbE\Big[\big|\big(\phi_\infty\circ\varphi_k^{-1}(w)-\phi_{0}\circ\varphi_k^{-1}(w)\big)-\big(\phi_\infty\circ\varphi_k^{-1}(w') - \phi_{0}\circ\varphi_k^{-1}(w')\big)\big|\Big|\cH,\cD\Big]<4\e^2.
   \end{equation}
   Consider a connected component $K$ of $\big({2^k}\bbD\big)\backslash \cP_k$ intersecting $2^{k-1}\bbD$. Then $\diam(K)< {\sqrt{2}}e^{-1/\e^4}2^k$, and by Lemma~\ref{lem:area/length-2}, for some constant $c_1$, for $k$ large enough and all choices of $K$, we have
   \begin{equation}\label{eq:pf-harmonic-sublinear-6}
       \diam\bigg(\bigcup_{H:\varphi_k(P_H)\cap K\neq\emptyset} H\bigg)\leq c_1\e^22^k.
   \end{equation}
    In particular, for each $w,w'\in K$,~\eqref{eq:pf-harmonic-sublinear-6} along with Lemma~\ref{lem:no-macroscopic-cell} implies that a.s.\  for $k$ large enough, 
    \eqb\label{eq:pf-harmonic-sublinear-7} \sup_K\sup_{w,w'\in K}|\phi_0\circ\varphi_k^{-1}(w)-\phi_0\circ\varphi_k^{-1}(w')|<4c_1\e^22^k.\eqe
Using the maxim{um} principle for the harmonic function $ \phi_\infty\circ\varphi_k^{-1} $, a.s.\ for large enough $k$, for each $K$, by~\eqref{eq:pf-harmonic-sublinear-5} and~\eqref{eq:pf-harmonic-sublinear-7},  we have 
\begin{equation}\label{eq:pf-harmonic-sublinear-80} 
    \begin{split}
       & \sup_{w,w'\in K} \bbE\Big[| \phi_\infty\circ\varphi_k^{-1}(w)- \phi_\infty\circ\varphi_k^{-1}(w')|\Big|\cH,\cD \Big]= \sup_{w,w'\in \partial K} \bbE\Big[| \phi_\infty\circ\varphi_k^{-1}(w)- \phi_\infty\circ\varphi_k^{-1}(w')|\Big|\cH,\cD \Big]\\&\leq \sup_{w,w'\in\partial K}\bbE\Big[| (\phi_\infty\circ\varphi_k^{-1}(w)- \phi_\infty\circ\varphi_k^{-1}(w'))-(\phi_0\circ\varphi_k^{-1}(w)- \phi_0\circ\varphi_k^{-1}(w'))|\Big|\cH,\cD \Big] \\
       &\qquad + \sup_{w,w'\in\partial K}  \bbE\Big[| \phi_0\circ\varphi_k^{-1}(w)- \phi_0\circ\varphi_k^{-1}(w')|\Big|\cH,\cD \Big] \leq   {4(1+c_1)\e^22^k}.
    \end{split}
\end{equation}
We conclude by~\eqref{eq:pf-harmonic-sublinear-5},~\eqref{eq:pf-harmonic-sublinear-7}, ~\eqref{eq:pf-harmonic-sublinear-80} along with a triangle inequality that a.s.\  for $k$ large enough, 
\begin{equation}\label{eq:pf-harmonic-sublinear-8} 
   \frac{1}{2^k} \sup_{w,w'\in{2^{k-1}}\bbD}\Bigg|\bbE\Big[\big(\phi_\infty\circ\varphi_k^{-1}(w)-\phi_{0}\circ\varphi_k^{-1}(w)\big)-\big(\phi_\infty\circ\varphi_k^{-1}(w') - \phi_{0}\circ\varphi_k^{-1}(w')\big)\Big|\cH,\cD\Big]   \Bigg|<32(1+c_1)\e^2.
\end{equation}
By Lemma~\ref{lem:area/length-1}, there is some deterministic constant $c_2>0$ such that a.s.\ for $k$ large enough, if $H\cap [-c_22^k,c_22^k]^2\neq\emptyset$, then $\varphi_k(P_H)\subseteq 2^{k-1}\bbD$. Therefore using the no-macroscopic-cell result from Lemma~\ref{lem:no-macroscopic-cell}, it follows from~\eqref{eq:pf-harmonic-sublinear-8} that, a.s.\ for large enough $k$, 
\begin{equation}\label{eq:pf-harmonic-sublinear-9}
     2^{-k} \sup_{x,x'\in   M([-c_22^{k-1},c_22^{k-1}]^2)}\Bigg|\bbE\Big[\big(\phi_\infty (x)-\phi_{0}(x)\big)-\big(\phi_\infty(x') - \phi_{0}(x')\big)\Big|\cH,\cD\Big]   \Bigg|<32(1+c_1)\e^2.
\end{equation}
Pick $\ell_0$ such that $2^{\ell_0}c_2>4$. Since $\phi_\infty(0)=0$, replacing $k$ with $k+\ell_0$ in~\eqref{eq:pf-harmonic-sublinear-9}  further implies that for large enough $k$ and some deterministic constant $c_3>0$,
\begin{equation}\label{eq:pf-harmonic-sublinear-94}
      2^{-k} \sup_{x\in   M([-2^k,2^k]^2)}\Bigg|\bbE\Big[\phi_\infty (x)-\phi_{0}(x)\Big|\cH,\cD\Big]   \Bigg|<c_3\e^2.
\end{equation}

Finally, Lemma~\ref{lem:no-macroscopic-cell} implies that almost surely,
$$\lim_{k\to\infty} \frac{1}{2^k}\sup_{x\in M([-2^k,2^k]^2)}\big|\phi_0(x)-\bbE[\phi_0(x)|\cH,\cD]\big| = 0.$$
Combining with~\eqref{eq:pf-harmonic-sublinear-94} gives ~\eqref{eq:prop:harmonic-sublinear}. 
\end{proof}

\subsection{Proof of Theorem~\ref{thm:invariance-uniformization}}\label{subsec:pf-thm:invariance-uniformization}

In this section we complete the proof of Theorem~\ref{thm:invariance-uniformization}.   Recall that by Lemma~\ref{lem:parabolic}, the surface $M(\cH)$ is a.s.\ parabolic, and by Proposition~\ref{prop:harmonic-sublinear}, $\bbE[\phi_\infty|\cH,\cD]$ is close to $\phi_0$ on a large scale. Let $\varphi$ be a uniformization map $\varphi:M(\cH)\to\bbC$ chosen in a way measurable with respect to $\cH$ such that $\varphi(0)=0$. We will prove that  the harmonic function $\wt\phi_\infty: \bbC\to\bbC$ defined by $\wt\phi_\infty(z):=\bbE[\phi_\infty\circ\varphi^{-1}(z)|\cH,\cD]$ 
is the composition of a deterministic linear transform  {$\mathbf{A}$}, a rotation and a scaling, 
which implies that the embedding induced by $\varphi(M(\cH))$ is close to $\cH$ up to a rotation.  
This finishes the proof of Theorem~\ref{thm:invariance-uniformization} when combined with Proposition~\ref{prop:harmonic-sublinear}.

We prove the decomposition of $\wt\phi_\infty$ as follows.  
 First we prove that the harmonic function $\wt\phi_\infty$ must be a polynomial (Lemma~\ref{lem:harmonic-polynomial}). This follows from Proposition~\ref{prop:harmonic-sublinear} as well as the polynomial growth property in Lemma~\ref{lem:cell-polynomial-growth}. Then in Lemmas~\ref{lem:polynomial-A},~\ref{lem:polynomial-B} and~\ref{lem:polynomial-C}, we prove that this polynomial must have degree 1, by arguing that if this is not the case, then one can find cells $H_1,H_2$ which are close in $\cH$ and have macroscopic distance in $\varphi(M(\cH))$, which contradicts the continuity as in Lemma~\ref{lem:area/length-1} and Lemma~\ref{lem:area/length-4}. Finally,  observe that the linear transform $\wt\phi_\infty$ has to be measurable with respect to $\cH$ since it approximates $\cH$ by Proposition~\ref{prop:harmonic-sublinear}, and $\wt\phi_\infty\wt\phi_\infty^\rmT/|\det(\wt\phi_\infty)|$ defines a function over $\cH$ which is invariant of the choice of $\varphi$ and also invariant under scaling and translation. Then Definition~\ref{def:ergodic-modulo-scaling} implies that it must be constant over $\cH$, and the desired decomposition is a consequence of standard linear algebra. 

Before giving the details of this proof, let us remark that {our approach in this section is completely different from the counterpart of Theorem~\ref{thm:invariance-uniformization} for the Tutte embedding  {in \cite{GMSinvariance,GMS21tutte}} and circle packings  {in our Section \ref{sec:circle-packing}}. In the setting of Tutte embedding, there is no  well-defined  {canonical} Tutte embedding for infinite maps; instead an invariance principle is proven in \cite{GMSinvariance}, where the simple random walk on $\cH$ a.s.\ converges to the Brownian motion (Theorem~\ref{thm:GMSinvariance}). In the finite volume setting convergence of the conformal embedding is an immediate consequence of the invariance principle by the definition of the Tutte embedding \cite{GMS21tutte}. In the setting of circle packings, we proved in Section~\ref{subsec:Covergence-RW-cell-Dubejko} and Section~\ref{subsec:Covergence-RW-CP-Dubejko} that the random walk on $\cH$ and the circle packing for $\cH$ with Dubejko weights a.s.\ converges to the Brownian motion, and in Section~\ref{subsec:pf-Cell-system-Pack} we used this to conclude that the cell configuration $\cH$ is close to its circle packing. 

We remark that it might be possible to adapt the arguments in~\cite[Section 3]{GMSinvariance} to show that for a Brownian motion $(W_t)_{t\geq 0}$ on $M(\cH)$, $(\phi_0(W_t))_{t\geq 0}$ is approximately a Brownian motion on $\bbC$, and, combined with the same arguments from Section~\ref{subsec:pf-Cell-system-Pack}, we would conclude that the a priori embedding $\phi_0$ of $M(\cH)$ is close to a conformal map. On the other hand, our proof in this section is relatively elementary and has independent interest, as it gives another way of proving that harmonic functions are conformal under certain conditions.  {We emphasize that our argument works under rather mild assumptions; the key inputs are polynomial growth, almost injectivity, and that the average Dirichlet energy over large regions is approximately constant.}

\begin{lemma}\label{lem:harmonic-polynomial}
   There exists a deterministic constant $\beta>0$, such that almost surely,  for some  {random} constant $\wt C>0$ and every large enough $r$, $|\wt\phi_\infty(z)|\leq  4\wt Cr^\beta$  {for every $z\in  {B(0;r)}$.} 
    {Furthermore,} 
   the function $\wt\phi_\infty$ is a polynomial over $z$ and $\ol z$ of degree at most $\beta$.
\end{lemma}

\begin{proof}
    Let $\wt C$ and $\beta$ be the constants from Lemma~\ref{lem:cell-polynomial-growth}, and assume $r$ is large such that~\eqref{eq:lem:cell-polynomial-growth} holds. Let $z\in B(0;r)$, and suppose $\varphi^{-1}(z)\in M(\cH)$ is on the triangle formed by the cells $H_1,H_2,H_3$. Then for some $j\in\{1,2,3\}$, $H_j\subseteq B(0;\wt Cr^\beta)$, and by Lemma~\ref{lem:no-macroscopic-cell}, $H_1,H_2,H_3\subseteq B(0;2\wt C r^\beta)$, and by Lemma~\ref{lem:M(S)}, ${\varphi^{-1}(z)}\in M([-3\wt Cr^\beta,3\wt Cr^\beta]^2)$. Therefore by~\eqref{eq:lem:cell-polynomial-growth}, $|\phi_0\circ\varphi^{-1}(z)|\leq 2\wt Cr^\beta$, and by~\eqref{eq:prop:harmonic-sublinear}, for large enough $r$, $|\phi_0\circ\varphi^{-1}(z)-\wt\phi_\infty(z)|\leq \wt Cr^\beta$. This verifies the first claim. For the second claim, note that $\wt\phi_\infty$ is a harmonic function defined on the whole plane. By standard properties of harmonic functions, for each $n$, there exists absolute constant $C_n$ such that for every harmonic function $f$ defined on $\bbC$, $$\sup_{z\in B(0;r)}|D^nf(z)|\leq C_nr^{-n} \sup_{z\in B(0;2r)}|f(z)|,$$
    where $D^nf$ is the $n$th derivative of $f$. Therefore the first claim implies that for all $n> \beta$, $D^n\wt\phi_\infty = 0$, which implies the second claim.
\end{proof}

 {In order to get the next lemma we argue that the lower-order terms of the polynomial $\wt\phi_\infty$ (i.e., those having exponent less than the degree of the polynomial) are of a smaller order than $2^k$ when evaluated on $M([-2^k,2^k]^2)$.}
\begin{lemma}\label{lem:polynomial-A}
    There exist possibly random constants $\beta\in\bbN$, $a_\beta,b_\beta\in\bbC$ with $|a_\beta|^2+|b_\beta|^2\neq 0$, such that almost surely
    \begin{equation}\label{eq:lem:polynomial-A}
        \lim_{k\to\infty} \frac{1}{2^k}\sup_{x\in M([-2^k,2^k]^2)}|\phi_0(x) - a_\beta \mathrm{Re}(\varphi(x)^\beta) - b_\beta\mathrm{Im}(\varphi(x)^\beta)\mathrm|=0.
    \end{equation}
\end{lemma}

\begin{proof}
    By Lemma~\ref{lem:harmonic-polynomial}, we may assume that $\wt\phi_\infty$ is a polynomial of degree $\beta\in\bbN \cup\{0\}$. From the definition of the surface $M([-2^k,2^k]^2)$ and Lemma~\ref{lem:M(S)}, we infer from~\eqref{eq:prop:harmonic-sublinear} that $\beta\neq0$. Note that any harmonic polynomial  {in $z$ and $\ol z$} of degree $\beta$ can be written as a linear combination of $\{1,\mathrm{Re}\,z, ..., \mathrm{Re}\,z^\beta, \mathrm{Im}\,z, ..., \mathrm{Im}\,z^\beta\}$, and we may assume $\wt\phi_\infty(z) = a_\beta \mathrm{Re}\,z^\beta+b_\beta \mathrm{Im}\,z^\beta + p_\infty(z)$, where $|a_\beta|^2+|b_\beta|^2\neq 0$, and $p_\infty(z)$ is a polynomial of $z,\ol z$ of degree at most $\beta-1$. In particular, there exists some random $\wt C>0$ such that $|\wt\phi_\infty(z)-a_\beta \mathrm{Re}\,z^\beta-b_\beta \mathrm{Im}\,z^\beta|\leq \wt C|z|^{\beta-1}$ for $|z|$ large enough, and there exists some $\theta_0\in[0,2\pi]$ and $\wt c>0$ such that $|\wt \phi_\infty(re^{i\theta_0})|\geq \wt c r^\beta$ for $r$ large enough. Recall the constant $\alpha$ from Lemma~\ref{lem:area/length-4} and the definition of $d_k$ from~\eqref{eq:def-d_k-alpha}. By~\eqref{eq:lem:area/length-5}, for some constant $c_2$,  $B(0;c_2d_k)\subseteq \varphi(M([-\alpha 2^k,\alpha 2^k]^2))$ for all large enough $k$. Then   \eqb\label{eq:pf-lem:polynomial-A-0}  \phi_0\circ\varphi^{-1}(B(0;c_2d_k))\subseteq [-\alpha 2^k,\alpha 2^k]^2\eqe  from the definition of $M([-\alpha 2^k,\alpha 2^k]^2)$, and further by~\eqref{eq:prop:harmonic-sublinear}, $|\wt\phi_\infty(z)|\leq 2^{k+1}\alpha$ for $z\in B(0;c_2d_k)$. Therefore a.s.\ for large enough $k$, we have 
    \eqb\label{eq:pf-lem:polynomial-A-1} \wt c(c_2d_k)^\beta\leq 2^{k+1}\alpha.\eqe
    On the other hand, a.s.\ for all $k$ sufficiently large, 
   we have
    \begin{equation}\label{eq:pf-lem:polynomial-A-2}
        |\wt\phi_\infty(z)-a_\beta \mathrm{Re}\,z^\beta-b_\beta \mathrm{Im}\,z^\beta|\leq \wt Cd_k^{\beta-1}, \ \ \text{for all }z\in \ol{B(0;d_k)}.
    \end{equation}
    Combining~\eqref{eq:pf-lem:polynomial-A-1},~\eqref{eq:pf-lem:polynomial-A-2} with~\eqref{eq:prop:harmonic-sublinear}, and using $\varphi(M([-\alpha 2^k,\alpha 2^k]^2))\subseteq \ol{B(0;d_k)}$, we conclude that a.s.,
    \begin{equation*} 
        \lim_{k\to\infty} \frac{1}{2^k}\sup_{x\in M([-\alpha 2^k,\alpha 2^k]^2)}|\phi_0(x) - a_\beta \mathrm{Re}(\varphi(x)^\beta) - b_\beta\mathrm{Im}(\varphi(x)^\beta)\mathrm|=0,
    \end{equation*}
    which further implies~\eqref{eq:lem:polynomial-A}. 
\end{proof}

{As discussed above, we get that $\beta=1$ since regularity results for $\varphi$ imply that if two cells are close in $\cH$ then their images under $\varphi$ are also close.}
\begin{lemma}\label{lem:polynomial-B}
    In the setting of Lemma~\ref{lem:polynomial-A}, almost surely $\beta=1$.
\end{lemma}

\begin{proof}
Let $\alpha$ and $d_k$ be given by Lemma~\ref{lem:area/length-4} and~\eqref{eq:def-d_k-alpha} as in the proof of Lemma~\ref{lem:polynomial-A}, and $c_2$ be the constant in Lemma~\ref{lem:area/length-5}.  Using identical arguments as in the proof of Lemma~\ref{lem:area/length-1} and Lemma~\ref{lem:area/length-6}, one can show that there exists a deterministic constant $c>0$, such that for each $\delta\in(0,0.01)$, a.s.\ for each $w\in B(0;\alpha2^{k-1})$ and large $k$,
\begin{equation}\label{eq:lem:polynomial-B}
    \diam\Big(\bigcup_{H\in\cH(w+[-\delta2^k,\delta2^k]^2)} \varphi(P_H) \Big)\leq c(-\log \delta)^{-1/2}d_k.
\end{equation}
Fix $\delta>0$ such that $ c(-\log \delta)^{-1/2}<0.01c_2$. We assume that $\beta>1$ and want to derive a contradiction. Choose $\theta_\beta$ such that $\beta\theta_\beta\in 2\pi\bbZ$ and $|e^{i\theta_\beta}-1|\geq1$. By~\eqref{eq:pf-lem:polynomial-A-0}  and ~\eqref{eq:lem:polynomial-A}, a.s.\ for all large $k$, for every $z\in B(0;c_2d_k)$, $|\phi_0\circ\varphi^{-1}(z)|\leq \alpha 2^k$ and $|\phi_0\circ\varphi^{-1}(z)-a_\beta\mathrm{Re}\,z^\beta-b_\beta\mathrm{Im}\,z^\beta|<\alpha 2^{k-1}$ and therefore $|a_\beta \mathrm{Re}\,z^\beta+b_\beta \mathrm{Im}\,z^\beta|\leq \alpha 2^{k+1}$.  {This implies that}, a.s.\ for all large $k$, $|a_\beta \mathrm{Re}\,z^\beta+b_\beta \mathrm{Im}\,z^\beta|\leq \alpha 2^{k-3}$ for $z\in B(0;c_2d_k/4)$ since $\beta\geq2$. 
By another application of~\eqref{eq:lem:polynomial-A}, a.s.\ for large $k$,
\begin{equation}\label{eq:pf:lem:polynomial-B-1}
    |\phi_0\circ\varphi^{-1}(z)|\leq \alpha 2^{k-2}, \ \ \text{for } z\in B(0;c_2d_k/4).
\end{equation}
Now let $w_0 = a_\beta(c_2d_k/4)^\beta$. Then   by~\eqref{eq:lem:polynomial-A}, a.s.\ for large enough $k$, 
\begin{equation}
    |\phi_0\circ\varphi^{-1}(\frac{c_2d_k}{4} ) - w_0|\leq 0.01\delta2^k, \ \  \ |\phi_0\circ\varphi^{-1}(\frac{c_2d_ke^{i\theta_\beta}}4) - w_0|\leq 0.01\delta2^k.
\end{equation}
By Lemma~\ref{lem:no-macroscopic-cell}, a.s.\ for large enough $k$, 
\begin{equation}\label{eq:pf:lem:polynomial-B-2}
   \frac{c_2d_k}{4} , \frac{c_2d_ke^{i\theta_\beta}}4 \in  \bigcup_{H\in\cH(w_0+[-\delta2^{k-1},\delta2^{k-1}]^2)} \varphi(P_H).
\end{equation}
By the line connectivity property as in  Definition~\ref{def:connectedness}, a.s.\ for large $k$, the set in~\eqref{eq:pf:lem:polynomial-B-2} is connected, and thus by~\eqref{eq:lem:polynomial-B}, $$|\frac{c_2d_k}{4} -  \frac{c_2d_ke^{i\theta_\beta}}4|\leq c(-\log \delta)^{-1/2}d_k< 0.01c_2d_k.$$
This contradicts our assumption that $|e^{i\theta_\beta}-1|\geq1$.
\end{proof}

\begin{lemma}\label{lem:polynomial-C}
    The function $\wt\phi_\infty$ is a.s.\ a linear transform of $\bbC$. Furthermore, there exists a deterministic constant $c\in(0,1)$ such that 
    \begin{equation}\label{eq:lem:polynomial-C}
       B(0;c\diam (\wt\phi_\infty(\bbD)))\subseteq \wt\phi_\infty(\bbD).
    \end{equation}
\end{lemma}

\begin{proof}
    The claim that $\wt\phi_\infty$ is a linear transform of $\bbC$ follows from Proposition~\ref{prop:harmonic-sublinear}, Lemma~\ref{lem:polynomial-B} and the fact that $\phi_\infty(0)=0$. For the claim~\eqref{eq:lem:polynomial-C}, let $c_2,\ell$ be the constant from Lemma~\ref{lem:area/length-5}. We first show that  
    \begin{equation}\label{eq:lem:polynomial-C-1}
        [-\alpha 2^{k-\ell-1},\alpha2^{k-\ell-1}]^2\subseteq \wt\phi_\infty(B(0;c_2d_k)) \ \ \text{a.s.\ for large }k.
    \end{equation}
    Assume~\eqref{eq:lem:polynomial-C-1} does not hold. Using linearity of $\wt\phi_\infty$ along with Proposition~\ref{prop:harmonic-sublinear}, there exists some point $w_0\in [-0.9 \alpha 2^{k-\ell},0.9\alpha 2^{k-1}]^2$ such that $w_0$ has distance at least $0.1\alpha 2^{k-\ell}$ from $\phi_0\circ \varphi^{-1}(B(0;c_2d_k))$. By Lemma~\ref{lem:no-macroscopic-cell} and Lemma~\ref{lem:M(S)}, there exists cells $H_1,H_2,H_3$ with distance at least $0.05\alpha 2^{k-\ell}$ from $\phi_0\circ \varphi^{-1}(B(0;c_2d_k))$ and form a triangular face on $M([-\alpha 2^{k-\ell},\alpha 2^{k-\ell}]^2)$. By~\eqref{eq:lem:area/length-5}, the $\varphi$-image of this face is contained within $B(0;c_2d_k)$, and this contradicts the statement that $H_1,H_2,H_3$ with distance at least $0.05\alpha 2^{k-\ell}$ from $\phi_0\circ \varphi^{-1}(B(0;c_2d_k))$. This verifies~\eqref{eq:lem:polynomial-C-1}. By the second half of~\eqref{eq:lem:area/length-5} along with Lemma~\ref{lem:no-macroscopic-cell}, a.s.\ for large enough $k$, $\phi_0\circ\varphi^{-1}(B(0;c_2d_k)) \subseteq [-\alpha 2^{k+1},\alpha 2^{k+1}]^2$, and by Proposition~\ref{prop:harmonic-sublinear},  
    \begin{equation}\label{eq:lem:polynomial-C-2}
         \wt\phi_\infty(B(0;c_2d_k))  \subseteq  [-\alpha 2^{k+2},\alpha2^{k+2}]^2 \ \ \text{a.s.\ for large }k.
    \end{equation}
    Using linearity of $\wt\phi_\infty$, from~\eqref{eq:lem:polynomial-C-1} and~\eqref{eq:lem:polynomial-C-2} we see that  a.s.\  for some $b>0$,
    \begin{equation}
        B(0;b) \subseteq \wt\phi_\infty(\bbD) \subseteq B(0;  2^{\ell+4}b).
    \end{equation}  
    This verifies~\eqref{eq:lem:polynomial-C} for $c = (  2^{\ell+5})^{-1}$.
\end{proof}

 Recall that for a square $S$, $\mathring{S}^a$ is the square concentric to $S$ with side length $a$. Below we also write $\wt\phi_\infty$ as the $2\times 2$ matrix for the linear transform $\wt\phi_\infty$. 

 \begin{proof}[Proof of Theorem~\ref{thm:invariance-uniformization}]
  By Lemma~\ref{lem:parabolic}, $M(\cH)$ is a.s.\ parabolic.  Recall that $\varphi$ is a uniformization map $\varphi:M(\cH)\to\bbC$ chosen in a way measurable with respect to $\cH$ such that $\varphi(0)=0$. Let $\wt\phi_\infty$ be as in Lemma~\ref{lem:polynomial-C}. From our construction, $\wt\phi_\infty$ is measurable with respect to $\cH$ and the dyadic system $\cD$. Now take an independent dyadic system $\cD'$ independent from $\cD$, let $\phi_\infty^{\cD'}$ be defined in the same way as $\phi_\infty$ using $(\cH,\cD')$ in place of $(\cH,\cD)$ and let $\wt\phi_\infty^{\cD'} = \bbE[\phi_\infty^{\cD'}\circ \varphi^{-1}|\cH,\cD']$. Since  $\wt\phi_\infty^{\cD'}$ and $\wt\phi_\infty$ are both linear and approximate $\phi_0\circ \varphi^{-1}$ on a large scale by Proposition~\ref{prop:harmonic-sublinear}, it follows that we must have $ \wt\phi_\infty^{\cD'} = \wt\phi_\infty$. As a consequence, $ \wt\phi_\infty$ is measurable with respect to $\cH$. By Lemma~\ref{lem:polynomial-C}, $\det( \wt\phi_\infty)\neq 0$.  Now the function $$\mathbf{G_*}(\cH):= \frac{\wt\phi_\infty  \wt\phi_\infty^\rmT}{|\det( \wt\phi_\infty)|}$$
  is invariant of the choice of the conformal map $\varphi$ since changing $\varphi$ by $re^{i\theta}\varphi$ will change $\wt\phi_\infty$ by $\wt\phi_\infty\circ(r^{-1}e^{-i\theta}\cdot)$. Furthermore, as we replace $\cH$ with $\lambda\cH-z$ where $\lambda>0$, then $\wt\phi_\infty$ will be replaced with $\lambda\wt\phi_\infty$.  This implies that $\mathbf{G_*}$  is invariant under any translation and scaling of $\cH$. Therefore by Definition~\ref{def:ergodic-modulo-scaling}, $\mathbf{G_*}$ must be deterministic. Likewise, the function $\sigma_*(\cH):=\mathrm{sgn}(\det( \wt\phi_\infty))$ is also deterministically equal to 1 or $-1$. If we set 
  \begin{equation}
      \mathbf{J}_{\sigma_*} = \begin{pmatrix}
          1 & 0\\0&\sigma_*
      \end{pmatrix}, \ \  \mathbf{A}  =  \mathbf{J}_{\sigma_*}\mathbf{G_*}^{-1/2},
  \end{equation}
  then $\det(\mathbf{A}) = \sigma_*$, and $(\mathbf{A} \wt\phi_\infty )(\mathbf{A} \wt\phi_\infty )^\rmT =|\det( \wt\phi_\infty)| \mathbf{A}\mathbf{G}_*\mathbf{A}^\rmT  =|\det( \wt\phi_\infty)|\,\mathbf{I}$
  where $\mathbf{I}$ is identity matrix. Let $\mathbf{R}= |\det( \wt\phi_\infty)|^{-1/2} \mathbf{A} \wt\phi_\infty $. Then $\det(\mathbf{R})=1$ and $\mathbf{R}$ must be a  random rotation.  This gives the desired decomposition 
  $$  \wt\phi_\infty  = \mathbf{A}^{-1}\big(|\det( \wt\phi_\infty)|^{1/2}\mathbf{R}\big).$$ This together with Proposition~\ref{prop:harmonic-sublinear} and Lemma~\ref{lem:polynomial-C} give~\eqref{eq:thm:invariance-uniform} for our choice of $\mathbf{A}$ and $\varphi_0 = \big(|\det( \wt\phi_\infty)|^{1/2}\mathbf{R}\big)\varphi$. If we further assume the rotational invariance, then for some $\theta_0\neq\pi$, as we write $\mathbf{R}_{\theta_0} = \begin{pmatrix}
    \cos{\theta_0} & -\sin {\theta_0} \\ \sin\theta_0 & \cos\theta_0
\end{pmatrix},$ then $\mathbf{R}_{\theta_0}\mathbf{G}_*\mathbf{R}_{\theta_0}^\rmT = \mathbf{G}_*$, which further forces $\mathbf{G}_* = \mathbf{I}$ so $\mathbf{A}$ has to be the identity matrix or the reflection matrix. 

Finally, a.s.\ for large enough $k$,  for each edge $e\in\cE\cH(B(0;2^k))$, by Lemma~\ref{lem:no-macroscopic-cell} and Lemma~\ref{lem:M(S)}, the edge $\mathfrak{e}_{e}\subseteq M([-2^{k+1},2^{k+1}]^2)$. By Lemma~\ref{lem:no-macroscopic-cell}, 
 \begin{equation}\label{eq:pf:thm:invariance-uniformization-5}
     \lim_{k\to\infty} 2^{-k}\max_{e\in\cE\cH(B(0;2^k))} \diam(\phi_0(\mathfrak{e}_{e})) = 0.
 \end{equation}
 The relation~\eqref{eq:thm:invariance-uniform-1} then follows from Proposition~\ref{prop:harmonic-sublinear} and~\eqref{eq:pf:thm:invariance-uniformization-5}.
\end{proof}

\section{Extensions to general planar maps}\label{subsec:extension}
In this section, we prove Theorem~\ref{thm:CellSystemPack-extension} and Theorem~\ref{thm:invariance-uniformization-extension}, extending Theorem~\ref{thm:CellSystemPack} and Theorem~\ref{thm:invariance-uniformization} to the setting of general planar maps. 
In Section~\ref{subsec:large}, we prove Theorem~\ref{thm:large}, which is a version of Theorems~\ref{thm:CellSystemPack} and~\ref{thm:invariance-uniformization}  with a weaker notion of the line connectivity property in Definition~\ref{def:connectedness}. 

\subsection{Proof of Theorem~\ref{thm:CellSystemPack-extension}}\label{subsec:extension-circ}

Let $\cM$  be a   2-connected simple planar map with whole-plane topology and face set $\cF\cM$. Recall from Section~\ref{subsec:intro-4} that one can add a vertex $v_f$ to every face $f\in\cF\cM$, and connect the vertices $v_f$ to the vertices incident to $f$, to get a triangulation $\cT_\cM$. {The triangulation $\cT_\cM$ will be simple since we required $\cM$ to be 2-connected.} We define the circle packing of $\cM$ to be the circle packing for $\cT_\cM$ minus the disks for the vertices $\{v_f:f\in\cF\cM\}$. In this section we prove a variant of Theorem~\ref{thm:CellSystemPack} in this setting.

 Recall the simple infinite plane triangulation $\cT_\cM$ constructed from $\cM$.  To prove Theorem~\ref{thm:CellSystemPack-extension}, we construct a cell configuration $\wt\cH$ whose associated map is $\cT_\cM$. For each cell $H\in\cH$, we partition $H$ into  {cells} $\wt H_{H,0}$ and $\{\wt H_{H,f}:v_H\sim f\}$ as follows. We independently pick $\theta_H\sim\rm{Unif}((0,2\pi))$ and consider the straight line $L_0$ passing through the Euclidean center  $c(H)$ of $H$ with angle $\theta$. By continuously varying the angle, one can find $0<\theta_{H,1}<...<\theta_{H,\deg(H)}<\pi$ such that if one rotates $L_0$ clockwise around $c(H)$ over  angles $\theta_{H,1},...,\theta_{H,\deg(H)}$ to get the lines $L_1,...,L_{\deg(H)}$, then for $k=0,...,\deg(H)$, the portion $\wh H_{H,k}$ of $H$ between $L_k$ and $L_{k+1}$ (with $1+\deg(H)$ identified with 0) has area $\area(H)/(\deg(H)+1)$. For $\{H,H'\}\in\cE \cH$, pick $\wt\theta_{H,H'}\sim\rm{Unif}((0,2\pi))$ independently and set $J_{H,H'} = \{z: \mathrm{Re}\,(e^{i\wt\theta_{H,H'}}z) =\max_{w\in H\cap H'} \mathrm{Re}\,(e^{i\wt\theta_{H,H'}}w)\}$. Let $\wh D_H$ be the unbounded connected component of $\bbC\backslash H$, and let $\wt  H_{H,k} = \wh H_{H,k}\cup \partial \wh D_H\cup (((L_0\cup...\cup L_{\deg(H)})\cup_{H'\sim H}J_{H,H'})\backslash\wh D_H)$, such that each $\wt  H_{H,k}$ is connected and compact, and contains the outer boundary of $H$. The inclusion of $J_{H,H'}$ implies that $\wt H_{H,k}\cap \wt H_{H',k'}\neq\emptyset$ whenever $H'\sim H$.  For each face $f$ with $v_H\sim f$,  we further uniformly randomly assign $\wt H_{H,f}$ to be equal to one of $\wt  H_{H,k} $ for $k=1,...,\deg(H)$ in a bijective way. For each vertex $v_f$ added to generate $\cT_\cM$, let $\wt H_f = \cup_{H:v_H\sim f} \wt H_{H,f}$, and let $\wt H_H = \wt H_{H,0}$. This defines a cell configuration $\wt\cH$  {with cells $\wt H_f$ and $\wt H_H$} whose associated map is $\cT_\cM$. Now we check that $\wt\cH$ satisfies all the conditions in Theorem~\ref{thm:CellSystemPack}.

 \begin{lemma}\label{lem-extension-1.2}
     The cell configuration $\wt\cH$ is translation invariant modulo scaling as in Definition~\ref{def:translation-modulo-scaling}.
 \end{lemma}

 \begin{proof}
Let $\fD$ be a uniform dyadic system independent from $\cH$ and all the additional random variables in the construction of $\wt\cH$.  
  For $m>0$, let $\wh S_m$  be  defined for $(\cH,\fD)$  according to~\eqref{eq:def-S_m^z}.    
  Conditioned on $(\cH,\cD)$, sample $z_m$ uniformly from the Lebesgue measure on $\wh S_m$ independently from  the auxiliary random variables.
  By Lemma~\ref{lem:dyadic-resample}, 
  $(C_m( \cH-z_m),C_m(\fD-z_m))\overset{d}{=} (\cH,\cD)$ for some random constant $C_m>0$ depending on $(\cH,z_m,\fD)$ but not on the auxiliary random variables.
  Since our construction of $\wt\cH$ in terms of $\cH$ commutes with  scalings and translations of $\cH$, it follows that $C_m( \wt\cH-z_m) \overset{d}{=} \wt\cH$. Since $\wh S_m\uparrow \bbC$ as $m\to\infty$ and given $\wh S_m$ and $\wt\cH$, $z_m$ is sampled uniformly from the Lebesgue measure, the second condition of~\cite[Definition 1.2]{GMSinvariance} holds for $\wh S_m$, which finishes the proof.  
 \end{proof}

 \begin{lemma}\label{lem-extension-1.3}
     The cell configuration $\wt\cH$ is ergodic modulo scaling as in Definition~\ref{def:ergodic-modulo-scaling}.
 \end{lemma}

  \begin{proof}
By Lemma~\ref{lem-extension-1.2}, we need to prove that, if $0\leq F\leq 1$ is a measurable function such that $F(C(\wt \cH-z))=F(\wt \cH)$ for every $C>0$ and $z\in\bbC$, then $F$ is a.s.\ a constant. Let $X = F(\wt\cH)$. Since our construction of $\wt\cH$ in terms of $\cH$ commutes with translation and scaling, by Definition~\ref{def:ergodic-modulo-scaling} for $\cH$, the function $F_1(\cH):=\bbE\big[X|\cH\big]$ is a.s.\ equal to a constant $p\in[0,1]$. It remains to  prove  $X=p$ a.s.. For $H\in\cH$, let $\Theta_k(H)$ be the collection of the additional random variables $\theta_{H'}$, $\theta_{H',H''}$ and those random bijections in the partition of $H'$ where $H',H''$ have graph distance less than $k$ from $H$ in our construction of $\wt\cH$. Let $X_k = \bbE[X|\cH,\Theta_k(H_0)]$. Then the martingale convergence theorem implies that $\bbE[|X_k-X|]\to0$ as $k\to\infty$. For $z\in\bbC$ and contained within a unique cell $H_z$, define $X_k^z = \bbE[X|\cH,\Theta_k(H_z)]$.

Let $\cD$ be a uniform dyadic system independent from $(\cH,\cup_{k\geq1}\Theta_k)$.
 Let $\wh {S}_m^z$  be as in ~\eqref{eq:def-S_m^z} for the cell configuration $\cH$, $\wh S_m=\wh S_m^0$ and $z_m$ be uniformly randomly chosen from the Lebesgue measure on $\wh S_m$. By Lemma~\ref{lem:dyadic-resample} applied to $\cH$, for any $m>0$, \eqb\label{eq:pf-lem-extension-1.2-1} 
(\wt C(\cH-  z_m), \wt C(\cD-  z_m))\overset{d}{=} (\cH, \cD) \text{ for some random constant $\wt C>0$}. 
\eqe 
 Furthermore, since $F$ is invariant under scaling and translation, ~\eqref{eq:pf-lem-extension-1.2-1} implies that $\bbE[|X-X_k|] = \bbE[|X-X_k^{z_m}|]$ for every $m\in\bbN$.  Now consider the event $A_{k,m}$ that  the graph distance between $H_0$ and  $H_{z_m}$ is at least $3k$. Let $\cG_m$ be the $\sigma$-algebra generated by $(\cH,\cD,z_m)$. It is straightforward to verify that $A_{k,m}\in \cG_m$ and for fixed $k$, $\bbP[A_{k,m}^c]\to 0$ as $m\to\infty$.  The ergodicity of $\cH$ and the independence between $\cD$ and $(\cH,\cup_{k\geq1}\Theta_k)$ as well as  ~\eqref{eq:pf-lem-extension-1.2-1} imply that $\bbE[X|\cG_m] = \bbE[X_k^{z_m}|\cG_m]=p$. On the other hand,  as we further condition on $\cG_m$ and $A_{k,m}$, $X_k$ and $X_k^{z_m}$ are determined by independent random variables and therefore are conditionally independent. In particular, $\bbE[\mathds{1}_{A_{k,m}}(X_k-p)(X_k^{z_m}-p)|\cG_m]=0$ and therefore $|\bbE[(X_k-p)(X_k^{z_m}-p)]|\leq \bbP[A_{k,m}^c]$. Now 
 $$(X-p)^2=(X_k-p)(X_k^{z_m}-p) + (X-X_k)(X-p)+(X_k-p)(X-X_k^{z_m}),$$
 and since $0\leq X,p,X_k\leq 1$, using $\bbE[|X-X_k|]=\bbE[|X-X_k^{z_m}|]$ for every $m\in\bbN$ we have $$\mathrm{Var}(X) = \bbE[(X-p)^2]\leq |\bbE[(X_k-p)(X_k^{z_m}-p)]|+2\bbE[|X-X_k|]\leq \bbP[A_{k,m}^c]+2\bbE[|X-X_k|].$$ 
 Therefore if we first send $m\to\infty$ and later $k\to\infty$, we have $\mathrm{Var}(X) = 0$. This finishes the proof.

  \end{proof}
 
\begin{lemma}\label{lem-extension-1.5}
    There are no macroscopic faces of $\cH$, in the sense that almost surely for any $\e>0$, for large enough $k$, for each cell $H_1,...,H_j\in\cH$ with $(H_1\cup...\cup H_j)\cap S_k\neq\emptyset$ and $v_{H_1},...,v_{H_j}$  {lie on the same} 
    face of $\cM$, $\diam(H_1\cup...\cup H_j)<\e|S_k|$.
\end{lemma}
\begin{proof}
    The proof is identical to that of~\cite[Lemma 2.9]{GMSinvariance} where one considers the set $\ol H_0 = \bigcup_{f\in\cF\cM:v_{H_0}\sim f}\bigcup_{H\in\cH:v_H\sim f}H$ and use the moment bound from~\eqref{eq:thm:CellSystemPack-extension}.
\end{proof}

 \begin{proof}[Proof of Theorem~\ref{thm:CellSystemPack-extension}]
     Let $\wt\cH$ be the cell configuration constructed as above. The ergodicity modulo scaling property for $\wt\cH$ holds by Lemma~\ref{lem-extension-1.3}. For $H\in\cH$, observe that if a line segment intersects $\bbC\backslash \wh D_H$ and has diameter at least $2\diam(\bbC\backslash \wh D_H)$, then it must intersect the outer boundary of $H$.
     Since each $\wt H_{H,k}$ contains the outer boundary of $H$, together with the shared contact lines and original line connectivity as well as the no-macroscopic cell lemma for $\cH$ (Lemma~\ref{lem:no-macroscopic-cell}) we see that the line connectivity property from Definition~\ref{def:connectedness} holds. 
     Furthermore, the moment bound~\eqref{eq:thm-inv-circle} for $\wt\cH$ follows from the moment bound~\eqref{eq:thm:CellSystemPack-extension}. Therefore the conditions in Theorem~\ref{thm:CellSystemPack} hold for $\wt\cH$ and the conclusion follows from Theorem~\ref{thm:CellSystemPack}.
 \end{proof}

\subsection{Proof of Theorem~\ref{thm:invariance-uniformization-extension}}\label{subsec:extension-unif}


Let $\cM$ be a planar map with whole-plane topology and let $\cH$ be a cell configuration whose associated map is $\cM$ as in Theorem~\ref{thm:invariance-uniformization-extension}.  We assume $\cM$ only has faces of degree at least 3; we may assume this upon removing loops and collapsing degree 2 faces as described above the theorem statement and {observing that if the original map satisfies the assumptions of the theorem then the modified map with only face degrees at least 3 also satisfies these conditions.} 
Let $\wt\cT_\cM$ be the triangulation generated by $\cM$ as follows{; see Figure \ref{fig:halfdisk} for an illustration}. We add a vertex $v_f$ to each face $f\in\cF\cM$ and a vertex $v_e$ to each edge $e\in\cE\cM$. {Adding the vertices $v_e$ cause each edge $e$ to be split into two edges, i.e., $e$  is replaced by two edges each connecting one end-point of $e$ and $v_e$.} For each face $f$, we  connect $v_f$ to all the vertices on $\cM$ incident to the face $f$ and all the vertices $v_e$ where $e$ is an edge of $f$. Observe that $\wt\cT_\cM$ is an infinite plane triangulation without loops.  Recall the construction of  the cell configuration $\wt\cH$ defined in terms of $\cH$ as in Section~\ref{subsec:extension-circ}. Let $d_H$ be the number of faces and edges of $\cM$ incident to the vertex $v_H$. Let $(\wt H_{H,k})_{k=0,...,d_H}$ be the partition of the cell $H\in\cH$ as defined in Section~\ref{subsec:extension-circ}, except that we partition each cell into $d_H+1$ parts instead of $\deg(H)+1$ parts. For each edge $e$ (resp.\ face $f$) of $\cM$ such that $v_H$ is on $e$ (resp.\ $f$), we uniformly randomly set $\wt H_{H,e}$ (resp.\  $\wt H_{H,f}$) to be equal to one of $\wt H_{H,k}$ for $k=1,...,d_H$ in a bijective way. For each vertex $v_f$ (resp.\ $v_e$) added to generate $\cT_\cM$, let $\wt H_f = \cup_{H:v_H\sim f} \wt H_{H,f}$ (resp.\ $\wt H_e = \cup_{H:v_H\sim e} \wt H_{H,e}$), and let $\wt H_H = \wt H_{H,0}$. This defines a cell configuration $\wt\cH$  {with cells $\wt H_f$, $\wt H_e$ and $\wt H_H$} whose associated map is $\wt\cT_\cM$. We further define the conductance of edges $\wt e$ of $\wt \cT_\cM$ as follows. 
If $\wt e$ is connecting
  a vertex of $\cM$ and a vertex $v_e$ in $\cV\wt\cT_\cM$, we set $\fc(\wt e)=1$. 
  If $\wt e$ is connecting a vertex of $\cM$ and a vertex $v_f$ in $\cV\wt\cT_\cM$, we set $\fc(\wt e)=\deg(f)+2$. 
  If  $\wt e$ is connecting two vertices $v_e$ and $v_f$, we set $\fc(\wt e)=\deg(f)+1.5$. Note that by our choice of the conductance, the planar map $\cM$ is measurable with respect to $(\wt\cT_\cM,\fc)$. 


Recall the definition of semi-flowers in Section~\ref{subsec:regularity} in the setting where $\cM$ is a triangulation. For a vertex $v$ of $\wt\cT_\cM$, we define the semi-flower $P_v\subseteq M(\cM)$ as follows. First suppose $v$ is also a vertex of $\cM$, {let $f$ be a face incident to $v$, and let $p\geq 3$ denote the degree of $f$.} 
Let $Q_f$ be the unit side length regular $p$-gon associated with $f$, and we place $Q_f$ on the plane. Let $v_1, ..., v_{j_{v,f}}$ be vertices of $Q_f$ such that as we glue the polygon $Q_f$ according to the adjacency pattern of $\cM$ as in the construction of $M(\cM)$, these vertices are identified with $v$. For each such vertex $v_i$, let $e_1,e_2$ be the edges of $Q_f$ incident to $v_i$. We mark the points $x_1,x_2$ on $e_1,e_2$ such that they have distance $1/4$ from $v_i$.  Let $e_1^\perp,e_2^\perp$ be line segments started from  $x_1,x_2$ of length $\frac{1}{6\sqrt3}$ (the choice of this number is unimportant as long as it is less than $\frac{1}{4\sqrt3}$), perpendicular to $e_1,e_2$, and pointing inwards to $Q_f$. Let $e_{12}$ be the line segment connecting the  endpoints of  $e_1^\perp,e_2^\perp$ other than $x_1,x_2$, and let $P_{v,f,i}\subseteq M(\cM)$ be the  part of $Q_f$ enclosed by $e_1,e_2, e_1^\perp,e_2^\perp$ and $e_{12}$ and we view $Q_f$ as part of $M(\cM)$.   We let $P_v = \cup_{f\in\cF\cM:v\sim f}\cup_{i=1}^{j_{v,f}}P_{v,f,i}\subseteq M(\cM)$.  Next suppose $v=v_e$ where $e$ is an edge on the face $f$ of $\cM$. First assume $e$ is incident to $f$ and another face $\wt f$.   Let $Q_f$ be the regular $p$-gon placed on the plane as before. Let $x_3,x_4$ be points on $e$ of distance $1/4$ from the midpoint of $e$. Let  $e_3^\perp,e_4^\perp$ be line segments perpendicular to $e$ of length  $\frac{1}{6\sqrt3}$, started from $x_3,x_4$ and pointing inwards to $Q_f$. Let $e_{34}$ be the line segment connecting the  endpoints of  $e_3^\perp,e_4^\perp$ other than $x_3,x_4$. Let $P_{e,f}\subseteq M(\cM)$ be the corresponding  part of $Q_f$ enclosed by $e,e_3^\perp,e_4^\perp,e_{34}$ when we view $Q_f$ as part of $M(\cM)$.  Define $P_{e,f'}$ analogously, and let $P_{v_e}=P_{e,f}\cup P_{e,f'}$. If $e$ is incident to the face $f$ only, then there is also some other edge $e'$ of $Q_f$ identified with $e$ in our gluing of $Q_f$; in this case we define $P_{e,f}$ and $P_{e',f}$ by drawing line segments on $Q_f$ in the same way, and let $P_{v_e}=P_{e,f}\cup P_{e',f}$. Finally for a vertex $v_f$ where $f$ is a face of $\cM$ of degree at least 3, let $P_{v_f} = Q_f\backslash\big( (\cup_{v\in\cV\cM:v\sim f}P_{v,f})\cup  (\cup_{e\in\cE\cM:e\sim f}P_{v_e})\big)$ be the remaining parts of $Q_f$. 
See Figure~\ref{fig:halfdisk} for an illustration.

\begin{figure}
    \centering
    \begin{tabular}{cc}
          \includegraphics[scale=0.45]{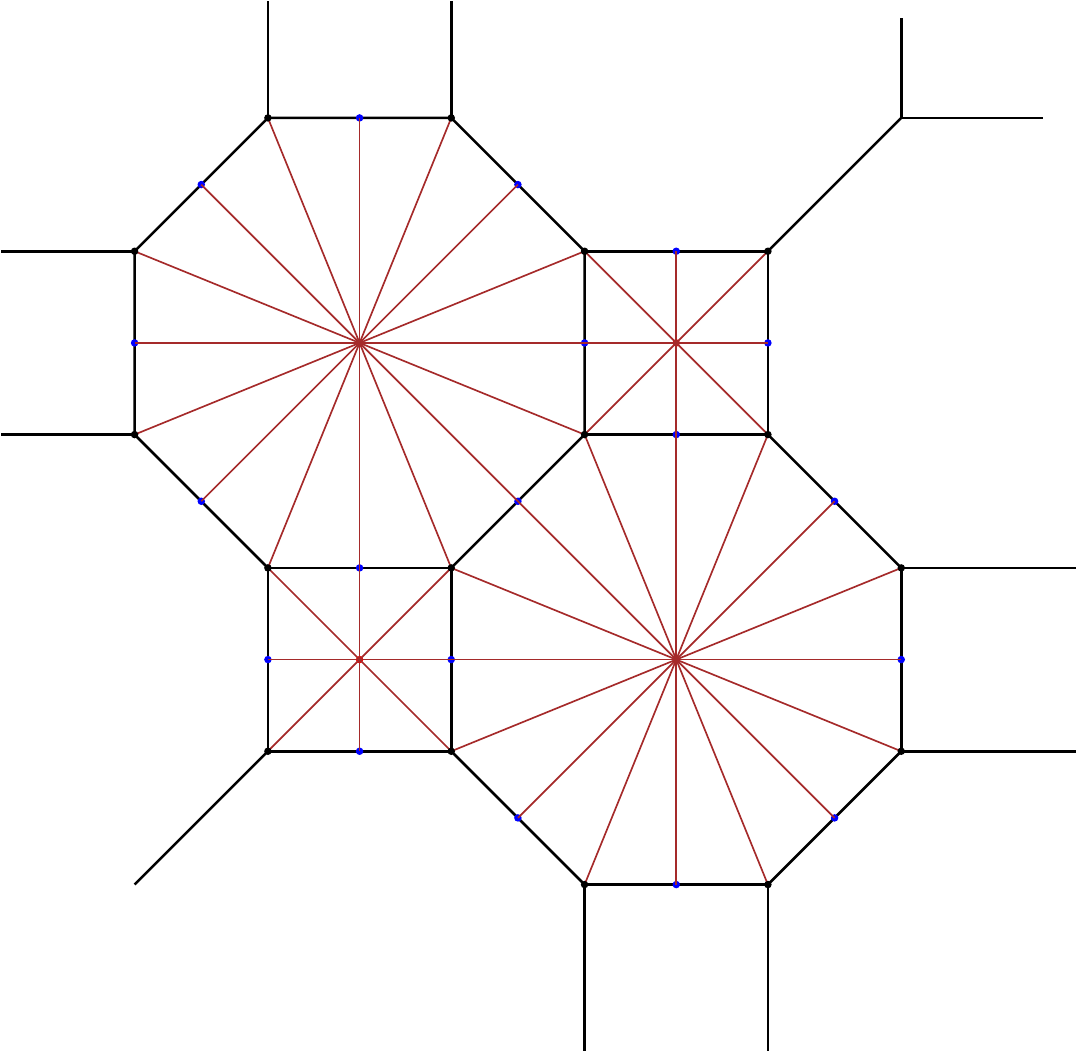}   &       \includegraphics[scale=0.45]{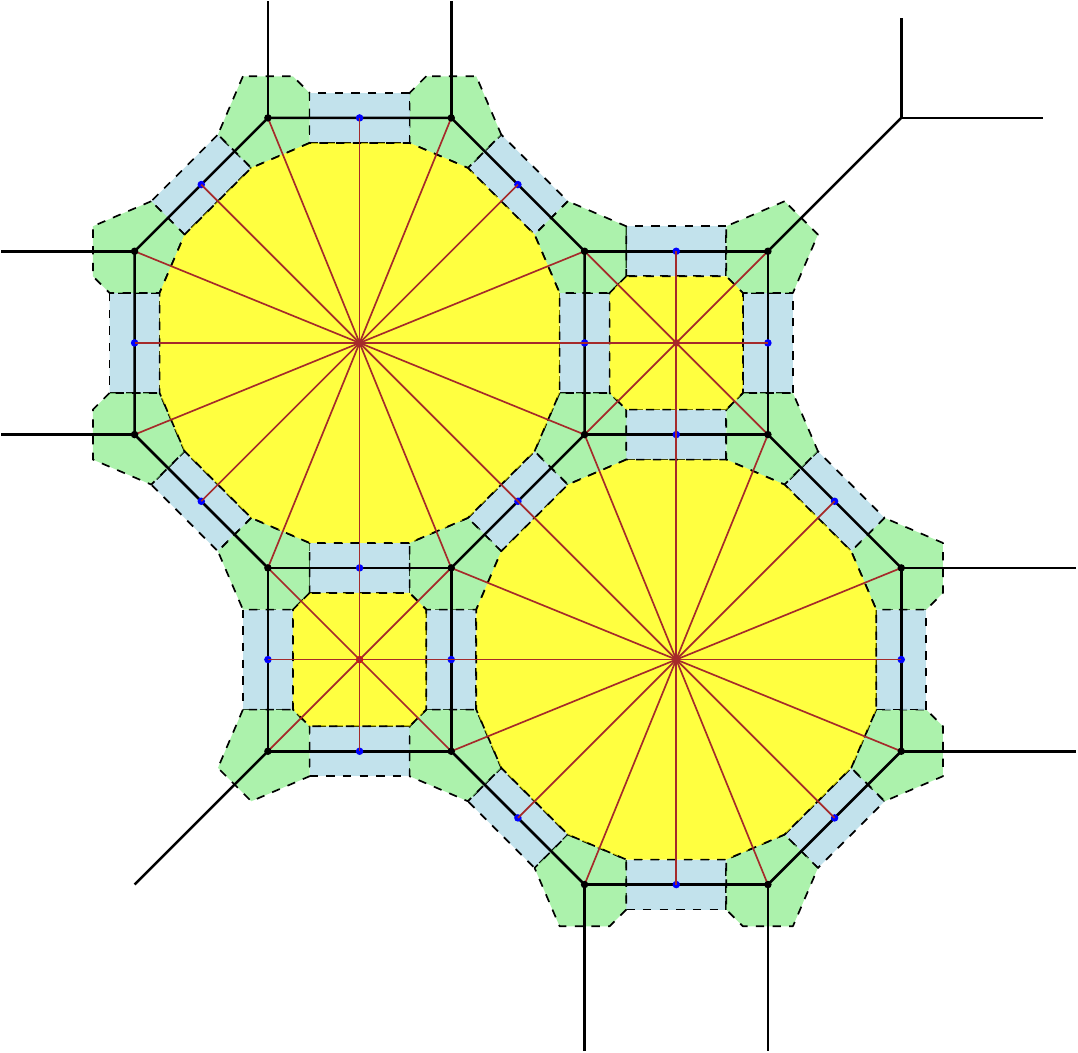}
    \end{tabular}
    \caption{\textbf{Left}: An illustration  of part of $\wt\cT_\cM$, where the black edges are edges of the original map $\cM$, and the brown edges are the edges added to construct $\wt\cT_\cM$. \textbf{Right}:  An illustration of the semi-flowers on the surface $M(\cM)$. 
    The green regions depict the semi-flowers $P_v$ where $v$ is a vertex of $\cM$, 
    the yellow regions depict the semi-flowers $P_{v_f}$ for $f\in\cF\cM$, and 
    the blue regions are for the semi-flowers $P_{v_e}$ with $e\in\cE\cM$.}
    \label{fig:halfdisk}
\end{figure}

\begin{lemma}\label{lem:Koebe-distortion-1}
Let $M_0\subseteq M(\cM)$ be either equal to $M(\cM)$ or equal to the  closure of a union of some faces of $\wt\cT_\cM$ on $M(\cM)$ such that $M_0$ is homeomorphic to the closed unit disk.    The conclusions of Lemma~\ref{lem:Koebe-distortion} hold for {a} conformal map $\varphi:M_0\to\bbC$ or $\varphi:M_0\to\bbD$  and the semi-flowers defined as above.
\end{lemma}

\begin{proof}
Let $v$ be a vertex of $\wt\cT_\cM$ on $M_0$. We only consider the setting where $v$ is in the interior of $M_0$; the boundary case is similar using Schwarz reflection. 
First suppose $v\in\cV\cM$.  Let $r\in (0,\frac{{1}}{2})$.  We define a cone $F_v(r)$ as follows. For each face $f$ incident to $v$ whose associated regular polygon is $Q_f$, let $F_{v,f}(r)$ be the intersection of $Q_f$ and the radius $r$ disks centered at the vertex $v_i$, where $v_1,...,v_j$ are the vertices identified with $v$ in the gluing of $Q_f$. Let $F_v(r) = \cup_{f\in\cF\cM:v\sim f}F_{v,f}(r)$ and observe that $F_v(r)$ has the same geometry of a Euclidean cone of radius $r$ and some angle $\alpha_v\in [\frac{\pi}{3}\deg(v),\pi\deg(v))$,  {where we use} that these disks are disjoint since $r<\frac{1}{2}$.  Furthermore, from our construction $F_v(\frac{1}{16\sqrt{3}})\subseteq P_v\subseteq F_v(\frac{1}{{2\sqrt 3}})\subseteq F_v(\frac{1}{3})\subseteq M(\cM)$. Therefore under the conformal map $\varphi_v:=z\mapsto (3{z})^{\frac{2\pi}{\alpha_v}}$,  there are some universal constants $c_0,c_1>0$ such that $B(0;c_0)\subseteq \varphi_v(P_v)\subseteq B(0;1-c_1\deg(v)^{-1})$, so the claim  {(i.e., \eqref{eq:lem:Koebe-distortion-1})} follows from the Koebe distortion theorem. Now suppose $v=v_f$ where $f\in\cF\cM$ has degree $p\geq 3$. We place $Q_f$ on the plane such that $Q_f$ is centered at 0. Then for $p\geq 4$, $$B\Big(0; \cos(\frac{\pi}{p})(\frac{1}{2}\cot (\frac{\pi}{p})-\frac{1}{2\sqrt{3}})\Big)\subseteq P_{v_f}\subseteq B\Big(0; \sqrt{(\frac{1}{4})^2+(\frac{1}{2}\cot(\frac{\pi}{p})-\frac{1}{6\sqrt 3})^2 }\Big)\subseteq B(0;\frac{1}{2}\cot (\frac{\pi}{p}))\subseteq Q_f,$$
and in particular there are universal constants $c_2,c_3,c_4>0$ such that for some $r>0$, $B(0;c_2r)\subseteq P_{v_f} \subseteq B(0;c_3r(1-c_4p^{-1}))\subseteq B(0;c_3r)\subseteq Q_f$. The claim when $p\geq 4$ follows again from the Koebe distortion theorem, while the $p=3$ case can be treated by a similar computation. Finally, let $v=v_e$ for some edge $e\in\cE\cM$. Then from our construction of $M(\cM)$, on the surface $M(\cM)$, $P_{v_e}$ is a rectangle of side lengths $\frac{1}{2}$ and $\frac{1}{3\sqrt 3}$ contained in the interior of a quadrangle of side length $\frac{1}{\sqrt3}$ and angles $\pi/3,2\pi/3,\pi/3,2\pi/3$, and the claim is again a consequence of the distortion theorem.
\end{proof}

\begin{proof}[Proof of Theorem~\ref{thm:invariance-uniformization-extension}]
    Let $\wt\cH$ be the cell configuration constructed as above. Then the moment bound~\eqref{eq:thm:CellSystemPack-extension} implies that in the cell configuration $\wt\cH$, 
    \begin{equation}\label{eq:pf:thm:invariance-uniformization-extension-0}
        \bbE\big[\frac{\diam(\wt H_0)^2}{\area(\wt H_0)}\deg(\wt H_0)^4  \big]<\infty; \ \  \bbE\big[\frac{\diam(\wt H_0)^2}{\area(\wt H_0)}(\sum_{e:H_0\sim e} \fc(e))  \big]<\infty.
    \end{equation}
    Furthermore, using identical arguments in Section~\ref{subsec:extension-circ}, $\wt \cH$ satisfies all the other conditions in Theorem~\ref{thm:invariance-uniformization}. In particular, by the same arguments as in~\cite[Lemma 2.10]{GMSinvariance}, for some deterministic constant $C>0$, for all large enough $k$, 
    \begin{equation}\label{eq:pf:thm:invariance-uniformization-extension-1}
        \frac{1}{|S_k|^2}\sum_{H\in\wt\cH(S_k)} \diam(H)^2(\sum_{ e:H\sim e}\fc(e))<C.
    \end{equation}
    Now the identical arguments in Section~\ref{subsec:regularity} for $\wt\cH$ carry through once we use Lemma~\ref{lem:Koebe-distortion-1} instead of Lemma~\ref{lem:Koebe-distortion}, and we define the function $\phi_0$ in Section~\ref{subsec:GMS2.3}  for $\wt\cH$ in the identical way. It is elementary to prove that, for $H_1,H_2,H_3\in\wt\cH$ a triangle $\Delta_{H_1,H_2,H_3}$ on $M(\cH)$ where $H_3$ is the cell for some vertex $v_f$, $\mathrm{Energy}(\phi_0;\Delta_{H_1,H_2,H_3})\leq c_0\deg(H_3)(\diam(H_1)^2+\diam(H_2)^2+\diam(H_3)^2)$ for some absolute constant $c_0>0$, which combined with~\eqref{eq:pf:thm:invariance-uniformization-extension-1} further implies the bound~\eqref{eq:energy-phi0-2}, i.e.,  the bound for the Dirichlet energy of $\phi_0$ as in Lemma~\ref{lem:energy-phi0}. The rest of the proof is identical to that of Theorem~\ref{thm:invariance-uniformization}, where~\eqref{eq:thm:invariance-uniform-10} is a consequence of~\eqref{eq:thm:invariance-uniform-1} for the cell configuration $\wt \cH$ (with the additional vertices and edges) since the diameter of a face $\varphi(\mathfrak{f}_f)$ is controlled by  $8\max_e\diam(\varphi_0(\mathfrak{e}_{e}))$ where the maximum is over all the edges of  $\mathcal{T}_{\cM}$ incident to a vertex $v$ on the face $f$. For the final claim below~\eqref{eq:thm:invariance-uniform-10}, almost surely for large enough $n$, if a face $f$ is incident to a cell in $\cH(\ol{B(0;2^{n+1})\backslash B(0;2^n)})$, then by the diameter control~\eqref{eq:thm:invariance-uniform-10} and Lemma~\ref{lem:no-macroscopic-cell}, it holds that $\mathbf{A}^{-1}\varphi_0(\mathfrak{f}_f)$ is disjoint from $B(0;2^{n-1})$. By an intersection over all large $n$, almost surely for large enough $n$, if $\mathbf{A}^{-1}\varphi_0(\mathfrak{f}_f)$ intersects $B(0;2^{n-1})$, then the face $f$ cannot be incident to any cell intersecting $\bbC\backslash B(0;2^n)$. The claim then follows from~\eqref{eq:thm:invariance-uniform-10}.
\end{proof}

\subsection{An extension with weaker notion of line connectivity}\label{subsec:large}

The following version of Theorems~\ref{thm:CellSystemPack}-\ref{thm:invariance-uniformization-extension}, which uses a weaker version of the line connectivity in Definition~\ref{def:connectedness}, will be used in our future work.
\begin{theorem}\label{thm:large}
    Suppose $\cH$ is a cell configuration that satisfies all the conditions in one of the settings of Theorems~\ref{thm:CellSystemPack}-\ref{thm:invariance-uniformization-extension} except for the line connectivity condition in Definition~\ref{def:connectedness}. Suppose  there exists a cell configuration $\cH^{\rm large}$ satisfying the condition in Definition~\ref{def:ergodic-modulo-scaling} and the moment bound~\eqref{eq:thm-inv-circle} in the setting of Theorem~\ref{thm:CellSystemPack} or Theorem~\ref{thm:invariance-uniformization} or the moment bound~\eqref{eq:thm:CellSystemPack-extension} in the setting of Theorem~\ref{thm:CellSystemPack-extension} or Theorem~\ref{thm:invariance-uniformization-extension} along with the following properties:
    \begin{enumerate}[(i)]
        \item The cell configuration  $\cH^{\rm large}$ is isomorphic to $\cH$, in the sense that $\cH$ and $\cH^{\rm large}$ have isomorphic associated maps; below we let $\phi:\cH\to\cH^{\rm large}$ denote  {the} isomorphism.  
        \item The convex hull of $H$ is contained in the convex hull of $\phi(H)$  {for any $H\in\cH$}.
        \item Almost surely,  {for each given $\e>0$, there exists $R>0$ such that for $r>R$ and each horizontal or vertical line segment $\ell\subseteq B(0;r)$ with length at least $\e r$}, the following holds. For any two different cells $H,H'\in\cH$ intersecting $\ell$, $\phi(H),\phi(H')$ also intersect $\ell$, and we can find a path $H\sim H_1\sim\cdots\sim H_k\sim H'$, such that $\phi(H_j)$ intersects $\ell$ for each $j$.
    \end{enumerate}
    Then the conclusions in Theorem~\ref{thm:CellSystemPack}-\ref{thm:invariance-uniformization-extension}  continue to hold.
\end{theorem}

Theorem~\ref{thm:large} can be viewed as a generalization of Theorem~\ref{thm:CellSystemPack}-\ref{thm:invariance-uniformization-extension}  since one can set $\cH^{\rm large}=\cH$ in the latter setting. 

\begin{proof}
We only explain the modifications needed in the settings of Theorems~\ref{thm:CellSystemPack} and~\ref{thm:invariance-uniformization} since the modifications needed in the settings of Theorems~\ref{thm:CellSystemPack-extension} and~\ref{thm:invariance-uniformization-extension} are identical using the proof in Sections~\ref{subsec:extension-circ} and~\ref{subsec:extension-unif}. 
Most of the proof is identical to that of Theorem~\ref{thm:CellSystemPack} and Theorem~\ref{thm:invariance-uniformization}, and we will only explain the differences. We first modify the  definition of $\cM(S)$   in Section~\ref{subsec:ergodic-cell-system} as below.  For each horizontal or vertical line segment $\ell$, let $\cH^{\rm mid}(\ell)\subseteq\cH$ be the union of the set $\{H,H_1,...,H_k,H'\}$ over all possible choices of $H\sim H_1\sim\cdots\sim H_k\sim H'$ where $H,H'\in\cH(\ell)$ and $\phi(H_j)\in\cH^{\rm large}(\ell)$ for every $j=1,...,k$. We let $\cH^{\rm mid}(\mathring{S_a})$ be the union of $\cH^{\rm mid}(\ell)\subseteq\cH$ over all possible $\ell\subseteq \mathring{S_a}$ of length $|\mathring{S_a}|$. Conditions (ii) and (iii) guarantee that for every $\e>0$, for large enough $r>0$, if $\mathring{S_a}\subseteq B(0;r)$ and $|\mathring{S_a}|\geq\e r$, then $\cH^{\rm mid}(\mathring{S_a})$ is connected. Then we define $\cM_0(S;a)$, $\cM(S;a)$, $\cM(S)$ and $\varphi_k$ in the identical way 
using $\cH^{\rm mid}(\mathring{S_a})$ instead of $\cH(\mathring{S_a})$. We also define $\cH^{\rm mid}(\ol{\bbC\backslash\mathring{S_a}})$ to be the  union of $\cH^{\rm mid}(\ell)\subseteq\cH$ over all possible  {line segments} $\ell\subseteq\ol{\bbC\backslash} \mathring{S_a}$, which is also connected. The arguments in Section~\ref{subsec:ergodic-cell-system} other than the proof of Lemmas~\ref{lem:M(S)} and~\ref{lem:phi_0-bbC} do not require the line connectivity property in Definition~\ref{def:connectedness},  and Lemma~\ref{lem:M(S)} can be proven via identical arguments. For the proof of Lemma~\ref{lem:phi_0-bbC}, using conditions (ii) and (iii) along with Lemma~\ref{lem:no-macroscopic-cell}, one can find cells $H_1\sim\cdots H_n\sim H_1 $ with identical properties and the conclusion still holds. 
To adapt the proof of~{Theorem \ref{thm:GMSinvariance}}
{(which was proven in \cite{GMSinvariance})} to the new setting, we note that the line connectivity property is used only in the proof of~\cite[Lemma 2.19]{GMSinvariance}, in the paragraph right after~\cite[Eq.\ (2.46)]{GMSinvariance}. To substitute for that,   Conditions (ii) and (iii) guarantee that a.s.\ for large enough $k$ and horizontal or vertical line segments $\ell\subseteq S_k$ of length equal to the side length of $S_k$, for any function $g:\cH^{\rm large}\to\bbR$ and $H_1,H_2\in \cH(\ell)$, 
$$|g(\phi(H_1))-g(\phi(H_2))|\leq \sum_{H,H'\in \cH^{\rm large}(\ell):H\sim H'}|g(H)-g(H')|.$$
The right side can be bounded exactly as~\cite{GMSinvariance} since the arguments there only rely on the inputs from Lemma~\ref{lem:dyadic-resample}-\ref{lem:average-diamdeg4}, which do not require the line connectivity property. Once this is done, the arguments in Section~\ref{subsec:pre-invprinciple} and Section~\ref{sec:circle-packing} work, except that in the construction of the map $\cM_{k,\ell}$ in Lemma~\ref{lem:almost-planarity}, we connect the cells $H_j$ and $H_{j+1}$ there using cells $H$ such that $\phi(H)$ intersects $L_{j,j+1}$. 


In the setting of the Riemann uniformization, Definition~\ref{def:connectedness} is used a few times. Let us explain the modifications.  To prove Lemma~\ref{lem:area/length-1} in this setting, we define 
$$L^{\rm large}(\rho) = \sum_{H^{\rm large}\in \cH^{\rm large}(\partial \mathring{S}^{\rho})}\diam(\varphi_k(P_{H^{\rm large}})).$$
Then ~\eqref{eq:pf-area/length-1-5} holds with $L^{\rm large}(\rho)$ instead of $L(\rho)$ using identical arguments applied to $\cH^{\rm large}$, and~\eqref{eq:lem:area/length-1} follows from identical arguments using   conditions (ii) and (iii). To be more precise, the proof of~\eqref{eq:pf-area/length-1-1}-\eqref{eq:pf-area/length-1-6} does not reply on the line connectivity property and can be carried over to the setting of $\cH^{\rm large}$ in place of $\cH$. Then we may take the cells in $\cH^{\mathrm{cell}}(\partial \mathring S^\rho)$ and use conditions (ii) and (iii) along with the planar embedding approximation in Section~\ref{subsec:pre-invprinciple} to construct connected cells in $\cH^{\mathrm{large}}(\partial \mathring S^\rho)$ where they enclose the cells $\cH^{\mathrm{large}}(S)$ in the associated planar map as in the proof of Lemma~\ref{lem:area/length-1}. This gives~\eqref{eq:lem:area/length-1} for $\cH^{\mathrm{large}}$. Lemma~\ref{lem:area/length-6} for $\cH^{\rm large}$ follows similarly. The reverse estimate in Lemma~\ref{lem:area/length-2} uses the planar embedding approximation in Section~\ref{subsec:pre-invprinciple} and does not use the line connectivity.  The arguments for Lemma~\ref{lem:polynomial-B} are identical, except that at the final step we use that by Lemma~\ref{lem:M(S)} the set on the right hand side of~\eqref{eq:pf:lem:polynomial-B-2} is contained within $\varphi(M(w+[-\delta 2^k,\delta2^k]^2))$, which is a connected subset of the set in~\eqref{eq:lem:polynomial-B}. The rest of the proofs in Sections~\ref{subsec:regularity}-\ref{subsec:pf-thm:invariance-uniformization} do not require the line connectivity property and can be copied to the setting of Theorem~\ref{thm:large}.
\end{proof}


\bibliographystyle{alpha}
\bibliography{circle}

\end{document}